\documentclass[12pt,reqno]{amsart}
\usepackage{amssymb,mathrsfs,stmaryrd,longtable}

\newcommand{\B}{\mathcal{B}}
\newcommand{\D}{\mathcal{D}}
\newcommand{\F}{\mathcal{F}}
\newcommand{\V}{\mathcal{V}}
\newcommand{\X}{\mathcal{X}}
\newcommand{\cL}{\mathcal{L}}
\newcommand{\cS}{\mathcal{S}}
\newcommand{\cP}{\mathcal{P}}

\newcommand{\cN}{\mathcal{N}}
\newcommand{\cZ}{\mathcal{Z}}

\newcommand{\cU}{\mathcal{U}}

\newcommand{\C}{\mathbb{C}}
\newcommand{\N}{\mathbb{N}}
\newcommand{\R}{\mathbb{R}}
\newcommand{\Z}{\mathbb{Z}}

\newcommand{\al}{\alpha}
\newcommand{\be}{\beta}
\newcommand{\de}{\delta}
\newcommand{\ga}{\gamma}
\newcommand{\e}{\varepsilon}
\newcommand{\fy}{\varphi}
\newcommand{\io}{\iota}
\newcommand{\ka}{\kappa}
\newcommand{\la}{\lambda}

\newcommand{\te}{\theta}
\newcommand{\si}{\sigma}
\newcommand{\ta}{\tau}
\newcommand{\x}{\xi}
\newcommand{\y}{\eta}
\newcommand{\z}{\zeta}

\newcommand{\De}{\Delta}
\newcommand{\Om}{\Omega}
\newcommand{\Ga}{\Gamma}
\newcommand{\La}{\Lambda}

\newcommand{\p}{\partial}
\newcommand{\na}{\nabla}
\newcommand{\Cu}{\bigcup}
\newcommand{\re}{\mathop{\mathrm{Re}}}

\newcommand{\weak}{\operatorname{w-}}
\newcommand{\supp}{\operatorname{supp}}
\newcommand{\sign}{\operatorname{sign}}

\newcommand{\lec}{\lesssim}
\newcommand{\gec}{\gtrsim}
\newcommand{\IN}[1]{\text{ in }#1}
\newcommand{\etc}{,\ldots,}
\newcommand{\I}{\infty}

\newcommand{\ti}{\tilde}
\newcommand{\ba}{\overline}
\newcommand{\ck}{\check}
\newcommand{\U}{\underline}
\newcommand{\LR}[1]{{\langle #1 \rangle}}
\newcommand{\Lim}{\lim\limits}
\newcommand{\Liminf}{\varliminf\limits}
\newcommand{\Limsup}{\varlimsup\limits}

\newcommand{\empt}{\varnothing}

\newcommand{\rad}{_{\operatorname{rad}}}
\newcommand{\HL}{\operatorname{HL}}
\newcommand{\LH}{\operatorname{LH}}
\newcommand{\HH}{\operatorname{HH}}

\newcommand{\EQ}[1]{\begin{equation}\begin{split} #1 \end{split}\end{equation}}

\newcommand{\BR}[1]{\left[#1\right]}
\newcommand{\Br}[1]{\left\{#1\right\}}

\newcommand{\define}{\overset{\operatorname{def}}{\iff}}

\newcommand{\tand}{\ \text{ and }\ }
\newcommand{\tor}{\ \text{ or }\ }
\newcommand{\squm}{\sum^{(2)}}

\newcommand{\Del}[1]{}
\newcommand{\CAS}[1]{\begin{cases} #1 \end{cases}}

\newcommand{\pt}{&}
\newcommand{\pr}{\\ &}
\newcommand{\pq}{\quad}
\newcommand{\pQ}{\qquad}
\newcommand{\pn}{}
\newcommand{\prq}{\\ &\quad}
\newcommand{\prQ}{\\ &\qquad}

\newcommand{\fn}{_\star}
\newcommand{\sF}{^{\scriptscriptstyle \textsc{(f)}}}
\newcommand{\sN}{^{\scriptscriptstyle \textsc{(n)}}}
\newcommand{\sD}{^{\scriptscriptstyle \textsc{(d)}}}
\newcommand{\Uf}{U_{\operatorname{f}}}
\newcommand{\rd}{_{\operatorname{rd}}}

\newcommand{\wt}{_{\operatorname{wt}}}
\newcommand{\pert}{_{\operatorname{pe}}}
\newcommand{\ds}{_{\operatorname{ds}}}
\newcommand{\tle}{\textless}
\newcommand{\tge}{\textgreater}

\numberwithin{equation}{section}
\newtheorem{thm}{Theorem}[section]
\newtheorem{cor}[thm]{Corollary}
\newtheorem{lem}[thm]{Lemma}
\newtheorem{prop}[thm]{Proposition}
\theoremstyle{remark}
\newtheorem{rem}{Remark}[section]
\newtheorem{defn}[thm]{Definition}

\begin{document}

\title[4D Zakharov below the ground state]{The Zakharov system in 4D radial energy space below the ground state}

\author[Z.~Guo, K.~Nakanishi]{Zihua Guo, Kenji Nakanishi} 

\address{School of Mathematical Sciences, Monash University, VIC 3800, Australia}
\email{zihua.guo@monash.edu}

\address{Research Institute for Mathematical Sciences, Kyoto University, Kyoto 606-8502, Japan} \email{kenji@kurims.kyoto-u.ac.jp}

\begin{abstract}
We prove dynamical dichotomy into scattering and blow-up (in a weak sense) for all radial solutions of the Zakharov system in the energy space of four spatial dimensions that have less energy than the ground state, which is written using the Aubin-Talenti function. The dichotomy is characterized by the critical mass of the wave component of the ground state. 
The result is similar to that by Kenig and Merle \cite{KM} for the energy-critical nonlinear Schr\"odinger equation (NLS). Unlike NLS, however, the most difficult interaction in the proof stems from the free wave component. 
In order to control it, the main novel ingredient we develop in this paper is a uniform global Strichartz estimate for the linear Schr\"odinger equation with a potential of subcritical mass solving a wave equation. 
This estimate, as well as the proof, may be of independent interest. 
For the scattering proof, we follow the idea by Dodson and Murphy \cite{DM}. 
\end{abstract}

\maketitle

\tableofcontents

\section{Introduction}
\subsection{The Zakharov system in four dimensions}
We continue from \cite{Zak4D1} the study of the Zakharov system in four space dimensions: 
\EQ{\label{Zak0}
 \CAS{ i\dot u - \De u = nu, \pq u(t,x):\R^{1+4}\to\C,
 \\ \ddot n/\al^2 - \De n = -\De|u|^2, \pq n(t,x):\R^{1+4}\to\R,}}
with a special focus on global behavior of solutions that are not small.

The Zakharov system is a mathematical model for Langmuir waves in a plasma, which couples the envelope $u$ of the electric field and the ion density $n$ with the sound speed $\al>0$. Since the main result is for the system with a fixed speed $\al>0$, we may and will take $\al=1$ without loss of generality, by rescaling. In a few places, however, we will restore $\al$ and consider the limits $\al\to\I$ and $\al\to 0$ for comparison with the limit equations. We refer to \cite{Zak4D1} and references therein for more detailed introduction on the Zakharov system and the preceding studies. 
 
As in \cite{Zak4D1}, the Zakharov system is equivalently transformed to a system of the first order equations by the change of variable
$n \mapsto N:=n-iD^{-1}\dot n$ with $D:=\sqrt{-\De}$: 
\EQ{ \label{Zak}
 \CAS{ (i\p_t-\De-\re N)u=0, \pq u(t,x):\R^{1+4}\to\C,
 \\ (i\p_t+D)N=D|u|^2, \pq N(t,x):\R^{1+4}\to\C.} }
In view of the conservation of the mass and the energy (Hamiltonian):
\EQ{ \label{def M EZ}
 M(u):=\int_{\R^4}|u|^2dx, \pq 
 E_Z(u,N):=\int_{\R^4}|\na u|^2+\frac{|N|^2}{2}-\re N|u|^2 dx,} 
it is natural to study solutions in the energy space $(u(t),N(t))\in H^1(\R^4)\times L^2(\R^4)$ at each $t$, which is the Sobolev space in $x\in\R^4$ normed with the quadratic part of $M(u)+E_Z(u,N)$. 
In fact, once the local wellposedness of the Cauchy problem is established in the energy space, the standard argument yields conservation of $M(u)$ and $E_Z(u,N)$, which plays crucial roles in the global analysis of solutions. 

Although the four dimensional setting does not appear physical, it is interesting from the PDE viewpoint, because the system may be regarded as energy-critical in several aspects. 
Henceforth, the Lebesgue norm on the whole space is denoted by 
\EQ{ \label{def norm-p}
 \|\fy\|_p:=\|\fy\|_{L^p(\R^d)}.} 
The energy criticality is readily seen in the Hamiltonian $E_Z(u,N)$. The nonlinear part is barely controlled by the critical Sobolev embedding as 
\EQ{
 \left|\int_{\R^4}N|u|^2dx\right| \le \|N\|_2\|u\|_4^2 \le C_S^2\|N\|_2\|\na u\|_2^2,}
where $C_S>0$ denotes the best constant in the Sobolev inequality 
\EQ{ \label{Sob}
 \|u\|_{L^4(\R^4)} \le C_S\|\na u\|_{L^2(\R^4)}.}
Another simple way to observe the criticality is by the subsonic limit. Let $\al\to \I$ in \eqref{Zak0}, then the Zakharov system (formally) converges to the energy-critical nonlinear Schr\"odinger equation (NLS):   
\EQ{\label{eq:NLS}
 i\dot u - \De u = |u|^2u, \pq N=|u|^2}
which is scaling invariant for 
\EQ{ \label{scal u}
 u(t,x)\mapsto \la^2u(\la^2t,\la x) \pq (\la>0)} 
in the homogeneous energy space $u\in\dot H^1(\R^4)\subset L^4(\R^4)$. 

The main difficulty, however, stems from the free wave component of $N$ in $L^2(\R^4)$, which can not be seen in \eqref{eq:NLS}, since it is lost as the initial layer in the limit (cf.~\cite{Z2NLS} for a rigorous justification of the limit in the energy space on $\R^3$). 
If we ignore for the moment the evolution of $N$ and just consider the Schr\"odinger equation with a potential:
\EQ{ \label{SP}
 i\dot u - \De u = (\re N)u,}
in general space dimensions $\R^d$, the invariant norm for the scaling \eqref{scal u} is $\|N\|_{L^{d/2}(\R^d)}$. The four dimensional case is special because the free wave equation $(i\p_t-D)N=0$ is wellposed in $L^{d/2}(\R^d)$ if and only if $d=4$. 
If we use dispersive estimates on $N$ such as $L^p(\R)$ in $t$ with $p<\I$ in estimating the interaction $(\re N)u$, then it will break the scale invariance, since the wave equation has weaker dispersion than the Schr\"odinger. 
Thus we are essentially forced to use $\|N(t)\|_{L^2(\R^4)}$. 
Since it does not decay in $t$, it raises difficulty for large solutions to be treated in perturbative ways.

To see the above more concretely, recall the endpoint Strichartz estimate \cite{KT}: 
\EQ{ \label{end Str}
 \int_\R\|u(t)\|_{2^*}^2dt \le C\|u(0)\|_2^2 + C\int_\R\|(i\dot u-\De u)(t)\|_{2_*}^2dt,}
where $d\ge 3$, $2^*:=2d/(d-2)$ and $2_*:=2d/(d+2)$. 
Since the multiplication with a $L^{d/2}$ function sends the norm on the left to that on the right by H\"older: 
\EQ{ \label{end prod}
 \|N(t)u(t)\|_{2_*} \le \|N(t)\|_{d/2}\|u(t)\|_{2^*},}
one can thereby control the global Strichartz norm of $u$ in the linear Schr\"odinger equation \eqref{SP} if $\sup_t\|N(t)\|_{d/2}$ is small enough, as was observed in \cite{RS}. 
Combining such estimates with the normal form argument to avoid the derivative loss, 
the local wellposedness of \eqref{Zak} was proven in \cite{Zak4D1} in the $L^2$-Sobolev space $(u,N)\in H^s(\R^4)\times H^l(\R^4)$ in a certain region of regularity $(s,l)$, improving the results in \cite{GTV}, including 
\EQ{\label{eq:slrange}
 1/2\le s<1, \pq l=0}
as the boundary with the lowest regularity of the wave component $N$, as well as persistence of regularity, and scattering for small initial data. See \cite[Theorem 1.2, Proposition 5.1]{Zak4D1} for the detail. 

If we would replace the endpoint norm in \eqref{end prod} on $\R^4$ by some non-endpoint one, then $N$ would have to disperse in the same way as the Schr\"odinger solutions in terms of some Strichartz norm, because of the scaling of \eqref{SP}. However, such norms can not be uniformly bounded for free waves in $L^2(\R^4)$, even in the radial case. The idea in \cite{Zak4D1} to treat large solutions in \eqref{eq:slrange}, locally in time, was to approximate $N$ by nicer solutions, which must depend on the profile of $N$ at each time. 

The case of energy space $(s,l)=(1,0)$ is more subtle and can not be treated in the same way as for \eqref{eq:slrange} due to logarithmic failure in some crucial estimates, which is related to that of the critical Sobolev embedding: $W^{1,4}(\R^4) \not\subset L^\I(\R^4)$, 
naturally arising in the endpoint Strichartz estimate \eqref{end Str} with a space derivative in $d=4$. 
However, in \cite{Zak4D1} it was also observed that by a  weak compactness argument and conservation of the energy and the mass, the energy space can be handled in an indirect way using the results in the other exponents, and thus the global wellposedness with scattering in the energy space $H^1(\R^4)\times L^2(\R^4)$ was proven for small initial data. 

The main interest of this paper is in the existence and behavior of large solutions to the Zakharov system \eqref{Zak} in the energy space. 
Following the same approach as in \cite{Zak4D1}, the difficulty for large global solutions is mainly in $N\in L^2(\R^4)$ as mentioned above, rather than $u\in H^1(\R^4)$. 
Then from the dynamical viewpoint, it is natural to ask exactly how much smallness of $\|N(0)\|_2$ is needed to ensure the global existence and the scattering. 
This type of question was solved in the case of the energy-critical NLS by Kenig and Merle \cite{KM}, who developed the concentration-compactness and rigidity method,  and proved dynamical dichotomy into scattering and blow-up for solutions to \eqref{eq:NLS} with energy below the ground state in the radial case.  The ground state is the celebrated Aubin-Talenti function, which is the unique (modulo scaling and translation) maximizer of the Sobolev inequality. On $\R^4$, it has the form
\EQ{ \label{def W}
 W:=[1+|x|^2/8]^{-1}.}
Its maximality for \eqref{Sob} can be written precisely as 
\EQ{ \label{max Sob}
 \|W\|_4/\|\na W\|_2=\sup_{0\not=\fy\in H^1(\R^4)}\|\fy\|_4/\|\na\fy\|_2=:C_S,}
with the Euler-Lagrange equation being the static NLS:  
\EQ{ \label{stNLS}
 -\De W = W^3.}

It turns out that the same type of dynamic dichotomy holds for the Zakharov system \eqref{Zak} under the ground state static solution:
\EQ{
 (u,N)=(W,W^2).}
Since it is obviously not a scattering solution, the condition 
\EQ{ \label{mass threshold}
  \|N(t)\|_2 < \|W^2\|_2 }
is possibly the best upper bound on the wave mass for the scattering. Note that it is also the optimal upper bound for positivity of the Schr\"odinger operator in \eqref{SP}:
\EQ{ \label{positivity}
 \|N\|_2 < \|W^2\|_2 \implies -\De + \re N >0,}
since $W$ is the Sobolev optimizer. Indeed, for any $\fy\in H^1(\R^4)$ we have 
\EQ{
 \LR{(-\De-\re N)\fy|\fy}\ge \|\na\fy\|_2^2-\|N\|_2\|\fy\|_4^2 \ge (1-\|N\|_2\|W\|_4^{-2})\|\na\fy\|_2^2,}
hence $\|N\|_2<\|W\|_4^2=\|W^2\|_2$ implies $-\De+V>0$.

Since $\|N(t)\|_2$ is not conserved in the flow of the Zakharov system, \eqref{mass threshold} is not in general preserved for later time. It turns out, however, that if the conserved energy $E_Z(u,N)$ is less than that of the ground state, then \eqref{mass threshold} is topologically preserved. 
Thus we obtain the following, which is the main result of this paper: the scattering and a weak blow-up in the radial energy space below the ground state. 
Henceforth, the subspace of radially symmetric functions in a function space $X$ on $\R^d$ is denoted by 
$X\rad(\R^d)$. 
\begin{thm} \label{main}
For any $(\fy,\psi)\in H^1\rad(\R^4)\times L^2\rad(\R^4)$ satisfying 
\EQ{
 E_Z(\fy,\psi)<E_Z(W,W^2),}
there exists a unique\footnote{The uniqueness is proved only under a space-time integrability condition as in Definition \ref{def:sol}.} local solution $(u,N)\in C(I;H^1(\R^4)\times L^2(\R^4))$ of \eqref{Zak} with the initial condition $(u(0),N(0))=(\fy,\psi)$ on the maximal existence interval $I\ni 0$. 
Moreover, it satisfies either (1) or (2) exclusively. 
\begin{enumerate}
\item $\|N(t)\|_2<\|W^2\|_2$ for all $t\in\R=I$. $(u,N)$ is uniformly bounded in $\dot H^1(\R^4)\times L^2(\R^4)$, while $\|u(t)\|_2$ is conserved. There exist unique $(\fy_\pm,\psi_\pm)\in H^1(\R^4)\times L^2(\R^4)$ such that 
\EQ{
 \|u(t)-e^{-it\De}\fy_\pm\|_{H^1} + \|N(t)-e^{itD}\psi_\pm\|_2 \to 0 \pq (t\to\pm\I).}
\item \label{it:bup} $\|N(t)\|_2>\|W^2\|_2$ for all $t\in I$. If $\sup I=\I$ then $\Limsup_{t\to\I}\|\na u(t)\|_2+\|N(t)\|_2=\I$. If $\inf I=-\I$ then $\Limsup_{t\to-\I}\|\na u(t)\|_2+\|N(t)\|_2=\I$. 
\end{enumerate}
\end{thm}
\begin{rem}
In the case of \eqref{it:bup}, there are three possible scenarios for singularity in $t>0$. Let $T:=\sup I$ be the maximal time. (a)Norm blow-up: $T<\I$ and $\Limsup_{t\to T}\|\na u(t)\|_2+\|N(t)\|_2=\I$. (b)Concentration blow-up: $T<\I$ and $\Limsup_{t\to T}\|\na u(t)\|_2+\|N(t)\|_2<\I$. (c)Norm grow-up: $T=\I$ and $\Limsup_{t\to\I}\|\na u(t)\|_2+\|N(t)\|_2=\I$. The above theorem asserts only that one of the three happens in the case of \eqref{it:bup}. 
\end{rem}
\begin{rem}
The local wellposedness holds in the energy space for general data with no size or symmetry restriction, and follows essentially from the same argument as in \cite{Zak4D1}, see Section \ref{sect:LWP} below.
\end{rem}

Compared with Kenig-Merle's result \cite{KM} for NLS, the above result is almost the same if $N$ is replaced with $|u|^2$, except the scattering of $N$ and that the blow-up result in \cite{KM} is the true one in finite time. 
As mentioned before, however, the main difficulty for the proof in the Zakharov system is in the free wave component of $N$ rather than $|u|^2$. A new estimate is derived to control that term: uniform global Strichartz estimates for the linear Schr\"odinger equation \eqref{SP} when $N$ solves the wave equation, which is explained below. 

\subsection{Strichartz estimate with a wave potential}
The main novel ingredient in the proof is a Strichartz estimate for the Schr\"odinger equation in $H^s(\R^4)$ ($s<1$) with a large $L^2(\R^4)$ potential solving a wave equation, which may be of independent interest. 
The estimate is an extension of the following three versions (Lemmas \ref{lem:freeStz}, \ref{lem:stSt}, \ref{lem:smallSt}). 
The full range of the Strichartz admissible exponents, namely the double endpoint, is needed to treat the wave potential $N\in L^\I_tL^2$, cf.~\eqref{end Str}--\eqref{end prod}. 

Before recalling the known Strichartz estimates, we introduce some notation. 
The homogeneous Besov and Sobolev spaces on $\R^d$ is denoted respectively by $\dot B^s_{p,q}(\R^d)$ and $\dot H^s(\R^d)=\dot B^s_{2,2}(\R^d)$, while the inhomogeneous Sobolev space is denoted by $H^s(\R^d)$. 
For brevity, the Sobolev exponent is denoted by 
\EQ{ \label{def p(s)}
 d\in\N,\ p\in[1,\I],\ s\in\R \implies \frac{1}{p(s)}:=\frac{1}{p}+\frac{s}{d},}
so that the Sobolev embedding can be written as 
\EQ{
 \e\ge 0 \implies \dot B^{s}_{p,q}(\R^d) \subset \dot B^{s-\e}_{p(-\e),q}(\R^d).}
For example, we have on $\R^4$ ($d=4$), 
\EQ{
 2(-2)=\I,\pq 2(-1)=4,\pq 2(1)=4/3, \pq 2(2)=1.}
The following notation is introduced for the endpoint Strichartz norms: for $\de\in\R$, 
\EQ{ \label{def X}
 \pt \X^\de:=L^2_t \dot B^{\de}_{2(\de-1),2}, \pq \X^\de_*:= L^2_t \dot B^{-\de}_{2(-\de+1),2},
 \pr \ck \X^\de:= L^\I_t L^2 \cap \X^\de, \pq \ck\X^\de_* := L^1_t L^2 + \X^\de_*,
 \pr \ti \X^\de:=\cL^\I_t L^2 \cap \X^\de, \pq \ti \X^\de_*:=\cL^1_t L^2 + \X^\de_*,}
where $\cL^p_t Z$ denotes the space-time-frequency mixed norm defined by 
\EQ{ \label{def cL}
 \|u\|_{\cL^p_t Z}^2 = \sum_{j\in 2^\Z} \|u_j\|_{L^p_t Z}^2,}
where $u=\sum_{j\in 2^\Z}u_j$ is the Littlewood-Paley decomposition on $\R^d$ such that 
\EQ{ \label{def LP}
 \supp\F u_j \subset \{\x\in\R^d \mid j/2 <|\x|<2 j\}.}
Henceforth the Fourier transform is denoted by 
$\F\fy(\x)=\hat\fy(\x):=\int_{\R^d}\fy(x)e^{-ix\x}dx$. 
We have the following embeddings for $d\ge 3$ by the Sobolev inequality:
\EQ{
 \de_1\ge\de_2\ge\de_3 \implies \CAS{ \ti \X^{\de_1}\subset \ti \X^{\de_2} \subset \ck \X^{\de_2} \subset \X^{\de_2} \subset \X^{\de_3}, 
 \\ \ti \X_*^{\de_1}\supset \ti \X_*^{\de_2}\supset \ck \X_*^{\de_2} \supset \X_*^{\de_2}\supset \X_*^{\de_3},}}
as long as the $L^p$ exponents are within $[1,\I]$. 
For any Banach space $X$ of functions on $\R$, its restriction to any interval $I\subset\R$ is denoted by
\EQ{ \label{def X(I)}
 \|u\|_{X(I)} := \|u_I\|_{X}, \pq u_I(t)=\CAS{u(t) &(t\in I),\\ 0 &(t\not\in I).} }

Using the above notation, we first recall the improved Strichartz estimate for radial solutions of the free Schr\"odinger equation with regularity gain $\de\ge 0$. The following double-endpoint estimate was obtained in \cite[Theorem 1.3]{endStr}, while the single endpoint estimate was proven before in \cite{GW}. We refer to \cite{GW,endStr} and the references therein for more background and related works. 
\begin{lem}[Radial improved Strichartz for free waves] \label{lem:freeStz}
Let $d\ge 3$. There exists a continuous function $C\rd:[0,\de_\star)\to(0,\I)$ with $\de_\star:=\frac{d-1}{2d-1}$ such that for any interval $I\subset\R$ and any $u\in C(I;H^s\rad(\R^4))$, we have 
\EQ{ \label{RFD}
 \|\LR{D}^su\|_{\ti\X^\de(I)} \le C\rd(\de)\BR{\inf_{t\in I}\|u(t)\|_{H^s}+\|\LR{D}^s(i\p_t-\De)u\|_{\ti\X^\de_*(I)}}, }
where $\LR{D}:=\sqrt{1+D^2}=\sqrt{1-\De}=\F^{-1}(1+|\x|^2)^{1/2}\F$. 
\end{lem} 
Note that the estimate in $\ti\X^\de$ is essentially equivalent to that in $\ck\X^\de$ via the Littlewood-Paley decomposition, which commutes with $\LR{D}^s$ and $i\p_t-\De$. It is also easy to reduce to the case $s=0$ and then to $u=u_1$ by scaling. 

Next we recall the Strichartz estimate for the Schr\"odinger equation with a static potential. 
The following double-endpoint estimate was obtained in \cite{M} in general dimensions, while the 3D case was proven before in \cite{Be}, and the single endpoint in general dimensions in \cite{Go}. 
\begin{lem}[Strichartz with static potential] \label{lem:stSt}
Let $d\ge 3$ and $V\in L^{d/2}(\R^d;\R)$ satisfying $-\De+V>0$ on $L^2(\R^d)$. Then there exists $C\in(0,\I)$ such that for any interval $I\subset\R$ and any $u\in C(I;L^2(\R^d))$, 
\EQ{
 \|u\|_{\ck\X^0(I)} \le C\BR{\inf_{t\in I}\|u(t)\|_2 + \|(i\p_t-\De-V)u\|_{\ck\X^0_*(I)}}. }
\end{lem}
All of \cite{Go,Be,M}, as well as the preceding results, work under more general spectral conditions on the potential and with the continuous spectral projection $P_c$ in front of $u$ in the estimate, but the positivity $-\De+V>0$ is a sufficient condition for them together with $P_c=id$. 
Besides that, their potentials are in more general classes, as \cite{Be,M} considered the completion of $C_c^\I(\R^d)$ in the Lorentz space $L^{d/2,\I}(\R^d)\supset L^{d/2}(\R^d)$, while \cite{Go} considered time-periodic potentials. 
For more detail, background and related works, we refer to those papers and the references therein. 

Finally, we give the Strichartz estimate for small free wave potentials on $\R^4$, which follows immediately from the analysis in \cite{Zak4D1}. Similar arguments were used before in \cite{radZak3D} in the 3D radial setting.  This estimate is mentioned just for comparison, but will not be used in this paper. 
\begin{lem} \label{lem:smallSt}
For any $0\le s<1$, there exist $\e\in(0,1)$ and $C\in(1,\I)$ such that for any interval $I\subset\R$ and any $V\in C(I;L^2(\R^4))$ satisfying $(i\p_t+D)V=0$ with $\|V(t)\|_2\le\e$ on $I$, we have the Strichartz estimate for any $u(t,x)\in C(I;H^s(\R^4))$:
\EQ{
 \|\LR{D}^su\|_{\ti \X^0(I)} \le C\BR{\inf_{t\in I} \|u(t)\|_{H^s} + \|\LR{D}^s(i\p_t-\De-\re V)u\|_{\ti \X_*^0(I)}}.}
\end{lem}

The following is the main Strichartz estimate to be proven in this paper, which is in the 4D radial setting.  
\begin{thm} \label{thm:Stz} 
Let $0\le s<1$, $0\le \de<\de_\star$,  $0\le B<\|W^2\|_2$ and $0<M<\I$. 
Then there exists $\e\fn=\e\fn(s,\de,B,M)\in(0,1)$ and $C\fn=C\fn(s,\de,B,M)\in(0,\I)$ such that for any interval $I\subset\R$ and any $V\in L^\I(I;L^2\rad(\R^4))$ in the form $V=V\sF+V\sN+V\sD$ satisfying $(i\p_t+D)V\sF=0$ on $I$ and 
\EQ{ \label{norm cond V}
 \pt \|V\sF\|_{L^\I_tL^2(I)} \le B, 
 \pq  \|V\sN\|_{(L^\I_t\dot B^1_{4/3,1}+L^1_t\dot B^1_{4,1})(I)}+\|V\sD\|_{L^\I_t\dot B^{-1}_{4,\I}(I)} \le \e\fn,
 \pr \|(i\p_t+D)V\sD\|_{(L^1_t\dot B^{-1}_{4,\I}+L^\I_t\dot B^{-1}_{4/3,\I})(I) } +\|V\sD\|_{L^\I_tL^2(I)} +\|V\sN\|_{L^\I_tL^2(I)}\le M,
}
we have the Strichartz estimate for any $u\in C(I;H^s\rad(\R^4))$, 
\EQ{ \label{Stz}
 \|\LR{D}^su\|_{\ti \X^\de(I)} \le C\fn\BR{ \inf_{t\in I}\|u(t)\|_{H^s}+\|\LR{D}^s(i\p_t-\De- \re V)u\|_{\ti \X^\de_*(I)}}. }
\end{thm}
\begin{rem}
If $V=0$, then the above result becomes the same as Lemma \ref{lem:freeStz} in the 4D radial case. 
If $V$ is a free wave, i.e.~$V=V\sF$, then the condition \eqref{norm cond V} on $V$ becomes simply $\|V(0)\|_2\le B$, and then Lemma \ref{lem:smallSt} is covered in the radial case. 
If $s=0$, then all the norms in \eqref{norm cond V} and \eqref{Stz} are scaling invariant, which allows us to replace $i\p_t+D$ with $i\p_t+\al D$ for any $\al>0$, keeping the uniform constant $C\fn$ (cf.~\eqref{parab scal} and the argument around it). Sending $\al\to 0$ yields the same estimate for static potentials, which implies Lemma \ref{lem:stSt} in the 4D radial case. 
\end{rem}
\begin{rem}
We can weaken the small norm of $V\sN$ to $L^\I_t\dot B^1_{4/3,\I}+L^1_t(\dot B^1_{4,\I}\cap L^\I)$, by a slight modification of product estimates in the proof, using the radially improved Strichartz estimate. In the application to the Zakharov system, however, we use only $L^\I_t B^{1+\e}_{4/3,2}\subset L^\I_t \dot B^1_{4/3,1}\subset L^\I_tL^2$. 
\end{rem}
The uniformness of $C\fn$ in the above Theorem \ref{thm:Stz} is crucial for our application to the Zakharov system. 
If we do not care about the uniformness, then it is not so hard to derive such an estimate from the static case Lemma \ref{lem:stSt}, by some approximation of the potential. 
The uniformness is not obvious either in the proof of the static case in the literature, but one can expect its importance in genral in nonlinear problems (especially in the critical case as we are considering). 
The threshold $B<\|W^2\|_2$ is optimal for the uniform global estimate, because of the ground state and the scaling invariance. 
In the case of static potential, it is necessary even for time-local estimates and for more general exponents. 
See Appendix \ref{ss:fail} for precise statements. 

The case of $I\subsetneq\R$ is reduced to the case of $I=\R$ by extending $V$ and $u$ such that $(i\p_t+D)V\sF=0$, $(i\p_t+D)V\sD=0$, $V\sN=0$ and $(i\p_t-\De-\re V)u=0$ on $\R\setminus I$. 
Therefore in the proof, we may assume $I=\R$ without losing generality. 
For convenience, the best constant for $V$ in the above estimate with $s=0$ is denoted by 
\EQ{ \label{def C0de}
 C_0^\de(V):=\sup_{0\not=u \in C(\R;L^2\rad(\R^d))}\frac{\|u\|_{\ti \X^\de}}{\inf_{t\in\R}\|u(t)\|_2+\|(i\p_t-\De- \re V)u\|_{\ti \X^\de_*}}.}

The decomposition $V=V\sF+V\sN+V\sD$ is motivated by the normal form transform for the Zakharov system, see \eqref{decop N} below. $V\sF$ corresponds to the main part solving the free wave equation, $V\sN$ corresponds to a bilinear form that is more regular, and $V\sD$ corresponds to the Duhamel term of regular quadratic and higher order terms. 

One may wonder if $V\sN$ and $V\sD$ are really needed in the proof of Theorem \ref{thm:Stz}, as they are assumed to be small. In other words, one may try to derive the general case from the simple case $V=V\sF$ by regarding $(V\sN+V\sD)u$ as perturbation. 
On one hand, it does not essentially simplify the proof, since it uses the normal form argument generating this type of error terms. 
On the other hand, perturbation of the potential is a difficult problem if $s>2\de_\star$, because the potential term and the normal form transform do not commute. It can be seen schematically as follows. 
Consider a potential of the form $V=V_0+V_1$ where $V_0\in L^2$ is the main part:
\EQ{
 (i\p_t-\De-V_0)u=V_1u.}
If $V_1\in L^2$ and $u\in H^s$, then $V_1u$ has no regularity, so the regularity gain $2\de_\star<s$ in the radial Strichartz estimate is not enough to recover the loss. The normal form yields some approximate solution $b$ to remove the singular part (namely high-low interactions) of this contribution, for the free Schr\"odinger equation: 
\EQ{
 (i\p_t-\De)b = V_1u + \text{error}.}
If $V_0$ were absent, $u-b$ would solve a better equation. However, in the presence of $V_0$, it satisfies 
\EQ{
 \pt (i\p_t-\De-V_0)(u-b)=V_0b - \text{error},
 \pr (i\p_t-\De)(u-b)=V_0u - \text{error}.}
Since $V_0\in L^2$, the regularity loss remains the same. 
Approximating $V_0$ by more regular functions, one can avoid the loss of regularity, but then it is difficult to keep the uniformness of estimate, since such an approximation depends on the profile of $V_0$ rather than the $L^2$ norm. 

The difficulty in perturbation of the potential prevents us from following the Kenig-Merle approach to prove the scattering in Theorem \ref{main}, in particular in the long-time perturbation argument needed in the profile decomposition. 
So we follow instead the idea of Dodson-Murphy \cite{DM}, which works directly on general solutions without any perturbation. 
The localized virial-type estimate, which is crucial in both the approaches, is derived similarly to the case of three space dimensions \cite{GNW}. 

\subsection{Outline of paper}
The rest of paper is organized as follows. First in Section \ref{sect:pre}, some notation and basic tools are prepared, namely, frequency decomposition, normal forms and a radial Strichartz estimate with static potential, as well as some dispersive properties of scattering solutions. 
Sections \ref{sect:prof dec}--\ref{sect:Stz} are devoted to a proof of Theorem \ref{thm:Stz}. 
Assuming that the Strichartz constant blows up for some sequence of wave potentials, in Section \ref{sect:prof dec} we apply the profile decomposition to the sequence, extracting some limit objects, namely the profiles. 
Then in Section \ref{sect:WT}, we treat the easier case where $s=0$ (no derivative) and the profiles are concentrated at one frequency, reducing it to the case of static potentials by approximation with step functions in time. In Section \ref{sect:Stz}, we reduce the general case to the previous one, using a sequence of Fourier weights adapted to the profiles' frequencies, and regularity gains by the normal forms for the high-low interactions and by the radial improved Strichartz estimate for the low-high and the high-high interactions, as well as for the remainder of the profile decomposition. 
The remaining Sections \ref{sect:var}--\ref{sect:bup} concern the Zakharov system. 
In Section \ref{sect:var}, we investigate static and variational properties of the system and the ground state, especially the link between the critical mass of the wave potential and the virial identity or the Nehari functional. 
In Section \ref{sect:LWP}, we study local and global wellposedness of the Zakharov system in larger spaces $H^s\times L^2$ ($1/2\le s\le 1$), without imposing the radial symmetry. 
Finally assuming the radial symmetry, we prove Theorem \ref{main} in Sections \ref{sect:GWP}--\ref{sect:bup}. The global wellposedness part is proved in Section \ref{sect:GWP} using the uniform global Strichartz estimate Theorem \ref{thm:Stz}. The scattering part is proved in Section \ref{sect:sc}, following the Dodson-Murphy approach \cite{DM}, and deriving a localized virial-type estimate for the Zakharov system on $\R^4$. The blow-up part is proved in Section \ref{sect:bup} by contradiction argument and the same virial-type estimate.  

\section*{Acknowledgements}
This work was originally initiated from intensive discussions with Ioan Bejenaru and Sebastian Herr during the trimester program``Harmonic Analysis and Partial Differential Equations" at the Hausdorff Research Institute for Mathematics in 2014. 
The authors are grateful to them for the discussions, which clarified the main difficulties and possible approaches, and for their comments later on the manuscript. 
This work was partially supported by ARC DP170101060. 
The second author (K.~N.) was supported by JSPS KAKENHI Grant Number 25400159 and JP17H02854.

\section{Preliminaries} \label{sect:pre}
In this section, basic notation and tools are introduced. Some of them are in common with \cite{Zak4D1}. 
\subsection{Frequency decomposition} 
For any subset $Q\subset 2^\Z$, any proposition on $x>0$ of the form $x\cP$, or any element $j\in 2^\Z$, the corresponding smooth cut-off in the dyadic frequencies is denoted by 
\EQ{ \label{def LPproj}
 \fy_j := P_j \fy:=\F^{-1}\chi_0(\x/j)\hat \fy, \pq \fy_Q:=P_Q \fy:=\sum_{j\in Q}P_j \fy, \pq \fy_\cP:=\sum_{2^\Z\ni j\cP} P_j\fy,}
where $\chi_0\in C^\I(\R^d)$ is a fixed function satisfying 
\EQ{
 \pt 0\le\chi_0\le 1,\pq \sum_{j\in 2^\Z}\chi_0(\x/j)=1\ (\x\not=0), 
 \pr \supp\chi_0\subset\{\x\in\R^d \mid 1/2<|\x|<2\}.  }
For a single piece of the Littlewood-Paley decomposition, we will often omit the third index of Besov norms:
\EQ{
 \|\fy_j\|_{\dot B^s_p} := \|\fy_j\|_{\dot B^s_{p,1}} = \|\fy_j\|_{\dot B^s_{p,\I}} = j^s\|\fy_j\|_p.} 
For the square summation over dyadic frequencies, we use the notation
\EQ{ \label{def squm}
 \squm_k a_k := \BR{\sum_k |a_k|^2}^{1/2}.}

\subsection{Normal forms}  \label{ss:NF} 
An essential tool is the normal form transform, by which we can avoid the loss of regularity from the rough potential. 
With a small parameter $0<\io\ll 1$ for frequency separation, two bilinear Fourier multipliers $\Om_\io^\pm$ are defined as follows. First, their frequency region is denoted by 
\EQ{ \label{def HL}
 \HL_\io:=\{(j,k)\in(2^\Z)^2 \mid \io j\ge \max(k,2)\}, \pq \LH_\io:=(2^\Z)^2\setminus \HL_\io,}
and the corresponding decomposition of a product of two functions $f(x),g(x)$ by 
\EQ{ \label{def HL-prod}
  (f,g)_{\HL_\io}:=\sum_{(j,k)\in \HL_\io} f_j g_k, \pq (f,g)_{\LH_\io}:=fg-(f,g)_{\HL_\io}=\sum_{(j,k)\in\LH_\io} f_jg_k.}
Then the operators $\Om_\io^\pm$ are defined by 
\EQ{ \label{def Ompm}
 \F\Om_\io^\pm = \sum_{(k,l)\in \HL_\io}\int_{\R^4}\frac{\hat f_k(\x-\y)\hat g_l(\y)}{|\x|^2\mp|\x-\y|-|\y|^2}d\y.}
Note that the high-low restriction $(k,l)\in\HL_\io$, or $k\gg \max(l,1)$, ensures 
\EQ{
 |\x|^2\mp|\x-\y|-|\y|^2 \sim k^2 \gg l^2+1,}
so we have a product estimate with two derivative gain: 
\EQ{ \label{bil est Om}
 \|\Om_\io^\pm(f_j,g_k)\|_p \lec (1+j+k)^{-2}\|f_j\|_{p_1}\|g_k\|_{p_2}}
for any $p,p_1,p_2\in[1,\I]$ satisfying $1/p=1/p_1+1/p_2$. Moreover, they satisfy 
\EQ{
 -\De\Om_\io^\pm(f,g)=(f,g)_{\HL_\io} \pm \Om_\io^\pm(Df,g)+\Om_\io^\pm(f,-\De g).}
For the actual potential, it is convenient to introduce 
\EQ{ \label{def Om}
 \Om_\io(f,g):=\frac12\Om_\io^+(f,g)+\frac12\Om_\io^-(\bar f,g),}
which is $\R$-bilinear, satisfying
\EQ{ 
 (i\p_t -\De)\Om_\io(V,u)=(\re V,u)_{\HL_\io}+i\Om_\io((\p_t-iD)V,u)+\Om_\io(V,(i\p_t-\De)u),}
and the same bilinear estimate as \eqref{bil est Om}. 

In the same way, a bilinear operator $\ti\Om_\io$ is defined for the wave equation by
\EQ{ \label{def tiOm}
 \F\ti\Om_\io(f,g)(\x):=\sum_{(k,l) \tor (l,k)\in \HL_\io}\int_{\R^4}\frac{\hat f_k(\x-\y)\hat g_l(\y)}{|\x|-|\x-\y|^2+|\y|^2}d\y,}
while the remaining frequency part is denoted by 
\EQ{ \label{def HH-prod}
 \pt (f,g)_{\HH_\io} := \sum_{\{(j,k),(k,j)\}\cap\HL_\io=\empt}f_jg_k.}
Since $(j,k)$ or $(k,j)\in\HL_\io$ implies 
\EQ{
 |\x|-|\x-\y|^2+|\y|^2 \sim 1+j^2+k^2,}
we have a product estimate with two derivative gain:
\EQ{
 \|\ti\Om_\io(f_j,g_k)\|_p \lec (1+j+k)^{-2}\|f_j\|_{p_1}\|g_k\|_{p_2}}
for $p,p_1,p_2\in[1,\I]$ satisfying $1/p=1/p_1+1/p_2$, and also  
\EQ{
 (i\p_t+D)\ti\Om_\io(u,\ba v) \pt= (u,\ba v)_{\HL_\io}+(\ba v,u)_{\HL_\io} 
 \prQ+ \ti\Om_\io((i\p_t-\De)u,\ba v) - \ti\Om_\io(u,\ba{(i\p_t-\De)v}).}

Using those operators, we can transform the equation \eqref{Zak} into a more regular form. 
Let $(u,N)$ be a (local) solution of \eqref{Zak}. Then we have 
\EQ{ \label{eq normal}
 \CAS{(i\p_t-\De)(u-\Om_\io(N,u))=F_\io(u,N),\\
 (i\p_t+D)(N-D\ti\Om_\io(u,\bar u))=G_\io(u,N),} }
where $F_\io,G_\io$ are nonlinear operators defined by 
\EQ{ \label{def FG}
  \pt F_\io(u,N):=(n,u)_{\LH_\io}+i\Om_\io(iD|u|^2,u)-\Om_\io(N,nu),
  \pr G_\io(u,N):=D(u,\bar u)_{\HH_\io}-D\ti\Om_\io(nu,\bar u)+D\ti\Om_\io(u,n\bar u).}
The free propagator is denoted by 
\EQ{ \label{def Uf}
 \Uf(t)(\fy,\psi):=(e^{-it\De}\fy,e^{itD}\psi)}
and the Duhamel integral is denoted by 
\EQ{ \label{def Duh}
 \D_\io^T(u,N) \pt:=(\cU_\io^T(u,N),\cN_\io^T(u,N))
 \pr:=-i\int_T^t \Uf(t-\ta)(F_\io(u,N),G_\io(u,N))(\ta)d\ta.}
The normal form transform is denoted by, for $\io_1,\io_2\in(0,1)$,   
\EQ{ \label{def normal}
 \pt\vec\Om_{\io_1,\io_2}(\fy,\psi):=(\Om_{\io_1}(\psi,\fy),D\ti\Om_{\io_2}(\fy,\bar\fy)), 
 \pr\Psi_{\io_1,\io_2}(\fy,\psi):=(\fy,\psi)-\vec\Om_{\io_1,\io_2}(\fy,\psi),}
where we will need to choose $\io_1\not=\io_2$ only in the proof of Proposition \ref{prop:bc0}. 
Otherwise we can choose $\io_1=\io_2$, in which case it is abbreviated by 
\EQ{
 \vec\Om_\io(\fy,\psi):=\vec\Om_{\io_1,\io_2}(\fy,\psi), \pq \Psi_\io(\fy,\psi):=\Psi_{\io,\io}(\fy,\psi).}
Hence the Duhamel form of \eqref{eq normal} is written as 
\EQ{
 \Psi_\io(u,N)=\Uf(t-T)\Psi_\io(u,N)(T) + \D^T_\io(u,N),}
for any $T\in\R$ and $\io\in(0,1)$.

\begin{rem}
The operators $\Om^+_\io$ and $\HL_\io$ are essentially the same as $\Om$ and $XL$ in \cite{Zak4D1} (with the parameters $\al=1$ and $K=|\log_2\io|$ therein), except that the former's frequency does not contain $\text{high}<1/\io$, which is included into $\LH_\io$. 
Similarly, the operator $\ti\Om_\io$ is essentially the same as $\ti\Om$ in \cite{Zak4D1} except for $\text{high}<1/\io$, which is included into $\HH_\io$. 
This change is because the main analysis of this paper is essentially done in the inhomogeneous type of Besov spaces, starting from the Strichartz estimate without the derivative. The technical difference is minor in the nonlinear estimates, but this paper is written in a self-contained way.  

The other operator $\Om^-_\io$ was ignored in \cite{Zak4D1}, since the argument therein exhibits no difference from $\Om^+_\io$. The situation is the same in this paper, but $\Om_\io$ is introduced just for the sake of rigor. 
\end{rem}

\subsection{Automatic radial improvement}
In proving the radial improved Strichartz estimate with potential, we may fix the parameter $\de$ of the regularity gain. In the case of $L^2_x$ solutions, the following lemma allows us to ignore it. 
\begin{lem} \label{lem:auto-delta}
Let $d\ge 3$ and $\de\in[0,\de_\star)$. Then there exists $C\in(0,\I)$ such that for any $V\in L^\I(\R;L^{d/2}\rad(\R^d))$ we have $C_0^\de(V)\le C(1+\|V\|_{L^\I_t(L^{d/2})}C_0^0(V))^2$. 
\end{lem}
Hence Lemma \ref{lem:stSt} implies 
\begin{cor} \label{lem:St-st}
Let $d\ge 3$, $\de\in[0,\de_\star)$ and $V\in L^{d/2}\rad(\R^d)$ satisfying $-\De+\re V>0$ on $L^2(\R^d)$. Then $C_0^\de(V)<\I$. 
\end{cor}
\begin{proof}[Proof of Lemma \ref{lem:auto-delta}]
Let $v:=\re V$, $u\in C(\R;L^2\rad(\R^d))$ and $f:=(i\p_t-\De+v)u\in\ti\X^\de_*$. By definition of $C_0^0(V)$, we have for any $c\in I$, 
\EQ{ \label{est in Lp}
 \|u\|_{\ti\X^0(I)} \le C_0^0(V)\BR{\|u(c)\|_2 + \|f\|_{\ti \X^0_*}}.}
Applying Lemma \ref{lem:freeStz} to $(i\p_t-\De)u=f-vu$, and then using the embeddings $L^2_tL^{2(1)}\subset\ti \X^\de_*$ and $\ti X^0\subset L^2_tL^{2(-1)}$ with H\"older $L^{\I(2)}\times L^{2(-1)}\subset L^{2(1)}$ and the above estimate, we obtain 
\EQ{ \label{free Duh}
 \|u\|_{\ti \X^\de} \pt\le C\rd(\de)\BR{\|u(c)\|_2 + \|f-vu\|_{\ti \X^\de_*}}
 \pr\le C\rd(\de)\BR{\|u(c)\|_2 + \|f\|_{\ti \X^\de_*} + C(\de)\|V\|_{L^\I_tL^{\I(2)}}\|u\|_{\ti \X^0}}
 \pr\le C\rd(\de)(1+MC(\de)C_0^0(V))\BR{\|u(c)\|_2 + \|f\|_{\ti \X^0_*}}, }
where $M:=\|V\|_{L^\I_t(L^{d/2})}$ and $C(\de)$ is a constant coming from the embeddings. 

In the case $f=0$, the above estimate is enough. 
The estimate in the general case $f\not=0$ is reduced to the case of $f=0$ and the case of $u(c)=0$ by decomposing $u$ into two corresponding parts, namely
\EQ{
 u=u_0+u_1, \pq \CAS{ (i\p_t-\De+v)u_0=0, \pq u_0(c)=u(c),\\ (i\p_t-\De+v)u_1=f,\pq u_1(c)=0.} }
In order to improve the right side of the above estimate for $u_1$, we use the duality: for any $b\in(c,\I)$, we have
\EQ{ \label{duality}
 \|u_1\|_{\ti \X^0(c,b)} \sim \sup_{\|g\|_{\ti \X^0_*(c,b)}\le 1}\int_c^b \LR{u_1|g}dt,}
where the norm equivalence can be estimated by an absolute constant. 
Let $u_2\in C(\R;L^2)$ be the solution to the dual equation 
\EQ{
 (i\p_t-\De+v)u_2=\CAS{g &(c<t<b)\\ 0 &\text{(otherwise)}}, \pq u_2(b)=0.}
Then partial integration and the above estimate for $u_2$ yield 
\EQ{
 \int_c^b \LR{u_1|g}dt \pt= \int_c^b \LR{u_1|(i\p_t-\De+v)u_2}dt 
 \pr= [\LR{u_1|iu_2}]_c^b + \int_c^b \LR{(i\p_t-\De+v)u_1|u_2}dt
 = \int_c^b \LR{f|u_2}dt
 \pr\lec \|f\|_{\ti\X^\de_*(c,b)} \|u_2\|_{\ti\X^\de(c,b)} 
 \pr\le C\rd(\de)(1+MC(\de)C_0^0(V))\|f\|_{\ti\X^\de_*},}
where the embedding $\ti X^0_*\subset\ti X^\de_*$ was used for $g$. 
Thus by duality \eqref{duality}, we obtain 
\EQ{
 \|u_1\|_{\ti X^0(c,b)} \lec C\rd(\de)(1+MC(\de)C_0^0(V))\|f\|_{\ti\X^\de_*}.}
The norm on the other side $(a,c)$ for $a\in(-\I,c)$ is estimated in the same way. 
Since the bound is uniform for all $a<c<b$, it implies the same estimate on $\R$. 
Combining it with the estimate for $u_0$, we obtain 
\EQ{
 \|u\|_{\ti \X^0} \lec C\rd(\de)(1+MC(\de)C_0^0(V))\BR{\|u(c)\|_2 + \|f\|_{\ti \X^\de_*}}.}
Repeating \eqref{free Duh} with this improved estimate leads to the desired conclusion.
\end{proof}

\subsection{Dispersive decay estimates}
We will often rely on dispersive decay for large time of free solutions and more generally scattering solutions in $H^s\times L^2$.  
\begin{lem} \label{lem:scdec}
Let $s\in(0,2]$, $\de\in(0,2]$, and $(\fy,\psi)\in H^s(\R^4)\times L^2(\R^4)$. If $u(t),N(t)$ defined for large $t>0$ satisfies 
\EQ{ \label{def scat}
 \|u(t)-e^{-it\De}\fy\|_{H^s}+\|N(t)-e^{itD}\psi\|_2 \to 0 \pq (t\to\I),}
then as $T\to\I$,
\EQ{ \label{scat dec}
 \|u\|_{L^\I_t L^{2(-s)}(T,\I)}+ \|N\|_{L^\I_t \dot B^{-\de}_{2(-\de),2}(T,\I)}+\|N\|_{(L^\I_tL^2+L^2_tL^4)(T,\I)}\to 0.}
\end{lem}
When \eqref{def scat} holds, we say that {\it $(u,N)$ scatters in $H^s\times L^2$ with the scattering profile $(\fy,\psi)$.} 
\begin{proof}
For any $\e>0$, there exists $\fy',\psi'\in\cS(\R^4)$ such that $\|\fy-\fy'\|_{H^s}+\|\psi-\psi'\|_2<\e$. The dispersive decay for the free equations implies 
\EQ{
  \|e^{-it\De}\fy'\|_{2(-s)}+\|e^{itD}\psi'\|_{\dot B^{-\de}_{2(-\de),2}}\to 0, \pq \|e^{itD}\psi'\|_4 \le O(t^{-3/4})}
as $t\to\I$, so that we can find $T\in(0,\I)$ such that 
\EQ{
 \|e^{-it\De}\fy'\|_{L^\I_t L^{2(-s)}(T,\I)}+ \|e^{itD}\psi'\|_{L^\I_t \dot B^{-\de}_{2(-\de),2}(T,\I)} + \|e^{itD}\psi'\|_{L^2_tL^4(T,\I)}<\e.} 
On the other hand, the unitarity of the free propagator on $H^s\times L^2$ and the Sobolev embeddings $H^s\subset L^{2(-s)}$ imply 
\EQ{
 \|u(t)-e^{-it\De}\fy'\|_{L^{2(-s)}} + \|N(t)-e^{itD}\psi'\|_{L^2\cap \dot B^{-\de}_{2(-\de),2}} \lec \e +o(1),}
where $o(1)\to 0$ as $t\to\I$ is independent of $(\fy',\psi')$. 
Adding these two estimates and sending $\e\to+0$ yield the conclusion. 
\end{proof}
Note however that $L^2_tL^4$ is not admissible for the Strichartz estimate of $e^{itD}$ on $L^2(\R^4)$, so the $L^\I_tL^2$ component can not be eliminated from \eqref{scat dec}. 
The idea of estimating $N$ in $L^\I_tL^2+L^2_tL^4$ was already in \cite{Zak4D1}.

\section{Profile decomposition of the wave potential} \label{sect:prof dec}
We prove Theorem \ref{thm:Stz} by contradiction. 
Suppose that for some $s,\de,B,M$ in the range of the theorem, there is no $(\e\fn,C\fn)$ with the desired property. 
Then there exist sequences of radial functions $V_n(t,x),u_n(t,x)$ such that 
\EQ{ \label{choice Vn fn}
 \pt V_n=V\sF_n+V\sN_n+V\sD_n, \pq (i\p_t+D)V\sF_n=0,  
 \pr \|V\sF_n\|_{L^\I_tL^2} \le B,  \pq \|V\sN_n\|_{L^\I_t\dot B^1_{4/3,1}+L^1_t\dot B^1_{4,1}}+\|V\sD_n\|_{L^\I_t\dot B^{-1}_{4,\I}} \to 0,
 \pr \|(i\p_t+D)V\sD_n\|_{L^1_t \dot B^{-1}_{4,\I}+L^\I_t\dot B^{-1}_{4/3,\I}} + \|V\sD_n\|_{L^\I_t L^2} + \|V\sN_n\|_{L^\I_t L^2} \le M,
 \pr  \|\LR{D}^su_n\|_{\ti \X^\de} \to\I,
 \pr  \|u_n(0)\|_{H^s}+\|\LR{D}^s(i\p_t-\De-\re V_n)u_n\|_{\ti \X^\de_*} \to 0,   }
where the linearity in $u$ is used to add the last property, and the term $\inf_t\|u(t)\|_{H^s}$ in \eqref{Stz} is replaced with $\|u(0)\|_{H^s}$ by appropriate time translation of $V_n$ and $u_n$. 
Without losing generality, we may and do assume that 
\EQ{
 1+B \ll M.}

Apply the Bahouri-G\'erard profile decomposition to the sequence of free waves $V\sF_n$, which is bounded in $L^2_x$. Then passing to a subsequence, (which does not affect \eqref{choice Vn fn}), there exist sequences $\{t^j_n\}\subset\R$, $\{\si^j_n\}\subset(0,\I)$, $\{\psi^j\}\subset L^2\rad$ such that 
\EQ{ \label{initial PRD}
 \Ga^{0,J}_n:=V\sF_n-\sum_{0\le j<J} V^j_n, \pq V^j_n:=e^{i(t-t^j_n)D}S(\si^j_n)\psi^j,}
with $S(\si)\fy(x)=\si^2\fy(\si x)$ denoting the $L^2$-invariant dilation, satisfy 
\EQ{ \label{Ga small0}
 \lim_{J\to\I}\limsup_{n\to\I}\|\Ga^{0,J}_n\|_{L^\I_t \dot B^{-1}_{4,\I}}=0,}
and, for each $j,k<J$ with $j\not=k$, 
\EQ{ \label{weak van}
   S(\si^j_n)\Ga^{0,J}_n(t^j_n)\to 0 \IN{\weak{L^2_x}},}
\EQ{ \label{AOT}
 \log(\si^j_n/\si^k_n)\to\pm\I \tor 
  [\si^j_n=\si^k_n \tand  \si^j_n|t^j_n-t^k_n|\to\I], }
\EQ{ \label{scal lim0}
 \si^j_n \equiv 1,\pq \si^j_n\to 0, \tor \si^j_n\to\I,}
as $n\to\I$. The parameters $\si^j_n,t^j_n$ represent the scale and time of the profile $V^j_n$. 

The orthogonality \eqref{AOT}, together with the weak vanishing \eqref{weak van} and the dispersive decay of free waves, implies that each profile $\psi^j$ is given by the weak limit 
\EQ{
 S(1/\si^j_n) V\sF_n(t^j_n) \to \psi^j \IN{\weak{L^2_x}}\pq(n\to\I).}

We can include all the profiles with $\si^j_n\to 0$ into the other parts $\Ga^{0,J}_n,V\sN_n,V\sD_n$. More precisely, we have 
\begin{lem}
For any $\psi\in L^2(\R^4)$ and any sequence $\si_n\to+0$, there exists a sequence $\{\fy_n\}_n\subset W^{3,4/3}(\R^4)$ satisfying $\|\fy_n-\psi\|_2\to 0$ and $\|\fy_n\|_{W^{3,4/3}}\si_n^{3/4}\to 0$. 
Moreover, if $\fy_n$ is such a sequence, then for any $\chi\in C^1(\R)$ satisfying $\chi(t)=1$ for $|t|\le 1$ and $\chi(t)=0$ for $|t|\ge 2$, and 
\EQ{
 \pt v_n:=e^{itD}S(\si_n)\psi, \pq v^0_n:=e^{itD}S(\si_n)(\psi-\fy_n), 
 \pr v^1_n:=(1-\chi)(\si_n^2 t)e^{itD}S(\si_n)\fy_n,
 \pq v^2_n:=\chi(\si_n^2 t)e^{itD}S(\si_n)\fy_n,}
we have 
\EQ{ \label{decop expanding}
 \pt v_n = \sum_{k=0}^2 v^k_n, \pq\|v^0_n\|_{L^\I_tL^2}=\|\psi-\fy_n\|_2\to0,
 \pr \|v^1_n\|_{L^\I_t\dot B^{-1}_{4,1}}+\|(i\p_t+D)v^1_n\|_{L^1_t\dot B^{-1}_{4,1}}+\|v^2_n\|_{L^1_t\dot B^1_{4,1}}  \to 0.}
\end{lem}
\begin{proof}
The existence of a sequence $\fy_n$ is obvious from the density of $W^{3,4/3}\subset L^2$ and $\si_n\to 0$, so are the properties of $v_n,v^0_n$ in \eqref{decop expanding}. For $v^1_n$, the scaling invariance and the dispersive decay of the wave equation $e^{itD}$ yield 
\EQ{
 \|v^1_n\|_{L^\I_t\dot B^{-1}_{4,1}} \pt\le \|e^{itD}\fy_n\|_{L^\I_t\dot B^{-1}_{4,1}(|\si_n t|\ge 1)}
 \lec \si_n^{3/4}\|\fy_n\|_{\dot B^{1/4}_{4/3,1}},}
and similarly, 
\EQ{
 \|(i\p_t+D)v^1_n\|_{L^1_t\dot B^{-1}_{4,1}} \pt\le \|\chi'\|_{L^1_t}\|e^{itD}\fy_n\|_{L^\I_t\dot B^{-1}_{4,1}(|\si_n t|\ge 1)}
 \lec \si_n^{3/4}\|\fy_n\|_{\dot B^{1/4}_{4/3,1}}.}
The same estimates (with different exponents) yield for $v^2_n$ 
\EQ{
 \|v^2_n\|_{L^1_t \dot B^1_{4,1}} \pt\le \si_n \|e^{itD}\fy_n\|_{L^1_t\dot B^1_{4,1}(|\si_n t|\le 2)}
 \lec \si_n^{3/4}\|\fy_n\|_{\dot B^{9/4}_{4/3,1}}.}
The Besov norms of $\fy_n$ are dominated by $W^{3,4/3}$, hence \eqref{decop expanding} follows. 
\end{proof}

For the profile decomposition \eqref{initial PRD}, the above lemma yields a further decomposition 
$V^j_n = \sum_{k=0}^2 V^{j,k}_n$ for each $j\in D_J:=\{j<J \mid \si^j_n\to 0\}$. 
Distributing them into the three parts by 
\EQ{ \label{second PRD} 
 \pt \Ga^{1,J}_n := \Ga^{0,J}_n + \sum\{V^{j,0}_n\mid j\in D_J\},
 \pr V\sD_{n,J}:=V\sD_n+\sum\{V^{j,1}_n\mid j\in D_J\},
 \pr V\sN_{n,J}:=V\sN_n+\sum\{V^{j,2}_n\mid j\in D_J\},}
we obtain a modified decomposition
\EQ{ \label{decop Vn}
 V_n = V\sF_{n,J}+ V\sD_{n,J} + V\sN_{n,J},\pq V\sF_{n,J}=\sum_{j\in C_J} V^j_n  + \Ga^{1,J}_n,}
with $C_J:=\{j\in\Z\mid 0\le j<J\}\setminus D_J$, satisfying the parameter separation \eqref{AOT}, 
\EQ{ \label{Ga small1}
 \lim_{J\to\I}\limsup_{n\to\I}\|\Ga^{1,J}_n\|_{L^\I_t\dot B^{-1}_{4,\I}}=0,}
and
\EQ{ \label{scal lim1}
 j\in C_J \implies \si^j_n \equiv 1 \tor \si^j_n\to \I,}
\EQ{
 \pt\|V\sN_{n,J}\|_{L^\I_t\dot B^1_{4/3,1}+L^1_t\dot B^1_{4,1}} + \|V\sD_{n,J}\|_{L^\I_t\dot B^{-1}_{4,\I}} \to 0,
 \pr\|(i\p_t+D)V\sD_{n,J}\|_{L^1_t\dot B^{-1}_{4,\I}+L^\I\dot B^{-1}_{4/3,\I}}+\|V\sD_{n,J}\|_{L^\I_tL^2}+\|V\sN_{n,J}\|_{L^\I_tL^2} < 2M,}
as $n\to\I$ or by passing to a further subsequence, for the last $L^\I_tL^2$ bound. 
Note that the time cut-off functions in $V^{j,k}_n$ for $k=1,2$ do not disturb the asymptotic $L^2_x$ orthogonality among $j$ by the parameter separation \eqref{AOT}.

The Schr\"odinger-rescaling of the profiles is denoted by 
\EQ{ \label{def tiVjn}
 \ti V^j_n(t):=S(1/\si^j_n)V^j_n(t/(\si^j_n)^2)=e^{i(t/\si^j_n-\ta^j_n)D}\psi^j, \pq \ta^j_n:=\si^j_n t^j_n.} 
It is convenient to classify the profiles by their scales. Let $J_\sim$ be the quotient set for $j\le J$ with the scale $\si^j_n\equiv 1$ or $\to \I$:
\EQ{ \label{def Jsim}
 J_\sim := \{\{k\in C_J \mid \forall n\in\N,\ \si^j_n=\si^k_n\} \mid j\in C_J\}.}
For any $A\in J_\sim$, \eqref{AOT} implies that 
\EQ{
  j,k\in A \tand j\not=k \implies |\ta^j_n-\ta^k_n|\to\I \pq(n\to\I),}
and also that 
\EQ{ \label{sep R}
 \inf_{j,k<J,\ j\not\sim k}\max(\si^j_n/\si^k_n, \si^k_n/\si^j_n)\to \I \pq (n\to\I).}
Let $\si^A_n:=\si^j_n$ for any $j\in A\in J_\sim$. The set $J_\sim$ is naturally ordered by 
\EQ{ \label{scal sep}
 A,B \in J_\sim,\ A<B \define \lim_{n\to\I} \si^B_n/\si^A_n = \I,}
which is a linear order because of \eqref{sep R}. 
The profiles classified by the scaling are denoted by 
\EQ{ \label{def VnA}
 V^A_n:=\sum_{j\in A} V^j_n, \pq \ti V^A_n:=\sum_{j\in A}\ti V^j_n.}
Then we have 
\EQ{ \label{decop VF}
 \pt V\sF_n = \sum_{A\in J_\sim} V^{A}_n + \Ga^{1,J}_n, \pq V^{A}_n(t)=S(\si^A_n)\ti V^{A}_n((\si^A_n)^{2}t),}
and 
\EQ{
 \ti V^A_n = \sum_{j\in A}e^{i(t/\si^A_n-\ta^j_n)}\psi^j, \pq (j\not=k\implies |\ta^j_n-\ta^k_n|\to\I).}
Let us call such a sequence $\{\ti V^A_n\}_n$ {\it a wave train}.

\section{Strichartz with a single wave train} \label{sect:WT}
In this section, we consider the simple case with a single wave train as the potential and $s=0$, namely the Strichartz estimate for $L^2_x$ solutions. 
The goal of this section is to prove the following 
\begin{lem} \label{lem:St-wt}
There is a constant $\e\wt\in(0,1)$ with the following property. 
For any $\de\in[0,\de_\star)$ and for any finite subset $P\subset L^2\rad(\R^4)$, there exists $C\wt(\de,P)\in(1,\I)$ such that for any finite subset $\ti P\subset L^2\rad(\R^4)$ satisfying 
\EQ{ \label{later prof small}
 \sum_{\fy\in\ti P}\|\fy\|_2^2 < \|W^2\|_2^2 \tand \sup_{\fy\in \ti P\setminus P}\|\fy\|_2 \le \e\wt,}
and for any $\ta:\ti P\to\R$ with sufficiently large (depending on $P$ and $\ti P$) 
\EQ{
 d(\ta):=\sup_{\fy,\psi\in\ti P, \fy\not=\psi}|\ta(\fy)-\ta(\psi)|,}
we have 
\EQ{ \label{est wt}
 \sup_{\si>0}C_0^\de(S(\si)\sum_{\fy\in \ti P} e^{i(\si t-\ta(\fy))D}\fy) \le C\wt(\de,P). }
\end{lem}
$\ti P$ is the set of profiles of a wave train, and $P$ is a subset containing those profiles that are not small enough. 
The first condition in \eqref{later prof small} could be weakened to
\EQ{
 \sup_{\fy\in P\cap \ti P}\|\fy\|_2 < \|W^2\|_2,}
with the price for $\e\wt,C\wt$ to depend on $\sum_{\fy\in\ti P}\|\fy\|_2^2$, but the above version is enough for our purpose. 
In the case of \eqref{later prof small}, after minimizing $P\subset\ti P$ such that the second condition holds, the number of $P$ is a priori bounded.

The $L^2$-Strichartz has the following scale invariance. Suppose that $u,V,f$ satisfy the equation 
\EQ{
 (i\p_t-\De-\re V)u=f.}
For $\la>0$, define $u_\la,V_\la,f_\la$ by the parabolic rescaling
\EQ{ \label{parab scal}
 \pt u_\la(t,x):=\la^{2} u(\la^2 t,\la x), \pq V_\la(t,x):=\la^2 V(\la^2 t,\la x), 
 \pr f_\la(t,x):=\la^{4}f(\la^2t, \la x).}
Then we have  
\EQ{
 \pt (i\p_t-\De-\re V_\la)u_\la=f_\la,
 \pq \|u_\la\|_Z\approx\|u\|_Z, \pq  \|f_\la\|_Y\approx\|f\|_Y,}
for any $Z\in\{L^\I_tL^2, \cL^\I_tL^2, \X^\de\}$ and $Y\in\{L^1_tL^2, \cL^1_tL^2, \X^\de_*\}$, so 
\EQ{
  C_0^\de(V_\la) \approx C_0^\de(V).}
Hence \eqref{est wt} implies 
\EQ{ \label{rescaled wt}
 \sup_{\si>0}C_0^\de(\sum_{\fy\in \ti P} e^{i(t/\si-\ta(\fy))D}\fy) \lec C\wt(\de,P),}
and vice versa (up to a constant multiple).  
The norm equivalence in the above (or the implicit constants in `$\approx$') is uniform for $\de$. In fact, if the Besov norms are defined to be scaling invariant, then it becomes the exact equality, and so `$\lec$' in \eqref{rescaled wt} is also replaced with `$\le$'.

The above lemma is proved by approximating the free wave potential by a step function in time where the profiles are relatively large. 
We start with an easy perturbation lemma in $L^2_x$: 
\begin{lem} \label{lem:L2pert}
There exists $\e\pert:(0,\I)\to(0,1)$ such that for any $C_*\in(0,\I)$, $\de\in[0,3]$ and $v,v'\in L^\I(\R;L^2(\R^4))$ we have 
\EQ{
 C_0^\de(v) \le C_* \tand \|v-v'\|_{L^\I_tL^2} \le \e\pert(C_*) \implies C_0^\de(v')\le 2C_*.}
\end{lem}
\begin{proof}
Let $(i\p_t-\De+v')u=f$ and $\|v-v'\|_{L^\I_tL^2}\le\e\in(0,1)$. Then 
$(i\p_t-\De+v)u=f+(v'-v)u$ and by H\"older
\EQ{
 \|(v'-v)u\|_{L^2_tL^{4/3}} \le \|v'-v\|_{L^\I_tL^2}\|u\|_{L^2_tL^4}
 \le \e \|u\|_{L^2_tL^4}.}
Hence using the embeddings $\dot B^\de_{4(\de),2}\subset L^4$ and $L^{4/3}\subset \dot B^{-\de}_{4/3(-\de),2}$, we obtain 
\EQ{
 \pn\|u\|_{\ti \X^\de}
 \pt\le C_0^\de(v)\BR{\inf_t\|u(t)\|_{L^2}+\|f\|_{\ti \X^\de_*}+\|(v'-v)u\|_{L^2_tL^{4/3}}}
 \pr\le C_*\BR{\inf_t\|u(t)\|_{L^2}+\|f\|_{\ti \X^\de_*}}+CC_*\e\|u\|_{\X^\de},}
for some absolute constant for the embeddings. The last term can be absorbed by the left side if $CC_*\e<1/2$. Hence it suffices to choose $\e\pert\ll1/C_*$. 
\end{proof}

For the dispersed part of the potential, we need the following perturbation lemma, which relies on the radial improvement of the free Strichartz estimate.
\begin{lem} \label{lem:Bsmall}
There exist $\e\ds:(0,\I)\to(0,1)$ and $C\ds:[0,\de_\star)\to(0,\I)$ such that for any $\de\in[0,\de_\star)$,  $B\in(0,\I)$ and any $V\in L^\I(\R;L^2\rad(\R^4))$, we have 
\EQ{ \label{bd V}
  \|V\|_{L^\I_tL^2} \le B \tand \|V\|_{L^\I_t \dot B^{-1}_{4,\I}}\le\e\ds(B) \implies C_0^\de(V) \le C\ds(\de).}
\end{lem}
\begin{proof}
By Lemma \ref{lem:auto-delta}, it suffices to prove fo a fixed $\de\in(0,\de_\star)$. 
Suppose $\|V\|_{L^\I_t \dot B^{-1}_{4,\I}}\le\e$. The complex interpolation and the Sobolev embedding yield
\EQ{ \label{intpl BH}
 [L^2,\dot B^{-1}_{4,\I}]_{\de/2} = \dot B^{-\de/2}_{2(-\de/2),4/(2-\de)} \subset \dot B^{-\de/2}_{2(-\de/2),\I}.}
Hence $\|V\|_{L^\I_t \dot B^{-\de/2}_{2(-\de/2),\I}} \lec B^{1-\de/2}\e^{\de/2}$ can be made 
as small as we wish by choosing $\e\le\e\ds$ small enough. 
Now we use a standard product estimate in the Besov space:
\EQ{ \label{prod neg Bes}
 0<\de\le 1\implies \|fg\|_{\dot B^{-\de}_{4/3(-\de),2}} \le C(\de)\|f\|_{\dot B^{-\de/2}_{2(-\de/2),\I}}\|g\|_{\dot B^{\de}_{4(\de),2}}.}
For convenience and later use, we give a brief
\begin{proof}[Proof of \eqref{prod neg Bes}]
By H\"older and Sobolev (or Young for the convolution in $x$), 
\EQ{
 \|(f_jg_k)_l\|_{\dot B^{-\de}_{4/3(-\de)}} \pt\lec \min(j,k,l)^{3\de/2}l^{-\de}j^{\de/2}k^{-\de}\|f_j\|_{\dot B^{-\de/2}_{2(-\de/2)}}\|g_k\|_{\dot B^\de_{4(\de)}}
 \pr\lec(\min(j,k,l)/\max(j,k,l))^{\de/2}\|f_j\|_{\dot B^{-\de/2}_{2(-\de/2)}}\|g_k\|_{\dot B^\de_{4(\de)}},}
for all $j,k,l\in 2^\Z$ satisfying $j\lec k\sim l$, $k\lec l\sim j$ or $l\lec j\sim k$. 
Hence the desired estimate follows from Young for the convolution over $2^\Z$ in each of those three cases. 
\end{proof}
The free Strichartz estimate \eqref{RFD} with $s=0$ yields, with $v:=\re V$,
\EQ{ \label{RFD used}
 \|u\|_{\ti\X^\de} \le C\rd(\de)[\inf_t\|u(t)\|_2 + \|(i\p_t-\De+v)u\|_{\ti\X^\de_*}+\|vu\|_{L^2_t\dot B^{-\de}_{4/3(-\de),2}}],}
while \eqref{intpl BH} and \eqref{prod neg Bes} yield
\EQ{
 \|vu\|_{L^2_t\dot B^{-\de}_{4/3(-\de),2}} 
 \lec \|V\|_{L^\I_t \dot B^{-\de/2}_{2(-\de/2),\I}} \|u\|_{L^2_t \dot B^\de_{4(\de),2}}
 \lec B^{1-\de/2}\e^{\de/2}\|u\|_{\ti\X^\de}.}
If we choosing $\e>0$ small enough, depending on $B$ and the fixed $\de$, then the last term is absorbed by the left side of \eqref{RFD used}, hence $C^\de_*(V)<\I$. 
\end{proof}

Now we are ready to prove the $L^2_x$-Strichartz estimate for a single wave train. 
\begin{proof}[Proof of Lemma \ref{lem:St-wt}]
It suffices to prove for a fixed $\de\in(0,\de_\star)$, say $\de=\de_\star/2$, thanks to Lemma \ref{lem:auto-delta}. 
Henceforth we ignore the dependence on $\de$. 
Also, it suffices to prove the rescaled estimate \eqref{rescaled wt}. 
Let $P,\ti P\subset L^2\rad$, $\si>0$, $\ta: \ti P\to \R$ and $f\in L^2_t\dot B^{-\de}_{4/3(-\de),2}$. Let $u$ be a solution of  
\EQ{
 (i\p_t-\De-\re V)u=f}
with the potential of a rescaled wave train 
\EQ{
 V:=\sum_{\fy\in  \ti P}V(\fy), \pq V(\fy):=e^{i(t/\si-\ta(\fy))D}\fy.}

Let $\e_1:=\e\ds(\|W^2\|_2)\in(0,1)$ be the small constant given by Lemma \ref{lem:Bsmall}.  
Lemma \ref{lem:scdec} with $\de=1$ yield $T=T(P)\in(0,\I)$ and $\ti T=\ti T(\ti P)\in(T(P),\I)$ such that 
\EQ{ \label{V scat}
 \pt \sum_{\fy\in\ti P}\|e^{itD}\fy\|_{L^\I_t(|t|>\ti T;\dot B^{-1}_{4,2})} + \sup_{\fy\in P}\|e^{itD}\fy\|_{L^\I_t(|t|>T;\dot B^{-1}_{4,2})}<\e_1/2.}
Since $\|\fy\|_2<\|W^2\|_2$ for each $\fy\in P$, Corollary \ref{lem:St-st} implies that $C_0^\de(e^{it_0D}\fy)<\I$ for any fixed $t_0\in\R$. 
Then by the continuity of $e^{itD}$ on $L^2_x$ and the compactness of $[-T,T]$, Lemma \ref{lem:L2pert} implies that 
\EQ{
 C_*=C_*(P):=\sup_{\fy\in P,\ t_0\in[-T,T]}C_0^\de(e^{it_0D}\fy)<\I.}
Let $\e_2=\e_2(P):=\e\pert(C_*)\in(0,1)$ be the small constant given by Lemma \ref{lem:L2pert}.
The $L^2_x$ continuity of $e^{itD}$ implies that there exists $\ka=\ka(P)>0$ such that 
\EQ{
 \fy\in P,\ t_1,t_2\in[-T,T],\ |t_1-t_2|\le\ka \implies \|e^{it_1D}\fy-e^{it_2D}\fy\|_2<\e_2.}
Let $-T=t_1<t_2<\cdots<t_K=T$ be a finite sequence of time such that 
\EQ{
 \ka/2 \le |t_k-t_{k+1}| \le \ka \pq(k=1\etc K-1)}
for some $K=K(P)\in\N$. 

Requiring $d(\ta)> 2T$ ensures that the intervals $[\ta(\fy)-T,\ta(\fy)+T]$ around $\ta(\fy)\in \ti P$ are mutually disjoint. 
On the interval 
\EQ{
 \pt I(\fy):=\{t\in\R;|t/\si-\ta(\fy)|\le T\}=\Cu_{k=1}^{K-1}I_k(\fy), 
 \prQ I_k(\fy):=\{t\in\R;|t/\si-\ta(\fy)|\in[t_k,t_{k+1}]\},}  
we approximate the free wave $V(\fy)=e^{i(t/\si-\ta(\fy))D}\fy$ with the static potential $V_k(\fy):=V(\fy)|_{t=t_k}$. 
Since  
\EQ{
 \|V(\fy)-V_k(\fy)\|_{L^\I_t(I_k(\fy);L^2_x)}=\|(e^{itD}-e^{it_kD})\fy\|_{L^\I_t(t_k,t_{k+1};L^2_x)}<\e_2,}
applying the Strichartz estimate of Lemma \ref{lem:L2pert} to 
\EQ{
 (i\p_t-\De-V(\fy))u=f+(V-V(\fy))u,}
on the time interval $I_k(\fy)$, we obtain 
\EQ{ \label{St by step app}
 \pt\|u\|_{\ti \X^\de(I_k(\fy))}
 \le 2C_*\BR{\inf_{t\in I_k(\fy)}\|u(t)\|_2 + \|f+(V-V(\fy))u\|_{\ti \X^\de_*(I_k(\fy))}}.}
The potential term on the right is bounded by 
\EQ{ \label{V stat rem}
 \pt\|(V-V(\fy))u\|_{\X^\de_*(I_k(\fy))}
 \pn\lec \sum_{\psi\in  \ti P\setminus\{\fy\}}\|V(\psi)\|_{L^\I_t(I_k(\fy);\dot B^{-\de/2}_{2(-\de/2),\I})}\|u\|_{\X^\de(I_k(\fy))}, }
using the product estimate \eqref{prod neg Bes}. 
The negative Besov norm is bounded by using the interpolation \eqref{intpl BH}, as well as the scaling invariance, 
\EQ{
 \|V(\psi)\|_{L^\I_t(I_k(\fy);\dot B^{-\de/2}_{2(-\de/2),\I})}
 \lec \|\psi\|_2^{1-\de/2}\|e^{itD}\psi\|_{L^\I_t(\ta(\fy)-\ta(\psi)+[-T,T];\dot B^{-1}_{4,\I})}^{\de/2}.}
This tends to $0$ as $|\ta(\fy)-\ta(\psi)|\to\I$, depending on $\ti P$, but not on $R$. 
Hence making $d(\ta)$ large (depending on $\ti P$) ensures that \eqref{V stat rem} is small enough to be absorbed by the left side of \eqref{St by step app}. 
Thus for every $\fy\in P$, we obtain the uniform Strichartz estimate on each $I_k(\fy)$, and iterating it for $k=1\etc K-1$, that on $I(\fy)$. 

Next consider the complement $I_c:=\R\setminus\Cu_{\fy\in P}I(\fy)$. 
Imposing $d(\ta)>2\ti T> 2T$ ensures that at each $t\in I_c$ we have $|t/\si-\ta(\fy)|>\ti T$ for every $\fy\in\ti P$ except at most one, which can not belong to $P$. 
Hence using \eqref{V scat}, we obtain 
\EQ{
 \|V\|_{L^\I_t(I_c;\dot B^{-1}_{4,2})} 
 \le \sum_{\fy\in\ti P}\|e^{itD}\fy\|_{L^\I_t(|t|>\ti T;\dot B^{-1}_{4,2})} 
 + \sup_{\fy\in\ti P\setminus P}C\|\fy\|_2 < \e_1/2 + C\e\wt,}
where $C$ denotes the constant in the embedding $L^2(\R^4)\subset\dot B^{-1}_{4,2}$. 
Choosing $\e\wt<\e_1/(2C)$, we thus obtain 
\EQ{
 \|V\|_{L^\I_t(I_c;\dot B^{-1}_{4,2})} <\e_1.}
On the other hand, the weak convergence of $e^{itD}\to 0$ as $|t|\to\I$ implies that for $d(\ta)$ large enough, depending on $\ti P$, we have 
\EQ{
 \|V\|_{L^\I_t L^2}^2 = \|V(0)\|_2^2 \le \sum_{\fy\in \ti P}\|\fy\|_2^2 + o(1) < \|W^2\|_2^2=B^2,}
where $o(1)\to 0$ as $d(\ta)\to\I$. 
Then Lemma \ref{lem:Bsmall} implies the $L^2_x$-Strichartz estimate on each connected component of $I_c$ with a uniform constant $C\ds(\de)<\I$. 

Iterating the local Strichartz estimate on the decomposed intervals yields that on the whole $\R$, where the constant is exponentially magnified by the number of $P$. 
Note that if we would apply the same decomposition of $\R$ to all the profiles in $\ti P$, then this magnification of the constant would depend on the number of $\ti P$, which is not desired. 
\end{proof}

\section{Strichartz with profile decomposition} \label{sect:Stz}
In this section, we prove the main Strichartz estimate, namely Theorem \ref{thm:Stz}, using the estimate for a wave train proved in the previous section. 
Our task is twofold: superpose wave trains and shift the regularity from $0$ to $s\in(0,1)$. 

\subsection{Frequency weight adapted to the profiles} \label{ss:AFW}
We introduce a weight function in the frequency adapted to the scale separation between the profiles, since the standard weight $\LR{\x}^s$ would produce big commutator errors with the potential. 
The idea is to use the $L^2_x$-Strichartz for the main interactions with large profiles, so the weight should be flat around frequencies of each profile. 
Another option may be to replace it with $\LR{-\De+V(t)}^s$, but the time derivative would cause loss of regularity. 

For any $\be>1$ and any finite non-empty set $S\subset[1,\I)$ satisfying $1\in S$ and 
\EQ{ \label{sep S}
 \ti S:=\inf\{\max(a/b,b/a)\mid a,b\in S,\ a\not=b\}>\be^4,}
we define the {\it frequency weight $w$ adapted to $(\be,S)$} as a piece-wise logarithmic linear function $w:(0,\I)\to(0,\I)$ by the following. For each $\si\in S$, let 
\EQ{ \label{def w1}
 \si/\be^2 \le r \le \be^2\si \implies w(r):=\si,}
where the intervals of $r$ on the left are mutually disjoint among $\si\in S$ thanks to the above condition \eqref{sep S}. 
$w$ is defined in the rest of $r>0$ to be continuous and logarithmic linear on each interval. To be precise, let  
\EQ{
 S_r:=\sup\{\si\in S\mid \si\le r\}, \pq S^r:=\inf\{\si\in S\mid \si\ge r\},}
with the usual convention $\sup\empt=-\I$ and $\inf\empt=\I$. Let
\EQ{
 \CAS{ S_r=-\I \implies& w(r):=1=\min S, \\ 
  S^r=\I \implies &w(r):=r/\be^2,} }
and otherwise, namely on the intermediate intervals, let
\EQ{ \label{def w2}
 \be^2 S_r \le r \le S^r/\be^2 \implies p(r):=\frac{\log \be^2}{\log\sqrt{\frac{S^r}{\be^4 S_r}}}, \pq w(r) := r\Br{\frac{r}{\sqrt{S^r S_r}}}^{p(r)}.}
Then $w:(0,\I)\to[1,\I)$ is non-decreasing, continuous, and 
\EQ{ \label{w equiv}
 \CAS{0<r\le 1 \implies &w(r)=1,\\ r>1 \implies & r/\be^2 \le w \le \be^2 r.}}
Moreover, in the limit $\ti S\to\I$, we have $p(r)\to 0$ uniformly. 
The following is the essential properties of $w$ used in the estimates. 
\begin{lem} \label{lem:w}
For any $0<s<s'<\I$ and $\be\in(1,\I)$, there exists $L\in(1,\I)$ such that for any $S\subset(1,\I)$ satisfying \eqref{sep S} with $\ti S\ge L$, the frequency weight $w$ adapted to $(\be,S)$ satisfies the following.  $r^{s'}w(r)^{-s}$ is strictly increasing for $r>0$, or in other words, 
\EQ{ \label{w comp}
 \CAS{0<r<r' \implies [w(r)/w(r')]^s > [r/r']^{s'},\\ 0<r'<r \implies [w(r)/w(r')]^s < [r/r']^{s'}.}}
Moreover, for any bounded interval $I\subset(0,\I)$, 
\EQ{ \label{w sum}
 \sum_{r\in 2^\Z\cap I} r^{s'}w(r)^{-s} \lec \frac{1}{s'-s} \sup_{r\in 2^\Z\cap I} r^{s'}w(r)^{-s}.}
\end{lem}
\begin{proof}
The strict increasing is obvious on the intervals where $w(r)$ is constant, as well as for large $r$ where $w(r)=r/\be^2$. On each intermediate interval $J$,  we have $w(r)=Cr^\al$ with some constants $C,\al$ depending on $J$, with 
\EQ{
 \al = 1+p(r) \le 1+ \frac{\log \be^2}{\log\sqrt{L/\be^4}}  \to 1\pq (L\to \I).}
Hence if $L$ is large enough, then $\al s<s'$ and so $r^{s'}w^{-s}$ is decreasing.  
The power form also implies that 
\EQ{
 \int_J r^{s'}w(r)^{-s}\frac{dr}r=\BR{\frac{r^{s'}w(r)^{-s}}{s'-\al s}}_{\p J}
 \le \frac{2}{s'-s}[r^{s'}w(r)^{-s}]_{\p J}.}
Summing this estimate over subintervals where the integrand takes the power form, we obtain, for any bounded interval $I$, 
\EQ{
 \int_I r^{s'}w(r)^{-s}\frac{dr}{r} \le \frac{2}{s'-s}\sup_{r\in I}r^{s'}w(r)^{-s}.}
Dyadic discretization of $\log r$ yields \eqref{w sum}. 
\end{proof}

\subsection{Weighted product estimates}
Here we prepare some product estimates with the weight $w$. The Fourier multiplier associated with the weight $w$ is denoted by
\EQ{ \label{def ws}
 \U w^su:=\sum_{j\in 2^\Z}w(j)^s u_j.}
Recall that $\ti S$ is defined in \eqref{sep S}. 
\begin{lem}
For $s\in[0,1)$, there is $C(s)>1$ such that if $\ti S\ge C(s)$ then we have the following product estimates. 
\EQ{ \label{proest3}
 \|\U w^s (fg)\|_{\dot H^{-1}}
 \le C(s) \|f\|_2\|\U w^sg\|_4.}
For all $a,b\ge 0$ satisfying $a+b\le 3$,  
\EQ{ \label{proest3.5}
 \|\U w^s(fg)\|_{\dot B^0_{4(a+b),2}} \le C(s) \|f\|_{\dot B^1_{4(a),1}} \|\U w^s g\|_{\dot B^0_{4(b),2}}. }
\end{lem}

\begin{proof} 
Consider the Littlewood-Paley decomposition for $j,k,l\in 2^\Z$
\EQ{
 fg=\sum_{j,k,l}(f_jg_k)_l \pq m:=\min(j,k,l), \pq M:=\max(j,k,l),}
where the summation is restricted to $j\lec k\sim l$, $k\lec l\sim j$ and $l\lec j\sim k$. 

We start with the simpler \eqref{proest3.5}, for which by H\"older and Sobolev we have 
\EQ{
 \|(f_jg_k)_l\|_{4(a+b)} \lec m\|f_j\|_{4(a)}\|g_k\|_{4(b)}.}

Fix $l$. For the low-high interactions, we have 
\EQ{
 w(l)^s \sum_{j\lec k\sim l} \|(f_jg_k)_l\|_{4(a+b)} \pt\lec  \sum_{j\lec k\sim l} j\|f_j\|_{4(a)}w(k)^s\| g_k\|_{4(b)}
 \pr\lec \|f\|_{\dot B^1_{4(a),1}} \|\U w^s g_{\sim j}\|_{4(b)}.}
For the high-low interactions we have, if $\ti S$ is large enough,  
\EQ{
 w(l)^s \sum_{k\lec j\sim l} \|(f_jg_k)_l\|_{4(a+b)} \pt\lec w(l)^s\sum_{k\lec j\sim l} kw(k)^{-s} \|f_j\|_{4(a)} \|\U w^s g_k\|_{4(b)}
 \pr\lec \sum_{j\sim l}j\|f_j\|_{4(a)}\|\U w^s g\|_{\dot B^0_{4(b),\I}}, }
using \eqref{w sum} and $s<1$. For the high-high interactions, we have 
\EQ{
 w(l)^s \sum_{j\sim k\gec l}\|(f_jg_k)_l\|_{4(a+b)} 
 \pt\lec \sum_{j\sim k\gec l}lw(l)^s\|f_j\|_{4(a)}\|g_k\|_{4(b)}
 \pr\lec \sum_{j\gec l}(l/j)\|f_j\|_{\dot B^1_{4(a)}} \|\U w^s g\|_{\dot B^0_{4(b),\I}}, }
using the monotonicity of $w$. 
Square summation over $l$, using Young's inequality for the convolution over $2^\Z$, leads to \eqref{proest3.5}.

For \eqref{proest3}, the low-high and the high-low interactions are similarly estimated by H\"older and Sobolev, 
\EQ{
 \|\U w^s(f_jg_k)_l\|_{\dot H^{-1}} \lec (w(l)/w(k))^s l^{-1}m\|f_j\|_2\|\U w^s g_k\|_4,}
where the coefficient in $j,k,l$ is bounded by 
\EQ{
 \CAS{j\lec k\sim l \implies (w(l)/w(k))^s l^{-1}m \sim j/l,\\
 k\lec j\sim l \implies (w(l)/w(k))^s l^{-1}m \sim \frac{kw(k)^{-s}}{lw(l)^{-s}}.} }
Hence the desired \eqref{proest3} in these regions follows from Young for the convolution over $2^\Z$, using  
\eqref{w sum} and $s<1$ in the second case, assuming that $\ti S$ is large enough. 

In the high-high case $l\lec j\sim k$, we need the Littlewood-Paley theory, or the Triebel-Lizorkin space with the embedding $\dot H^{-1}\supset \dot F^0_{4/3,\I}$. Using the monotonicity of $w$ as well, we have 
\EQ{
 \|\sum_l \sum_{j\sim k\gec l}w(l)^s(f_j g_k)_l\|_{\dot H^{-1}}
 \pt\lec \|\sum_{j\sim k}w(k)^s|f_j g_k|\|_{4/3} \pr\lec \|f\|_{\dot F^0_{2,2}} \|\U w^s g\|_{\dot F^0_{4,2}} \sim \|f\|_2\|\U w^sg\|_4.}
\end{proof}

For the normal form, we have 
\begin{lem}
Let $0\le s<\te_2$. There is $C(s,\te_2)>1$ such that if $\ti S\ge C(s,\te_2)$ then for all $\de_1,\de_2,\te_1\in\R$ satisfying 
\EQ{
 \de_1\ge -2, \pq \de_2\ge 0, \pq \de_1+\de_2\le 2,}
and for all $(j,k)\in \HL_\io$ with $0<\io\ll 1$, we have 
\EQ{ \label{proest4}
 \|\U w^s\Om_\io^\pm(f_j,g_k)\|_{\dot B^{\te_1-\te_2+2}_{2(\de_1+\de_2)}}
 \le C(s,\te_2)(k/j)^{(\te_2-s)/2} \|f_j\|_{\dot B^{\te_1}_{2(\de_1)}} \|\U w^sg_k\|_{\dot B^{-\te_2}_{\I(\de_2)}}.}
\end{lem}
\begin{proof}
By H\"older, we have 
\EQ{ 
 \|\U w^s\Om_\io^\pm(f_j,g_k)\|_{\dot B^{\te_1-\te_2+2}_{2(\de_1+\de_2)}}
 \pt\lec w(j)^s j^{-\te_2} \|f_j\|_{\dot B^{\te_1}_{2(\de_1)}} k^{\te_2} w(k)^{-s}\|\U w^s g_k\|_{\dot B^{-\te_2}_{\I(\de_2)}}
 \pr\lec (k/j)^{(\te_2-s)/2} \|f_j\|_{\dot B^{\te_1}_{2(\de_1)}} \|\U w^sg_k\|_{\dot B^{-\te_2}_{\I(\de_2)}},}
where we used \eqref{w comp} in the form $w(j)^s/j^{(s+\te_2)/2} \le w(k)^s/k^{(s+\te_2)/2}$, assuming that $\ti S$ is large enough. 
\end{proof}

For the remainder of the normal form, we have 
\begin{lem}
For any $s\in[0,1)$ and $\de\in(0,1]$, there exists $C(s), C(s,\de)>1$ such that if $\ti S\ge C(s)$ then we have
\EQ{ \label{proest5}
 \|\U w^s (f,g)_{\LH_\io}\|_{\dot B^{-\de}_{4/3(-\de),2}} \le \io^{-1}C(s,\de)\|f\|_{\dot B^{-\de/2}_{2(-\de/2),\I}}\|\U w^s g\|_{\dot B^{-\de}_{4(-\de),2}}.}
\end{lem}
\begin{proof}
In the decomposition $\U w^s(f,g)_{\LH_\io}=\sum_l \sum_{(j,k)\in\LH_\io} w(l)^s(f_j,g_k)_l$, the restriction $(j,k)\in\LH_\io$ implies $l\lec j+k \lec \io^{-1}\max(k,1)$. 
If $l\le\max(k,1)$, then $w(l)\le w(k)$. Otherwise \eqref{w comp} yields, for large $\ti S$,  
\EQ{
 \frac{w(l)^s}{w(k)^s} = \frac{w(l)^s}{w(\max(k,1))^s} \le \frac{l}{\max(k,1)} \lec \io^{-1}.}
Hence we have $w(l)^s \lec \io^{-1} w(k)^s$ in both the cases. 
After replacing $w(l)^s$ with $\io^{-1}w(k)^s$, the rest of proof is the same as for \eqref{prod neg Bes}. 
\end{proof}

\subsection{Start of the proof}
Let $s,\de,B,M$, $V_n$ and $u_n$ be as in Section \ref{sect:prof dec}, and consider the profile decomposition in \eqref{decop Vn}.  
The Duhamel part of the potential is regarded as a part of the remainder, denoted by
\EQ{ \label{def GaJn}
 \Ga^J_n:=\Ga_n^{1,J}+V\sD_{n,J}.}
By \eqref{choice Vn fn} and \eqref{Ga small0},  for any $\e\in(0,1)$ there exists $J\in\N$ such that 
\EQ{ \label{est Ga2}
 \|\Ga^J_n\|_{L^\I_t\dot B^{-1}_{4,\I}}+\|V\sN_{n,J}\|_{L^\I_t \dot B^1_{4/3,1}+L^1_t\dot B^1_{4,1}} \le \e}
for large $n$. 

We have introduced several parameters, and we should pay some attention to their mutual dependence or the order to choose the parameters. 
The parameters $s,\de,B,M,$ are fixed as given at the beginning. 
The limit $n\to\I$, or sufficiently large $n$, should be chosen in the end. 
We have the other parameters $J,\e,\be,\io$ in between. 
They are decided in the order
\EQ{
 (s,\de,B,M) \to \io \to \e \to J \to \be \to n,}
which means that each parameter can depend on all those on their left, but none of those on their right. 
The constants in the following estimates should not implicitly depend on these parameters except for $(s,\de)$. 

By construction, the profiles are uniformly bounded by 
\EQ{
 \sum_{0\le j<J} \|\psi^j\|_2^2 \le \liminf_{n\to\I}\|V\sF_{n,J}(0)\|_2^2 \le B^2 <\|W^2\|_2^2.}
Hence there exists $J^0\in\N$, independent of the parameters $(\io,\e,J,\be,n)$, such that 
\EQ{ \label{choice J0}
 J^0 \le \|W^2\|_2^2/\e\wt^2, \pq \sup_{j\ge J^0}\|\psi^j\|_2 \le \e\wt,}
where $\e\wt\in(0,1)$ is the constant given by Lemma \ref{lem:St-wt}. 
Let 
\EQ{ \label{def C1}
 C_1:=C\wt(\de,\{\psi^j\}_{0\le j<J^0}) \in (0,\I),}
which is independent of $(\io,\e,J,\be,n)$, 
so that Lemma \ref{lem:St-wt} yields a uniform bound on the $L^2_x$-Strichartz for the concentrating wave trains in the profile decomposition: 
\EQ{ \label{unif wt bd0}
 \sup_J \sup_{A\in J_\sim} \limsup_{n\to\I} C_0^\de(V^A_n) \le C_1.}
Lemma \ref{lem:L2pert} yields an admissible size of $L^2_x$ perturbation, denoted by 
\EQ{ \label{def e0}
 \e_0:=\e\pert(C_1)\in(0,1),}
for which the above uniform bound is essentially preserved. 

After choosing a finite $J$, we can modify each profile $\psi^j$ for $j<J$ 
to have Fourier support uniformly bounded and away from $0$, by smooth cut-off putting the removed parts into the remainder $\Ga^J_n$. 
For distinction, the original profiles are denoted by $\U\psi^j$ and the new ones by $\psi^j$. 
Choosing $\be>10$ large enough, depending on the original profiles $\{\U\psi^j\}_{j<J}$, $\e$, and $\e_0$ in \eqref{def e0}, we may thus assume 
\EQ{ \label{modif psia}
 \Cu_{j<J}\supp\hat\psi^j \subset \{\be^{-1}<|\x|<\be\}, \pq \sum_{j<J}\|\U\psi^j-\psi^j\|_2<\e_0,}
loosing the smallness of $\Ga^J_n$ from \eqref{est Ga2} into 
\EQ{ \label{bd' Ga}
 \|\Ga^J_n\|_{L^\I_t\dot B^{-1}_{4,\I}}+\|V\sN_{n,J}\|_{L^\I_t\dot B^1_{4/3,1}+L^1_t\dot B^1_{4,1}} \le 2\e,}
for large $n$. Note that this modification does not change $V\sN_{n,J},V\sD_{n,J}$.  
 By the choice of $\e_0$ in \eqref{def e0} and the right estimate in \eqref{modif psia}, Lemma \ref{lem:L2pert} implies that \eqref{unif wt bd0} is modified into 
\EQ{ \label{unif wt bd}
 \sup_{A\in J_\sim} \limsup_{n\to\I} C_0^\de(V^A_n) \le 2C_1,}
for the fixed $J$. 

Now let $w_n$ be the frequency weight adapted to $(\be,\{\si^A_n\}_{A\in J_\sim})$, defined in Section \ref{ss:AFW}. We may assume that $\si^A_n=1$ for some $A\in J_\sim$. 
The scale separation \eqref{sep R} implies the frequency separation 
\EQ{
 \ti S_n:=\inf\{\max(\si^A_n/\si^B_n,\si^B_n/\si^A_n)\mid A,B\in J_\sim,\ A\not=B\} \to \I\pq(n\to\I),} 
so that we can use the monotonicity lemma \ref{lem:w} for large $n$. 
Let $u_n,f_n,g^J_n$ satisfy 
\EQ{
 \pt (i\p_t - \De - \re V_n)u_n =: f_n, 
 \pr (\p_t-iD)\Ga^J_n=(\p_t-iD)V\sD_{n,J} =: g^J_n.}
Then after a normalization, we have from \eqref{choice Vn fn},
\EQ{
 \pt \|\LR{D}^su_n\|_{\ti \X^\de} \to\I,
 \pq \|\LR{D}^su_n(0)\|_2+\|\LR{D}^sf_n\|_{\ti \X^\de_*} \le 1. }
For brevity, the real part of each potential is denoted by 
\EQ{ \label{def vga}
 v_*^*:=\re V_*^*, \pq \ga_*^*:=\re \Ga_*^*}
with any index $*$. 

Associated with the above weight $w_n$, we introduce a frequency decomposition. 
For each $A\in J_\sim$, let 
\EQ{ \label{freq proj}
 \pt [A]:=\{r\in 2^\Z \mid w_n(r)=\si^A_n\}=2^\Z\cap[\si^A_n/\be^2,\si^A_n\be^2],
 \pr ]A[\ :=2^\Z \cap[\si^A_n/\be,\si^A_n \be] \subset [A],
 \pr \vec A:=\{r\in 2^\Z \mid A<\forall B\in J_\sim,\ \si^A_n<w_n(r)<\si^B_n\}.}
$[A]$ is the dyadic frequency band where the weight $w_n$ is flat. 
$]A[$ is a narrower band in which the profile $\F V^A$ is supported.  
$\vec A$ is the intermediate band next to and bigger than $[A]$. 
Note that all the above sets depend on $n$. 
 Some additional notation for any finite subset of $S\subset 2^\Z$: 
\EQ{ \label{freq befaf}
 \pt S_-:=2^{-1}\min S, \pq S_+:=2\max S, \pq S_\pm:=S\cup\{S_-,S_+\},
 \pr \tle S:=\{k\in 2^\Z \mid k\le S_-\},
 \pq \tge S:=\{k\in 2^\Z \mid k\ge S_+\},}

Henceforth the subscript $n$ is often omitted for simplicity of notation. 

\subsection{Estimate around the profile frequencies}
For each $A\in J_\sim$, consider the frequency projection of the equation around $\si^A$ defined in \eqref{freq proj}. Then we obtain the equation 
\EQ{
 (i\p_t-\De-v^A)u_{[A]}=(f+vu)_{[A]}-v^{A}u_{[A]}.}
Recall that $v^A$ denotes the real part of the concentrating wave train
\EQ{
 v^{A}_n=\re S(\si^A_n)\sum_{j\in A}e^{i(\si^A_n t-\ta^j_n)D}\psi^j, \pq \supp \F v^A_n \subset[\si^A_n/\be,\si^A_n\be], }
while $u_{[A]}=P_{[A]}u_n$ is the projection of the wider width $\be^2$ around its frequency $\si^A_n$. 
The frequency separation $\ti S_n\to\I$ implies that for large $n$ 
\EQ{
 \pn(vu)_{[A]}-v^Au_{[A]}
 \pt=(v^Au)_{[A]}-v^Au_{[A]}+\sum_{a<A}(v^a u_{[A]_\pm})_{[A]}+\sum_{a>A}(v^au_{[a]})_{[A]}
 \prQ+(\ga^J u+v\sN_J u)_{[A]},}
where the summation is over $a\in J_\sim$ with the indicated frequency restriction. 
The first two terms on the right is further expanded into 
\EQ{
 (v^Au)_{[A]}-v^Au_{[A]}
 =v^Au_{\tle[A]}+(v^Au_{[A]_+})_{[A]}-(v^Au_{[A]})_{[A]_+}-(v^Au_{[A]})_{\tle[A]},}
where the last three terms are junk. Put
\EQ{ \label{def fA}
 f^A:=f_{[A]} \pt+\sum_{a<A}(v^au_{[A]_\pm})_{[A]}+\sum_{a>A}(v^{a}u_{[a]})_{[A]}
 \pr+(v^Au_{[A]_+})_{[A]}-(v^Au_{[A]})_{[A]_+}-(v^Au_{[A]})_{\tle[A]}}
so that it contains the low-high and high-high interactions, besides the source $f_{[A]}$, and that we have
\EQ{
 (f+vu)_{[A]}-v^Au_{[A]}=f^A + (v\sN_J u)_{[A]} + v^Au_{\tle[A]} + (\ga^J u)_{[A]}. }
For the last two terms, we need the normal form introduced in Section \ref{ss:NF}. 
Thus we obtain the following equation
\EQ{ \label{eq at k}
 \pt(i\p_t-\De-v^A)(u_{[A]}-u^{\LR{A}}) 
 \prq= f^A+ (v\sN_J u)_{[A]} +P_{[A]}(\ga^J,u)_{\LH_\io}-f^{\LR{A}}-i\Om_\io(g^J,u)_{[A]}+v^Au^{\LR{A}}, }
where
\EQ{ \label{def ufLR}
 \pt u^{\LR{A}}:=\Om_\io(v^A,u_{\tle[A]})+\Om_\io(\ga^J,u)_{[A]}, 
 \pr f^{\LR{A}}:=\Om_\io(v^A,(f+vu)_{\tle[A]})+\Om_\io(\ga^J,f+vu)_{[A]}.}
The uniform $L^2_x$-Strichartz bound \eqref{unif wt bd} implies that 
\EQ{
 \pt\|\U w^s(u_{[A]}-u^{\LR{A}})\|_{\ti \X^\de}
  = w_A^s\|u_{[A]}-u^{\LR{A}}\|_{\ti \X^\de}
 \prQ\le 2C_1w_A^s\|(u_{[A]}-u^{\LR{A}})|_{t=0}\|_2 + 2C_1w_A^s\|\text{RHS of \eqref{eq at k}}\|_{\ti \X^\de_*},}
where the weight around the profile frequency is abbreviated by 
\EQ{ \label{def wA}
 w_A := w(\si^A_n).}
In the following estimates, $B$ will be absorbed by $M$, as we are assuming $1+B\ll M$. 

The part coming from the normal form is estimated by \eqref{proest4} with $\de_1=\de_2=\te_1=0$, $\te_2=2-\te$ for $\te=0,1$ and $p\ge 2$. 
Taking $\be$ large enough, depending on $s$, ensures the condition on $\ti S$. Thus we obtain  
\EQ{
 \pt \squm_A \|\U w^s u^{\LR{A}}\|_{\cL^p_t \dot H^\te}
 \pr\lec \squm_A \squm_{h\in[A]_\pm} \sum_{l\le \io h} (l/h)^{1-(\te+s)/2}\Br{\|v^A_h\|_{L^\I_tL^2}+\|\ga^J_h\|_{L^\I_tL^2}} \|\U w^s u_l\|_{L^p_t\dot B^{\te-2}_{\I}} 
 \pr\lec \sum_{\ka\le\io} \squm_A \squm_{h\in[A]_\pm} \ka^{1-(\te+s)/2}M \|\U w^s u_{\ka h}\|_{L^p_t\dot B^{\te-2}_{\I}} 
 \lec \io^{1-(\te+s)/2}M\|\U w^s u\|_{\cL^p_t\dot B^{\te-2}_{\I,2}}, }
where $\ka:=l/h\in 2^\Z$, and $\ell^1_\ka \ell^2_{A,h} \subset \ell^2_{A,h} \ell^1_\ka$ is used in the second inequality. 
Choosing $(p,\te)=(\I,0),(2,1)$ and using the Sobolev embeddings $\dot H^1\subset\dot B^0_{4,2}\subset \dot B^{-1}_{\I,2}$, $L^2\subset\dot B^{-2}_{\I,2}$, we obtain 
\EQ{ \label{est u^k}
 \squm_A \|\U w^s u^{\LR{A}}\|_{\ti \X^\de} \lec \io^{(1-s)/2}M\|\U w^s u\|_{\ti \X^0}.}  
Then its contribution to the right side of the equation is estimated by 
\EQ{
 \squm_A w_A^s \|v^Au^{\LR{A}}\|_{L^2_t L^{4/3}}
 \pt\lec \squm_A  \|v^A\|_{L^\I_t L^2} w_A^s \|u^{\LR{A}}\|_{L^2_tL^4}
 \pr\lec B M\io^{(1-s)/2}\|\U w^s u\|_{\ti \X^0},}
which is negligible by choosing $\io$ small enough, depending on $s,B,M,C_1$. 

Similarly, the normal form part on the right side is estimated by \eqref{proest4} with $(\de_1,\de_2,\te_1,\te_2)=(0,0,0,2),(0,1,0,2)$, for large $\be$, 
\EQ{ 
 \pt\squm_A \|\U w^s f^{\LR{A}}\|_{\ti \X_*^0}
 \pn\lec \squm_A \squm_{h\in[A]_\pm} \sum_{l\le\io h}(l/h)^{1-s/2}M\|\U w^s(f+vu)_l\|_{L^1_t \dot B^{-2}_{\I}+L^2_t \dot B^{-2}_4},}
then using $\dot B^{-1}_{4/3(-1),2}=\dot H^{-1}\subset\dot B^{-2}_{4,2}$, together with \eqref{proest3}, as well as the Minkowski inequality as above,  
\EQ{ \label{est f^A}
 \squm_A \|\U w^s f^{\LR{A}}\|_{\ti \X_*^0}\pt\lec M \io^{1-s/2} \BR{\|\U w^s f\|_{\ti \X_*^0} + \|\U w^s(vu)\|_{L^2_t \dot H^{-1}}} 
 \pr\lec M \io^{1-s/2} \|\U w^s f\|_{\ti \X^\de_*} + M^2 \io^{1-s/2}\|\U w^s u\|_{\X^0}. }
Similarly we obtain, using \eqref{proest4} with $(\de_1,\de_2,\te_1,\te_2)=(-1,1,-1,1),(1,0,-1,1)$ 
\EQ{ \label{est Omgu}
 \pt\squm_A \|\U w^s \Om_\io(g^J,u)_{[A]}\|_{\ti \X_*^0}
 \pr\lec \squm_A \squm_{h\in[A]_\pm} \sum_{l\le\io h}(l/h)^{(1-s)/2} \|g^J_h\|_{L^1_t\dot B^{-1}_4+L^\I_t\dot B_{4/3}^{-1}} \|\U w^s u_l\|_{L^\I_t \dot B^{-1}_4 \cap L^2_t \dot B^{-1}_\I} 
 \pr \lec \io^{(1-s)/2} M \|\U w^s u\|_{\ti \X^0},} 
cf.~the above argument for \eqref{est u^k} for the square summation. 

Next, the low-high interactions with the profile is estimated, exploiting the radial improvement of the Strichartz estimate.  By H\"older, we have 
\EQ{ \label{est vu-LH}
 \pt \squm_A \sum_{a<A}\|\U w^s(v^a u_{[A]_\pm})_{[A]}\|_{\X^\de_*}
 \pr\lec  \squm_A \sum_{a<A} \squm_{k\in[A]_\pm} \sum_{j\in]a[} (j/k)^\de w(k)^s \|j^{-\de}v^a_j\|_{L^\I_t L^{2(-\de)}} \|u_k\|_{L^2_t L^4} 
 \pr\lec  \squm_A  \sum_{a<A} \frac{(\si^a \be)^\de}{(\si^A/\be^2)^\de}  \squm_{k\in[A]_\pm}w_A^s\|v^a\|_{L^\I_t\dot B^{-\de}_{2(-\de),\I}} \|u_k\|_{\X^0}
 \lec \be^{-\de}B\|\U w^s u\|_{\X^0}.
}
This is negligible by choosing $\be$ large enough, depending on $\de,B,C_1$. 

Similarly, the high-high interactions are estimated by H\"older and Sobolev:
\EQ{
 \|(f_jg_k)_l\|_{4/3} \lec l^{\de/2} \|(f_jg_k)_l\|_{4/3(\de/2)}  \lec l^\de \|f_j\|_{L^{2(-\de/2)}}\|g_k\|_{L^{4(\de)}},}
thus we obtain 
\EQ{ \label{est vu-HH}
 \pt\squm_A \sum_{a>A}\|\U w^s(v^a u_{[a]})_{[A]}\|_{\X_*^0}
 \pr\lec \squm_A \sum_{a>A} \squm_{l\in[A]} \sum_{j\in]a[} \sum_{k\in\{j\}_\pm} \frac{w(l)^s l^{\de/2}}{w(k)^s k^{\de/2}}\|j^{-\de/2}v^a_j\|_{L^\I_t L^{2(-\de/2)}} \|k^\de \U w^su_k\|_{L^2_t L^{4(\de)}}
 \pr\lec \squm_A \sum_{a>A}  \frac{(\si^A \be^2)^{\de/2}}{(\si^a/\be)^{\de/2}} \|v^a\|_{L^\I_t \dot B^{-\de/2}_{2(-\de/2),\I}}\|\U w^su_{[a]}\|_{\X^\de} \lec \be^{-\de/2} B \|\U w^s u\|_{\X^\de}, }
where the monotonicity of $w$ is also used.  

For the remaining terms with $v^A$, namely 
\EQ{ 
 (v^A u_{[A]_+})_{[A]}-(v^Au_{[A]})_{[A]_+}-(v^Au_{[A]})_{\tle[A]},}
the Fourier support of $v^A$ and the definition of $[A]_+$, $\tle[A]$ implies that the first two terms are low-high interactions with $\be$ gap, and the last term is high-high interactions with $\be$ gap. 
Hence the same argument as above yields
\EQ{
 \pt\squm_A \|\U w^s[(v^A u_{[A]_+})_{[A]}-(v^Au_{[A]})_{[A]_+}]\|_{\X^\de_*} 
 \lec \be^{-\de}B\|\U w^s u\|_{\X^0}, 
 \pr\squm_A \|\U w^s(v^Au_{[A]})_{\tle[A]}\|_{\X_*^0} \lec \be^{-\de/2}B\|\U w^s u\|_{\X^\de}.} 

The remainder term is estimated by \eqref{proest5}, for large $\be$ depending on $s$, 
\EQ{ \label{est gau1}
 \squm_A \|\U w^s (\ga^J, u)_{\LH_\io}\|_{\X^\de_*} \pt\lec \io^{-1}\|\ga^J\|_{L^\I_t \dot B^{-\de/2}_{2(-\de/2),\I}} \|\U w^su\|_{\X^\de}
 \pr\lec \io^{-1}M^{1-\de/4}\e^{\de/4}\|\U w^s u\|_{\X^\de},}
where in the second inequality the interpolation inequality \eqref{intpl BH} was used, together with the smallness \eqref{bd' Ga}.  
Choosing $\e$ small enough, depending on $\de,M,C_1$ and $\io$, this term is also negligible. 

For $V\sN_J$, we use \eqref{proest3.5} with $s\in[0,1)$ and $(a,b)=(0,1),(2,0)$. Hence for large $\be$, 
\EQ{
 \pt \|\U w^s(v\sN_J u)\|_2 \lec \|v\sN_J\|_{\dot B^1_{4,1}}\|\U w^s u\|_2, 
 \pr \|\U w^s(v\sN_J u)\|_{\dot B^0_{4/3,2}} \lec \|v\sN_J\|_{\dot B^1_{4/3,1}}\|\U w^s u\|_{\dot B^0_{4,2}},}
thereby we obtain, using the standard $L^p_t$ rather than $\cL^p_t$, 
\EQ{ \label{est vNu}
 \|\U w^s(v\sN_J u)\|_{\ti \X_*^0} \pt\lec \|\U w^s(v\sN_J u)\|_{L^1_tL^2+L^2_tL^{4/3}} \pr\lec \|v\sN_J\|_{L^\I_t\dot B^1_{4/3,1}+L^1_t\dot B^1_{4,1}}\|\U w^su\|_{L^\I_tL^2 \cap L^2_t \dot B^0_{4,2}} \lec \e\|\U w^s u\|_{\ti \X^0}.}

Gathering all the above estimates, we obtain 
\EQ{ \label{est around prof}
 \squm_A \|\U w^s u_{[A]}\|_{\ti \X^\de} \pt\lec C_1\|\U w^s u(0)\|_2+ C_1(1+\io^{1-s/2}M) \|\U w^s f\|_{\ti \X^\de_*}
 \pn+ C_1 c_1\|\U w^s u\|_{\ti \X^\de},}
where
\EQ{ \label{def c1}
 c_1:= \io^{(1-s)/2}M^2+\be^{-\de/2}B+\io^{-1}\e^{\de/4} M^{1-\de/4}+\e.}

\subsection{Estimate for intermediate frequencies}
For the intermediate or large frequencies, we use the free Strichartz estimate, regarding the potential term as small perturbation. 

For each $A\in J_\sim$, consider the frequency projection of the equation next to $[A]$. Then we obtain 
\EQ{ \label{eq vecA}
 \pt(i\p_t-\De)u_{\vec A}=(f+vu)_{\vec A}=f^{\vec A}+((\ga^J+v\sN_J) u)_{\vec A},
 \prq f^{\vec A}:=f_{\vec A}+\sum_{a\le\si}(v^au)_{\vec A}+\sum_{a>A}(v^au)_{\vec A},}
where the potential terms with $a\le A$ are low-high interactions, and those with $a>A$ are high-high interactions, for large $\be$. 
The term with $\ga$ needs the normal form as before. 
Then the above equation is transformed into 
\EQ{
 \pt (i\p_t-\De)(u-\Om_\io(\ga^J,u))_{\vec A}
 \prq=f^{\vec A}+v\sN_J u+P_{\vec A}(\ga^J,u)_{\LH_\io}-\Om_\io(\ga^J,f+vu)_{\vec A}-i\Om_\io(g^J,u)_{\vec A}.} 
The free Strichartz \eqref{RFD} implies that 
\EQ{
 \pt\|\U w^s(u-\Om_\io(\ga^J,u))_{\vec A}\|_{\ti \X^\de}
 \prq\lec \|\U w^s(f^{\vec A}+(v\sN_J u)_{\vec A}+(\ga^J,u)_{\LH_\io}-\Om_\io(\ga^J,f+vu)-i\Om_\io(g^J,u))_{\vec A}\|_{\ti \X^\de_*}.}

The normal form part on the left is estimated in the same way as \eqref{est u^k},
\EQ{
 \squm_A \|\U w^s\Om_\io(\ga^J,u)_{\vec A}\|_{\ti \X^\de}
 \pt\lec \squm_A \squm_{h\in\vec A_\pm}\sum_{l\le\io h}(l/h)^{(1-s)/2}\|\ga_h\|_{L^\I_tL^2}\|\U w^s u_l\|_{\ti \X^0}
 \pr\lec \io^{(1-s)/2} M \|\U w^s u\|_{\ti \X^0}.}
The normal form part on the right is estimated in the same way as \eqref{est f^A}--\eqref{est Omgu},
\EQ{
 \pt\squm_A \|\U w^s\Om_\io(\ga^J,f+vu)_{\vec A}\|_{\ti \X_*^0}
  \pn \lec M \io^{1-s/2}\|\U w^s f\|_{\ti \X^\de_*} + M^2\io^{1-s/2}\|\U w^s u\|_{\X^0}, 
 \pr \squm_A \|\U w^s\Om_\io(g,u)_{\vec A}\|_{\ti \X_*^0}
  \lec \io^{(1-s)/2}M\|\U w^s u\|_{\ti \X^0}. }

The low-high profile interactions are estimated in the same way as \eqref{est vu-LH},
\EQ{
 \pn\squm_A \sum_{\si^a\le\si^A} \|\U w^s(v^au)_{\vec A}\|_{\X^\de_*}
 \pt\lec \squm_A \sum_{\si^a\le\si^A} \squm_{k\in\vec A_\pm} \sum_{j\in]a[} \frac{j^\de}{k^\de}  \|v^a_j\|_{L^\I_t \dot B^{-\de}_{2(-\de)}} \|\U w^su_k\|_{L^2_t \dot B^0_4}
 \pr\lec \squm_A \sum_{\si^a\le\si^A}\frac{(\si^a\be)^\de}{(\si^A\be^2)^\de} B \|\U w^s u_{\vec A_\pm}\|_{\X^0}
 \lec \be^{-\de}B\|\U w^su\|_{\X^0}.
}
The high-high profile interactions are estimated in the same way as \eqref{est vu-HH},
\EQ{
 \squm_A \sum_{a>A} \|\U w^s(v^a u)_{\vec A}\|_{\X_*^0}
 \pt\lec \squm_A \sum_{a>A} \frac{(\max \vec A)^{\de/2}}{(\min]a[)^{\de/2}} B \|\U w^s u_{[a]}\|_{\X^\de}
 \pr\lec \be^{-\de/2}B\|\U w^s u\|_{\X^\de}.}

The remaining interactions with $\ga^J$ and $v\sN_J$ were already estimated in \eqref{est gau1} and \eqref{est vNu}. Gathering all the above estimates including \eqref{est around prof}, we obtain 
\EQ{
 \pt \|\U w^s u\|_{\ti \X^\de}
 \lec C_1\|\U w^s u(0)\|_2 + C_1(1+\io^{1-s/2}M) \|\U w^s f\|_{\ti \X^\de_*}
 \pn + C_1 c_1\|\U w^s u\|_{\ti \X^\de},}
where $c_1$ is defined in \eqref{def c1}. 
Hence choosing $\io,\e,\be$ in this order, depending on $s,\de,B,M,C_1$, we can make 
\EQ{
 \io^{1-s/2}M \ll 1, \pq C_1c_1 \ll 1, \pq \io^{-1} \ll \be.}
Then the above implies 
\EQ{
 \|\U w^s u\|_{\ti \X^\de} \lec C_1[\|\U w^s u(0)\|_2 +  \|\U w^s f\|_{\ti \X^\de_*}]}
and, together with \eqref{w equiv},  
\EQ{
 \|\LR{D}^su_n\|_{\ti \X^\de} \lec \be^{2s}C_1[ \|u_n(0)\|_{H^s} + \|\LR{D}^sf_n\|_{\ti \X^\de_*}],}
where the $n$-dependence is exposed. 
However, this uniform estimate is contradicting the starting assumption \eqref{choice Vn fn}, thereby we conclude the proof of Theorem \ref{thm:Stz}. 

Note that the contradiction argument is essential. In other words, the above proof does not yield any explicit estimate on the constant in the final Strichartz estimate, because the above $C_1$ and $\be$ depend on the choice of $\{V_n\}$.

\section{Ground state constraint} \label{sect:var}
In this section, we derive some variational properties of the ground state $(u,N)=(W,W^2)$ for the Zakharov system \eqref{Zak}, together with quantitative estimates. 
In particular, it is shown that the crucial condition $\|N\|_2<\|W^2\|_2$ to use the Strichartz estimate is preserved if the Zakharov energy is below the ground state. 
Let 
\EQ{ \label{def ES}
 \pt E_S(u):=\frac12\|\na u\|_2^2-\frac14\|u\|_4^4}
denote the NLS energy, then the Zakharov energy is rewritten as 
\EQ{
  E_Z(u,N)=E_S(u)+\frac14\|N-|u|^2\|_2^2=\frac12\|\na u\|_2^2+\frac14\|N\|_2^2-\frac12\LR{N||u|^2}.} 
A key functional is denoted by 
\EQ{ \label{def K}
 K(u):=\|\na u\|_2^2-\|u\|_4^4,}
which is linked both to the virial identity and to Nehari's, since the NLS is energy-critical on $\R^4$.  
The following equivalence is well-known and immediate from the Sobolev inequality with the best constant (cf.~\cite{KM}). 
 For any $u\in \dot H^1(\R^4)$ satisfying $E_S(u)<E_S(W)$, we have 
\EQ{ \label{K cond NLS}
 K(u) \ge 0 \iff \|u\|_4 < \|W\|_4 \iff \|\na u\|_2 < \|\na W\|_2,}
where $K(u)=0$ only if $u=0$. It implies 
\EQ{
 E(W)=\inf\{E(u) \mid u\in\dot H^1(\R^4)\setminus\{0\},\ K(u)=0\}.} 
\eqref{K cond NLS} follows from the property of $W$ achieving the Sobolev best constant \eqref{max Sob} and solving the static NLS \eqref{stNLS}: 
\EQ{
 \pt \|W\|_4=C_S\|\na W\|_2, \pq K(W)=\|\na W\|_2^2-\|W\|_4^4=0, 
 \pr 1=C_S^4\|\na W\|_2^2=C_S^2\|W\|_4^2, \pq 4E_S(W)=\|\na W\|_2^2=\|W\|_4^4=C_S^{-4},}
and so 
\EQ{
 \pt K(u)\ge 0 \implies \|\na u\|_2^2 \le 4E_S(u),
 \pr \|\na u\|_2<\|\na W\|_2 \implies \|u\|_4 \le C_S\|\na u\|_2<C_S\|\na W\|_2=\|W\|_4,
 \pr \|u\|_4\le \|W\|_4 \implies K(u) \ge (C_S^{-2}-\|u\|_4^2)\|u\|_4^2 \ge 0.}
Thus we obtain \eqref{K cond NLS}.

For the Zakharov system, since
\EQ{
 4E_S(W)=C_S^{-4}, \pq 
 |\LR{N||u|^2}| \le C_S^2\|N\|_2\|\na u\|_2^2,}
we have
\EQ{ \label{EZ below equiv}
 2(1-C_S^2\|N\|_2)\|\na u\|_2^2 + \|N\|_2^2 \le 4E_Z(u,N) \le 2\|\na u\|_2^2+ \|N\|_2^2,}
where the factor on $\|\na u\|_2^2$ is non-negative if $\|N\|_2\le \|W\|_4^2=C_S^{-2}$. In short,
\EQ{
 \|N\|_2<\|W^2\|_2 \implies E_Z(u,N) \sim \|\na u\|_2^2+\|N\|_2^2.}
Moreover, we have the following equivalent conditions. 
\begin{lem} \label{K cond Zak}
For any $(u,N)\in H^1(\R^4)\times L^2(\R^4)$ under the energy constraint 
\EQ{ \label{EZ cons}
 E_Z(u,N)<E_S(W)=\|\na W\|_2^2/4=\|W\|_4^4/4,}
we have the equivalence  
\EQ{ \label{K pos equiv}
 K(u)\ge 0 \iff \|N\|_2<\|W\|_4^2 \iff \|N\|_2^2 \le 4E_Z(u,N),}
where $K(u)=0$ only if $u=0$, and 
\EQ{ \label{K neg equiv}
 K(u)<0 \iff \|N\|_2>\|W\|_4^2 \iff \|N\|_2^2 \ge 4E_Z(u,N).}
In particular, these conditions are closed and open relatively under the constraint \eqref{EZ cons}, hence preserved by any continuous flow which preserves or decreases $E_Z(u,N)$. 
\end{lem}
\begin{proof}
First, the right condition of \eqref{K pos equiv} implies the middle one, since $4E_Z(u,N)<\|W\|_4^4$. 
Conversely, \eqref{EZ below equiv} implies that under the energy constraint $E_Z<E_S(W)$, the bound $\|N\|_2\le\|W\|_4^2$ is improved to 
$\|N\|_2^2 \le 4E_Z(u,N) < \|W\|_4^4$. Thus we have equivalence of the right two conditions of \eqref{K pos equiv}. In particular $\|N\|_2=\|W\|_4^2$ is impossible. Hence the right two in \eqref{K neg equiv} are also equivalent. 
Since the middle condition is open and the right one is closed, they are preserved on any connected set of $(u,N)$ under the constraint $E_Z(u,N)<E_S(W)$, where $\sign K(u)$ is also preserved, because of \eqref{K cond NLS}, as is well known for NLS. 
It remains to see their matching. 

Consider the curves in the form $(u_\la,N_\la):=(\la u, \la^2 N)$ with $\la>0$. 
We have
\EQ{
 \pt E_Z(u_\la,N_\la) = E_S(u_\la)+\frac{\la^4}{4}\|N-|u|^2\|_2^2,
 \pr \la\p_\la E_Z(u_\la,N_\la) = K(u_\la) + \la^4\|N-|u|^2\|_2^2.}
Suppose that $E_Z(u,N)<E_S(W)$ and $K(u)\ge 0$. 
Then along the curve $C_1:\{(u_\la,N_\la)\}_{\la:1\to 0}$,  
the conditions $E_S(u_\la)\le E_Z(u_\la,N_\la)<E_S(W)$ and $K(u_\la)\ge 0$ are preserved, since $E_Z(u_\la,N_\la)$ is decreasing as long as $K(u_\la)\ge 0$. 
Since $\|N_\la\|_2<\|W\|_4^2$ is trivial at $\la=0$, it holds true also at $\la=1$, namely $\|N\|_2<\|W\|_4^2$.  

For the reverse implication, consider another deformation in terms of $\nu:=N-|u|^2$ given by 
$(u,\nu) \mapsto (u_\la,\nu)$, 
or in terms of $N$, $(u,N) \mapsto (u_\la, N^\la)$, with 
\EQ{
 N^\la:=\nu+|u_\la|^2=N-|u|^2+|\la u|^2.}
Then we have 
\EQ{
 \pt E_Z(u_\la,N^\la)=E_S(u_\la)+\frac14\|\nu\|_2^2, 
 \pq \la \p_\la E_Z(u_\la,N^\la)=K(u_\la).}

Suppose that $E_Z(u,N)<E_S(W)$ and $K(u)<0$. 
Then along the curve $C_2:\{(u_\la,N^\la)\}_{\la:1\to\I}$, 
the conditions $E_S(u_\la)\le E_Z(u_\la,N^\la)<E_S(W)$ and $K(u_\la)<0$ are preserved, since $E_Z(u_\la,N^\la)$ is decreasing as long as $K(u_\la)<0$. 
Moreover, $E_S(u_\la)\to-\I$ as $\la\to\I$, hence $E_Z(u_\la,N^\la)\to-\I$. 
If $\|N\|_2<\|W\|_4^2$, it would remain valid along the curve, but then \eqref{EZ below equiv} contradicts $E_Z(u_\la,N^\la)\to-\I$. 
Hence $\|N\|_2>\|W\|_4^2$. 
\end{proof}

The following estimate on $K$ yields a key monotonicity for the virial identity to show scattering and blow-up below the ground state. It is the 4D version of \cite[Lemma 2.4]{GNW}, with essentially the same proof. 
\begin{lem} \label{lem:estK}
Let $\fy\in\dot H^1(\R^4)$ and $a\ge 0$ satisfy
\EQ{
 E_S(\fy) + a^2/4 \le E_S(W).}
Then we have 
\EQ{
 \CAS{ K(\fy)\ge 0 \implies K(\fy) \ge  a\|\fy\|_4^2, \pq \|W\|_4^2>\|\fy\|_4^2+a,\\
 K(\fy)\le 0 \implies 4K(\fy)+a^2 \le -3a\|\fy\|_4^2.}}
\end{lem}
\begin{proof}
If $K(\fy)=0$ or $a=0$, then the conclusion is trivial. Let $K(\fy)\not=0$ and $a>0$. 
For the $L^2$-dilation $\fy_\mu:=\mu^{d/2}\fy(\mu x)$ we have 
\EQ{
 E_S(\fy_\mu)=\frac{\mu^2}{2}\|\na\fy\|_2^2-\frac{\mu^4}{4}\|\fy\|_4^4,
 \pq K(\fy_\mu)=\mu^2\|\na\fy\|_2^2-\mu^4\|\fy\|_4^4.}
Hence the unique $\mu\in(0,\I)$ solving $K(\fy_\mu)=0$ is given by 
\EQ{
 \mu = \|\na\fy\|_2\|\fy\|_4^{-2}\not=1.}
For this particular $\mu$, the variational property of $W$ implies 
\EQ{
 a^2/4 \le E_S(W)-E_S(\fy) \le E_S(\fy_\mu)-E_S(\fy) = \frac{(\mu^2-1)^2}{4}\|\fy\|_4^4.}
Put $X:=\|\fy\|_4^2/a$. Then the above implies 
\EQ{ \label{range X}
 1 \le |\mu^2-1|X.}

If $K(\fy)>0$, then $\mu>1$, $X>0$ and \eqref{range X} imply
\EQ{
 \frac{K(\fy)}{a\|\fy\|_4^2}=(\mu^2-1)X \ge 1.}
Hence 
\EQ{
 \|W\|_4^4 = 4E_S(W) \ge 4E_S(\fy)+a^2 = 2K(\fy)+\|\fy\|_4^4+a^2 \ge (\|\fy\|_4^2+a)^2.}

If $K(\fy)<0$, then $\mu<1$, and 
\EQ{
 f(X,\mu):=\frac{4K(\fy)+a^2}{a\|\fy\|_4^2} = 4(\mu^2-1)X+1/X}
is monotone in $X>0$. So its maximal value is attained at the boundary of \eqref{range X}
\EQ{
 f(X,\mu) \le -4+|1-\mu^2|=-3-\mu^2 < -3,}
which is, by definition of $f$, the desired estimate in this case. 
\end{proof}

\section{Blow-up and scattering criteria} \label{sect:LWP}
In this section, we study the Zakharov system \eqref{Zak} {\it without imposing the radial symmetry}. 
The first result is persistence of regularity for the local unique solutions given in \cite{Zak4D1}. 
To be precise, we give the definition of solutions. 
\begin{defn} \label{def:sol}
Let $s\ge 1/2$ and let $I\subset\R$ be a non-empty interval. 
A pair of space-time functions $(u,N):I\times\R^4\to\C^2$ is said to be {\it a solution of \eqref{Zak}} on $I$ in $H^s\times L^2$, if $(u,N)\in C(I;H^s\times L^2)$ satisfying \eqref{Zak} on $I$ (in the Duhamel form) and $u\in L^2_tB^{1/2}_{4,2}(J)$ for every compact $J\subset I$. 
\end{defn}
We introduce function spaces for the solutions. 
For $s\in\R$ and $\de\in(0,\de_\star)$, let 
\EQ{ \label{def Zs}
 \pt Z^s_0:=L^\I_tH^s \times L^\I_tL^2,
 \pr Z^s_1:=X^s_1\times Y_1, \ X^s_1:=L^\I_t L^{2(-s)}\cap L^2_tB^s_{4,2}, \ Y_1:=L^\I_tL^2+L^2_tL^4,}
and the Fr\'echet space $\cZ^s$ is defined by 
\EQ{ \label{def cZs}
 \cZ^s:=\{(u,N)\in \pt C(\R;H^s\times L^2)\cap Z^s_0 \mid 
 \prq s'<1 \tand s'\le s\implies u\in L^2_t B^{s'}_{4,2}\}.}
For $s<1$, it is a closed subspace of $Z^s_0\cap Z^s_1$, with the norm 
\EQ{
 \|(u,N)\|_{\cZ^s}:=\|u\|_{L^\I_tH^s}+\|u\|_{L^2_tB^s_{4,2}}+\|N\|_{L^\I_tL^2}.}
 
The local wellposedness \cite{Zak4D1} in $H^{1/2}\times L^2$ implies that every solution on an interval $I$ is uniquely determined by its value at any $t_0\in I$. 
The maximal existence time $T^*\le\I$ is uniquely defined such that the solution can be extended up to $t<T^*$ but not beyond it. 
The persistence of regularity \cite{Zak4D1} implies that the maximal existence time $T^*$ is independent of $s\in[1/2,1)$. It is true also in the energy space $s=1$: 
\begin{prop} \label{prop:persist}
Let $(u,N)$ be a solution of \eqref{Zak} on an open interval $I\ni 0$ in $H^{1/2}\times L^2$. 
If $u(0)\in H^1$, then $(u,N)\in \cZ^1(J)$ for all compact $J\subset I$.  
Moreover, every sequence of solutions $(u_k,N_k)$ satisfying $(u_k(0),N_k(0))\to (u(0),N(0))$ in $H^1\times L^2$, converges to $(u,N)$ in $\cZ^1(J)$. 
\end{prop}
Combined with the local wellposedness \cite{Zak4D1} in $H^{1/2}\times L^2$, it implies the local wellposedness in the energy space: 
\begin{cor} \label{cor:LWP}
For any $\fy\in H^1\times L^2$, there exist $\de>0$ and $T>0$ such that for any $(u(0),N(0))\in\B:=\{\psi\in H^1\times L^2\mid \|\psi-\fy\|_{H^{1/2}\times L^2}\le\de\}$, 
there exists a unique solution $(u,N)$ of \eqref{Zak} on $[0,T]$ in $H^1\times L^2$. 
The solution map $(u(0),N(0))\mapsto (u,N)$ is continuous from $\B$ to $C([0,T];H^1\times L^2)$. 
\end{cor}
\begin{proof}[Proof of Proposition \ref{prop:persist}]
As mentioned above, \cite[Proposition 5.1]{Zak4D1} implies that, for all $s\in[1/2,1)$, $(u,N)\in \cZ^s(J)$ as well as the continuity of the solution map in this topology. 
So it suffices to prove $u\in C(I;H^1)$, together with the map continuity.  

First, to prove $u\in C(J;H^1)$, we assume in addition that $N_k(0)\in H^{1/2}$ for all $k\in\N$. 
We may assume $0\in J$ with no loss. 
Then the local wellposedness in $H^1\times H^{1/2}$ together with the regularity persistence (cf.~\cite{Zak4D1}) implies that the corresponding solutions exist $(u_k,N_k)\in C(J;H^1\times H^{1/2})$ for large $k$, satisfying  
\EQ{
 \|u_k-u\|_{L^\I_t H^s(J)}+\|N_k-N\|_{L^\I_tL^2(J)}\to 0\ (k\to\I)}
for all $s<1$. Moreover, $(u_k,N_k)$ satisfies the conservation of $E_Z$ and $\|u_k\|_2$. 

Since $N_k\to N$ in $C(J;L^2)$, for any $\e>0$ there exists $j\ge 1$ such that 
\EQ{
 \sup_{k\in\N}\sup_{t\in J}\|P_{>j}N_k(t)\|_2 <\e.}
Let $N_k^0:=P_{>j}N_k$ and $N_k^1:=P_{\le j}N_k$. Then the nonlinear part of energy 
is controlled 
\EQ{ 
 |\LR{N_k||u_k|^2}| \pt\le \|N^0_k\|_2\|u_k\|_4^2+\|N^1_k\|_{8/3}\|u_k\|_{16/5}^2
 \pr\lec \|N^0_k\|_2\|\na u_k\|_2^2 + \|N^1_k\|_{H^{1/2}}\|\na u_k\|_2^{3/2}\|u_k\|_2^{1/2}
 \pr\lec \e\|\na u_k\|_2^2 + j^{1/2}\|N\|_{L^\I_tL^2(J)}\|\na u_k\|_2^{3/2}\|u_k\|_2^{1/2}+o(1),}
as $k\to\I$. Then, if $\e>0$ is chosen small enough, 
\EQ{
 2E_Z(u_k,N_k)
  \pt\ge \|\na u_k\|_2^2+\|N_k\|_2^2/2-|\LR{N_k||u_k|^2}|
 \pr\ge \|\na u_k\|_2^2/2 + \|N_k\|_2^2/2 
 \prQ-C  j^{1/2}\|N\|_{L^\I_tL^2(J)}\|u(0)\|_2^{1/2}\|\na u_k\|_2^{3/2},}
which implies that $u_k$ is bounded in $L^\I(J;H^1)$. 
Since $u_k\to u$ in $C(J;H^s)$, the uniform bound implies the convergence in $C(J;\weak{H^1})$. 
Since $N_k\to N$ in $C(J;L^2)$, the weak convergence of $u_k$ implies $\LR{N_k||u_k|^2}\to\LR{N||u|^2}$ and so  
\EQ{
 E_Z(u(t),N(t))\le E_Z(u(0),N(0))=\lim_{k\to\I}E_Z(u_k(0),N_k(0))} 
at each $t\in J$, 
but if the inequality were strict then solving the equation backward using the same argument starting from that time, would yield a contradiction, since the solution is unique (by \cite{Zak4D1}). Hence the energy must be conserved, which implies the strong continuity in time, namely $u\in C(J;H^1)$. 

Now that we have the regularity persistence in $H^1\times L^2$, repeating the same argument as above without the additional assumption $N_k(0)\in H^{1/2}$, we have $(u_k,N_k)\in C(J;H^1\times L^2)$ and the conservation law. Hence the above argument implies that $u_k\to u$ strongly in $C(J;H^1)$. 
\end{proof}

The first step towards global analysis is to characterize blow-up and scattering by the space-time norms of the solution. 
The following criteria, respectively for blow-up and for scattering, follow essentially from the analysis in \cite{Zak4D1}, though it is not entirely obvious. 
Proofs are given below in this section, for the sake of completeness. 
\begin{prop} \label{prop:bc0}
Let $s\in[1/2,1]$ and $(u,N)$ be a solution of \eqref{Zak} from $t=0$ in $H^s\times L^2$ with the maximal existence time $T^*<\I$. Then $\|u\|_{L^2_t B^{1/2}_{4,2}(0,T^*)}=\I$. 
\end{prop}
Since the existence time is common among all $s\in[1/2,1]$, the above statement is the strongest in the case $s=1/2$. 
We are not able to conclude blow-up of the norm $\|u(t)\|_{H^s}+\|N(t)\|_2$ because of possibility of blow-up by concentration.  
For the scattering, we have the following criteria. 
\begin{prop} \label{prop:sc}
Let $s\in[1/2,1]$ and $(u,N)$ be a solution of \eqref{Zak} on $[0,\I)$ in $H^s\times L^2$. Then the following conditions are equivalent. 
\begin{enumerate}
\item \label{it:sc} $(u,N)$ scatters in $H^s\times L^2$, i.e., \eqref{def scat}. 
\item \label{it:u-St} $u\in L^2((0,\I);B^{1/2}_{4,2})$. 
\item \label{it:Y1} $(u,N)\in Z^{1/2}_0(0,\I)$ and $\|N\|_{Y_1(T,\I)}\to 0$ as $T\to\I$.
\end{enumerate}
Moreover, they imply $(u,N)\in\cZ^s([0,\I))$. 
\end{prop}
To prove the scattering in the energy space by the weak convergence argument as in \cite{Zak4D1}, we need the wave operator for larger data, which has independent interest to characterize the global dynamics.  
\begin{prop} \label{prop:wo}
For any $s\in[1/2,1]$ and for any $\fy\in H^s\times L^2$, there exist $\de>0$ and $T>0$ such that for any $(u_+,N_+)\in\B:=\{\psi\in H^s\times L^2 \mid \|\psi-\fy\|_{H^{1/2}\times L^2}\le\de\}$, there exists a unique solution $(u,N)$ of \eqref{Zak} on $[T,\I)$ scattering in $H^s\times L^2$ with the scattering profile $(u_+,N_+)$. 
Moreover, the solution map $(u_+,N_+)\mapsto(u,N)$ is continuous from $\B$ to $\cZ^s([T,\I))$. 
\end{prop} 
A proof is given below in this section. 
As an immediate consequence of this and Proposition \ref{prop:sc}, we obtain persistence of regularity for scattering: 
\begin{cor} \label{cor:perwo}
Let $(u,N)$ be a solution of \eqref{Zak} on $[0,\I)$ scattering in $H^{1/2}\times L^2$. 
If either the initial data or the scattering profile belongs to $H^s\times L^2$ for some $s\in(1/2,1]$, then $(u,N)$ scatters also in $H^s\times L^2$ with the same profile. 
\end{cor}
\begin{proof}
If the profile is in $H^s\times L^2$, then Proposition \ref{prop:wo} yields a solution scattering in $H^s\times L^2$, which must be the same $(u,N)$ by the uniqueness in $H^{1/2}\times L^2$. 
If $u(0)\in H^s$, then $(u,N)$ is a solution in $H^s\times L^2$ by persistence of regularity for the Cauchy problem. Since the conditions \eqref{it:u-St} and \eqref{it:Y1} of Proposition \ref{prop:sc} are independent of $s$, we conclude the scattering in $H^s\times L^2$. 
\end{proof}

Before proving the above results (Propositions \ref{prop:bc0}, \ref{prop:sc} and \ref{prop:wo}), we recall some standard product estimates, including the bilinear operators. They are essentially the same as those used in \cite{Zak4D1}. 
In the following estimates, the implicit constants depend (continuously) on the exponent $s$. 

First, for $(j,k)\in\HL_\io$, $-2 \le b \le a \le 2$, and $a-b-2 \le c \le 2$, we have 
\EQ{ \label{Om est0}
 \pt\|\Om_\io(f_j,g_k)\|_{2(a)}+\|\ti\Om_\io(f_j,g_k)\|_{2(a)}+\|\ti\Om_\io(g_k,f_j)\|_{2(a)} 
 \prq\lec j^{-2}\|f_j\|_{2(b)}\|g_k\|_{2(-2+a-b)}
 \pn\lec j^{-2}k^{2-a+b+c}\|f_j\|_{2(b)}\|g_k\|_{2(c)}.}
Choosing $(a,b,c)=(0,0,-s)$ yields, after the summation over $\HL_\io$, 
\EQ{ \label{uN est1}
 -2\le s<2 \implies \|\Om_\io(f,g)\|_{H^s} \lec \io^{2-s}\|f\|_2\|g\|_{2(-s)}. }
Replacing $s$ with $s+1$ and using $H^{s+1}\subset B^s_{4,2}\subset L^{2(-1-s)}$, we obtain 
\EQ{ \label{uN est2}
 0\le s<1 \implies \|\Om_\io(f,g)\|_{B^s_{4,2}} \lec \io^{1-s}\|f\|_2\|g\|_{B^s_{4,2}}.}
Similarly, choosing $(a,b,c)=(1,-\de,1-s)$ in \eqref{Om est0}, we obtain 
\EQ{ \label{uN est3}
 0\le \de<2-s\le 3 \implies \|\Om_\io(f,gh)\|_{B^s_{4/3,2}} \pt\lec \io^{2-s-\de}\|f\|_{\dot B^{-\de}_{2(-\de),\I}}\|gh\|_{4/3(-s)}
 \pr\lec \io^{2-s}\|f\|_{\dot B^{-\de}_{2(-\de),\I}} \|g\|_2\|h\|_{4(-s)}
 \pr\lec \io^{2-s}\|f\|_2\|g\|_2\|h\|_{4(-s)},}
and with $b=a-1/2-s$ and $c=-1/2$, 
\EQ{ \label{uN est4}
 \pt -1/2\le s<1,\ -1/2\le a\le 2  
 \pr\implies\|\Om_\io(D(fg),h)\|_{B^s_{2(a),2}} \lec \io^{1-s}\|fg\|_{2(a-1/2-s)}\|h\|_{2(-1/2)}
 \prQ\pQ\lec \io^{1-s}\|f\|_{4(a-1/2)}\|g\|_{4(-s)}\|h\|_{2(-1/2)}.}

For the terms in the equation of $N$, choosing $(a,b,c)=(1,0,-s)$ in \eqref{Om est0} yields 
\EQ{ \label{uu est1}
 -2\le s<1 \implies \|D\ti\Om_\io(f,g)\|_{B^{2s}_{4/3,2}} \lec \io^{1-s}(\|f\|_{\dot H^s}\|g\|_{2(-s)}+\|g\|_{\dot H^s}\|f\|_{2(-s)}),}
while $(a,b,c)=(0,0,-s)$ yield
\EQ{ \label{uu est4}
 -2\le s<2 \implies \|D\ti\Om_\io(f,g)\|_2 \lec \io^{2-s}(\|f\|_{\dot H^s}\|g\|_{2(-s)}+\|g\|_{\dot H^s}\|f\|_{2(-s)}),}
and $(a,b,c)=(0,-s,-s)$ yields
\EQ{ \label{uu est4.1}
 0\le s<1 \implies \|D\ti\Om_\io(f,g)\|_{2} \lec \io^{2-2s}\|f\|_{2(-s)}\|g\|_{2(-s)}.}
Similarly, using 
\EQ{
 \CAS{(j,k)\in\HL_\io \implies \|D\ti\Om_\io(f_j,g_k)\|_2 \lec j^{-1+s}k^{1-s}\|f_j\|_{2(s)}\|g_k\|_{4(-s)},\\
 (k,j)\in\HL_\io \implies \|D\ti\Om_\io(f_j,g_k)\|_2 \lec jk^{-1}\|f_j\|_{2(s)}\|g_k\|_{4(-s)},}}
and the same estimates on $D\ti\Om_\io(g_k,f_j)$, we obtain for $0\le s<-1$, 
\EQ{ \label{uu est2}
 \|D\ti\Om_\io(fg,h)\|_2 + \|D\ti\Om_\io(h,fg)\|_2 \lec \io^{1-s}\|f\|_2\|g\|_{4(s-1)}\|h\|_{4(-s)}.}

For the remainder of the normal form, we have 
\EQ{ \label{uN est5}
 s>0,\ 0\le\te\le 1\implies \|(f,g)_{\LH_\io}\|_{B^s_{2(\te),2}} \lec \io^{-s}\|f\|_{4(\te)}\|g\|_{B^s_{4,2}}.}
\begin{proof}
By definition of $\LH_\io$, the bilinear term is decomposed in the frequency
\EQ{
 (f,g)_{\LH_\io}=(fg)_{<1/\io} + (f_{<2/\io}g_{<2})_{\ge 1/\io} + \sum_{k\in 2^\N}(f_{<k/\io} g_k)_{\ge 1/\io}.}
The first term on the right is simply bounded by H\"older and $B^0_{2(\te),2}\supset L^{2(\te)}$ 
\EQ{
 \|(fg)_{<1/\io}\|_{B^s_{2(\te),2}} \lec \io^{-s}\|fg\|_{2(\te)} \lec \io^{-s}\|f\|_{4(\te)}\|g\|_4
 \lec \io^{-s}\|f\|_{4(\te)}\|g\|_{B^s_{4,2}}.}
The second term is bounded in the same way, as its frequency is also bounded by $4/\io$. 
The third term is further decomposed into $(f_jg_k)_l$. The sum over $j\lec k\sim l$ and that over $l\lec j\sim k$ are estimated in $B^s_{2(\te),2}$ respectively by 
\EQ{
 \pt \squm_{l\ge 1/\io} \sum_{k\sim l} k^s \|f_{\lec k}\|_{4(\te)}\|g_k\|_4
 \lec \|f\|_{4(\te)}\|g\|_{B^s_{4,2}},
 \pr \squm_{l\ge 1/\io} \sum_{j\sim k \gec l} (l/k)^s\|f_j\|_{4(\te)}\|g_k\|_{B^s_4} \lec \|f\|_{B^0_{4(\te),\I}}\|g\|_{B^s_{4,2}},}
where Young's inequality is used in the second case. Similarly, the remaining part $1/\io\le l\sim j>k>\io j$ is bounded by 
\EQ{
 \squm_{l\ge 1/\io} \sum_{l\sim j>k>\io j} (l/k)^s\|f_j\|_{4(\te)}\|g_k\|_{B^s_4} \lec \io^{-s}\|f\|_{B^0_{4(\te),\I}}\|g\|_{B^s_{4,2}}.}
The proof is concluded with $B^0_{4(\te),\I}\supset L^{4(\te)}$. 
\end{proof}

Similarly, for the remainder in the equation of $N$, we have 
\EQ{ \label{uu est3}
1/2\le s<1 \implies \|D(f,g)_{\HH_\io}\|_2 \lec \io^{s-1}[\|f\|_{\dot B^s_{4,2}}\|g\|_{B^{1-s}_{4,\I}}+\|f\|_{B^{1-s}_{4,\I}}\|g\|_{\dot B^s_{4,2}}].}
\begin{proof}
Decomposing the bilinear term into $D(f_jg_k)_l$, the definition of $\HH_\io$ implies that either $j\le k<\max(2,j)/\io$ or $k<j<\max(2,k)/\io$. Since we have 
\EQ{
 \pt \squm_l \sum_{j\sim k\gec l}\|D(f_jg_k)_l\|_2 \lec \squm_l \sum_{j\sim k\gec l} l/j \|f_j\|_{\dot B^{1-s}_4}\|g_k\|_{\dot B^{s}_4} \lec \|f\|_{\dot B^{1-s}_{4,\I}} \|g\|_{\dot B^{s}_{4,2}},}
we may restrict to the region $j\ll k\sim l\lec\LR{k}/\io$ or $k\ll j\sim l\lec\LR{j}/\io$. 
By symmetry, it suffices to consider the former case. 
For $j\ge 2$, it is bounded by 
\EQ{
 \squm_l \sum_{l\sim k>j>\io k} (l/j)^{1-s}\|f_j\|_{\dot B^{1-s}_4}\|g_k\|_{\dot B^s_4} \lec \|f\|_{\dot B^{1-s}_{4,\I}} \|g\|_{\dot B^s_{4,2}}. }
For $j<2$, it is bounded by 
\EQ{
 \squm_{l\lec 1/\io}\sum_{k\sim l} \|D((f_{\ll k})_{<2}g_k)_l\|_2
 \lec \squm_{l\lec 1/\io} l^{1-s} \|f\|_4\|g_l\|_{\dot B^s_4} \lec \|f\|_4\|g\|_{\dot B^s_{4,2}}.}
Thus, using $B^{1-s}_{4,\I}=\dot B^{1-s}_{4,\I}\cap L^4$, we obtain the desired estimate.  
\end{proof}

The above estimates on the normal form imply its continuity and invertibility. 
\begin{lem} \label{lem:normal inv}
For any $s\in[1/2,2)$ and any $R\in(0,\I)$, let 
\EQ{
 \B_R^s:=\{(\fy,\psi)\in H^s\times L^2 \mid \|\fy\|_{2(-1/2)}+\|\psi\|_2\le R\}.}
Then $\Psi_{\io_1,\io_2}$ defined by \eqref{def normal} is a Lipschitz map $\B_R^s\to H^s\times L^2$. 
If $\io_j^{1-s}R\ll 1$ for $j=1,2$, then $\Psi_{\io_1,\io_2}(\B_R^s)\subset\B_{2R}^s$, and there is a unique inverse $\Psi_{\io_1,\io_2}^{-1}:\B_R^s\to\B_{2R}^s$, which is also Lipschitz. Moreover, if $(\fy',\psi')=\Psi_{\io_1,\io_2}(\fy,\psi)$ and either $(\fy,\psi)$ or $(\fy',\psi')$ is in $\B_R^s$, 
\EQ{ \label{normal equi}
 \pt \|\fy-\fy'\|_{H^s}+\|\psi-\psi'\|_2 \ll \|\fy\|_{2(-s)} \sim \|\fy'\|_{2(-s)},}
and for $s<1$ we have in addition 
\EQ{ \label{normal equi2}
 \|\fy\|_{B^s_{4,2}} \sim \|\fy'\|_{B^s_{4,2}}.}
In particular, if $(u'(t),N'(t))=\Psi_{\io_1,\io_2}(u(t),N(t))$ with either $(u,N)$ or $(u',N')$ is in $\B_R^s$ for large $t$, then the scattering of $(u,N)$ in $H^s\times L^2$ is equivalent to that of $(u',N')$, with the same profile. 
\end{lem}
\begin{proof}
\eqref{uN est1} and \eqref{uu est4}, together with $H^s\subset L^{2(-s)}$, imply that the map $\vec\Om_{\io_1,\io_2}$ is Lipschitz $\B_{2R}^s\to H^s\times L^2$ for $0\le s<2$ with a Lipschitz constant $\lec\io^{2-s}R$. 
Hence if $\io_j^{2-s}R\ll 1$, then it is a contraction $\vec\Om_{\io_1,\io_2}:\B_{2R}^s\to \B_R^s$. Therefore $\Psi_{\io_1,\io_2}:\B_R^s\to\B_{2R}^s$, and $\Psi_{\io_1,\io_2}^{-1}:\B_R^s\to\B_{2R}^s$ is obtained by the Banach fixed point theorem. Moreover, if $(\fy',\psi')=\Psi_{\io_1,\io_2}(\fy,\psi)$ and either $(\fy,\psi)$ or $(\fy',\psi')$ belongs to $\B_R^s$, then using \eqref{uN est1} and \eqref{uu est4}, we have 
\EQ{ \label{NT diff}
 \pt\|\fy-\fy'\|_{H^s} + \|\psi-\psi'\|_2 = \|\Om_{\io_1}(\psi,\fy)\|_{H^s}+\|D\ti\Om_{\io_2}(\fy,\bar\fy)\|_2
 \prq\lec (\io_1^{2-s}\|\psi\|_2+\io_2^{2-s}\|\fy\|_{H^s})\|\fy\|_{2(-s)} \ll \|\fy\|_{2(-s)},}
hence $\|\fy'\|_{2(-s)}\sim\|\fy\|_{2(-s)}$. 
Similarly, the equivalence in $B^s_{4,2}$ follows from \eqref{uN est2}, taking $\io_j>0$ smaller if needed such that $\io_j^{1-s}R\ll 1$. 
The equivalence of scattering follows from the $L^{2(-s)}$ decay by Lemma \ref{lem:scdec} and \eqref{normal equi}. 
\end{proof}

Next we estimate the nonlinear terms defined in \eqref{def FG}. For $F_\io$, we obtain from 
\eqref{uN est3}, \eqref{uN est4} with $a=0$ and \eqref{uN est5} with $\te\in\{0,1\}$, for $0<s<1$,
\EQ{ \label{est u-Duh}
 \pt \|\Om_\io(N,mu)\|_{L^2_tB^s_{4/3,2}(I)} \lec \io^{2-s}\|N\|_{L^\I_tL^2(I)}\|m\|_{L^\I_tL^2(I)}\|u\|_{L^2_tB^s_{4,2}(I)},
 \pr \|\Om_\io(D(uv),w)\|_{L^1_tH^s(I)} \lec \io^{1-s}\|u\|_{L^2_tB^{s}_{4,2}(I)}\|v\|_{L^2_tB^{1/2}_{4,2}(I)}\|w\|_{L^\I_tH^{1/2}(I)},
 \pr \|(n,u)_{\LH_\io}\|_{(L^2_tB^s_{4/3,2}+L^1_tH^s)(I)} \lec \io^{-s}\|N\|_{Y_1(I)}\|u\|_{L^2_tB^s_{4,2}(I)}.}
For $G_\io$, we obtain from \eqref{uu est2} and \eqref{uu est3}, for $1/2\le s<1$,
\EQ{ \label{est N-Duh}
 \pt \|D\ti\Om_\io(nu,v)\|_{L^1_tL^2(I)}+\|D\ti\Om_\io(v,n\bar u)\|_{L^1_tL^2(I)}
  \prQ\lec \io^{1/2}\|N\|_{L^\I_tL^2(I)}\|u\|_{L^2_tB^{1/2}_{4,2}(I)}\|v\|_{L^2_tB^{1/2}_{4,2}(I)},
 \pr \|D(u_0,u_1)_{\HH_\io}\|_{L^1_tL^2(I)} \lec \io^{s-1}\sum_{\te=0,1}\|u_\te\|_{L^2_t\dot B^s_{4,2}(I)}\|u_{1-\te}\|_{L^2_tB^{1-s}_{4,\I}(I)}.}

Combining them together with the free Strichartz estimate yields the following estimates on the Duhamel integrals defined in \eqref{def Duh}. 
For any function $(u,N)$ on any interval $I\subset\R$, any $T\in \bar{I}\subset[-\I,\I]$, and any $s\in(0,1)$,  we have 
\EQ{ \label{Duh est u}
 \|\cU_\io^T(u,N)\|_{(L^\I_tH^s\cap L^2_tB^s_{4,2})(I)} \pt\lec \|F_\io(u,N)\|_{(L^1_tH^s+L^2_tB^s_{4/3,2})(I)}
  \pr\lec \io^{2-s}\|N\|_{L^\I_tL^2(I)}^2\|u\|_{L^2_tB^s_{4,2}(I)}
  \prq+ \io^{1-s}\|u\|_{L^2_tB^{s}_{4,2}(I)}\|u\|_{L^2_tB^{1/2}_{4,2}(I)}\|u\|_{L^\I_tH^{1/2}(I)} 
  \prq+ \io^{-s}\|N\|_{Y_1(I)}\|u\|_{L^2_tB^s_{4,2}(I)},}
and for $s\in[1/2,1)$, 
\EQ{ \label{Duh est N}
 \|\cN_\io^T(u,N)\|_{L^\I_t L^2(I)} \pt\lec \|G_\io(u,N)\|_{L^1_tL^2(I)} 
 \pr\lec \io^{1/2}\|N\|_{L^\I_tL^2(I)}\|u\|_{L^2_tB^{1/2}_{4,2}(I)}^2
 \pn + \io^{s-1}\|u\|_{L^2_tB^{s}_{4,2}(I)}^2.}
Combining them with \eqref{uN est1}, \eqref{uN est2} and \eqref{uu est4.1} yields, for $s\in[1/2,1)$, 
\EQ{ \label{Duh est sum}
 \pt\|\vec\Om_\io(u,N)\|_{\cZ^s(I)}+\|\D_\io^T(u,N)\|_{\cZ^s(I)}
 \pr\lec \BR{\io^{1-s}\|(u,N)\|_{\cZ^s(I)}\LR{\|(u,N)\|_{Z^{1/2}_0(I)}}+\io^{-s}\|u\|_{X^s_1(I)}}\|(u,N)\|_{Z^{1/2}_1(I)}.}
Since they are multi-linear, the same argument yields a similar estimate on the difference, namely
\EQ{ \label{Duh est sum dif}
 \pt\|\vec\Om_\io(u_0,N_0)-\vec\Om_\io(u_1,N_1)\|_{\cZ^s(I)}+\|\D_\io^T(u_0,N_0)-\D_\io^T(u_1,N_1)\|_{\cZ^s(I)}
 \pr\lec \sum_{j=0,1}\BR{\io^{1-s}\|(u_j,N_j)\|_{\cZ^s(I)}\LR{\|(u_j,N_j)\|_{Z^{1/2}_0(I)}}+\io^{-s}\|(u_j,N_j)\|_{Z^s_1(I)}}
  \prQ\times \|(u_0-u_1,N_0-N_1)\|_{\cZ^s(I)}.}

With those estimates, we are ready to prove the propositions. 
\begin{proof}[Proof of Proposition \ref{prop:bc0}]
Since the existence time $T^*$ is independent of $s\in[1/2,1]$, it suffices to prove in the case $s=1/2$. 
Suppose the contrary and let 
\EQ{ \label{def M0}
 M_0:=\|u\|_{L^2_t B^{1/2}_{4,2}(0,T^*)}<\I.}
Let $\be>1$ large enough so that 
\EQ{
 \pt \io_1:=\be^{-3/2}M_0^{-4}<1/2, \pq \io_2:=\io_1^2=\be^{-3}M_0^{-8}, 
 \pr M_1:=\be^3M_0^9 \ge \|u(0)\|_{H^{1/2}}, \pq M_2:=M_1^{2/3}=\be^2M_0^6 \ge \|N(0)\|_2.}
Let $T\in(0,T^*]$ be the maximal time such that 
\EQ{ \label{apbd M12}
 \|u\|_{L^\I_t H^{1/2}(0,T)}\le 2M_1, \pq \|N\|_{L^\I_tL^2(0,T)}\le 2M_2.}
Since $(u,N)$ is a solution on $(0,T)$, we have 
\EQ{
 \pt (\ti u,\ti N):=\Psi_{\io_1,\io_2}(u,N)=\Uf(t)(\ti u(0),\ti N(0))+(\cU_{\io_1}^0(u,N),\cN_{\io_2}^0(u,N)).
}
Then \eqref{uN est1} and \eqref{Duh est u} with $s=1/2$, $T=0$ and $I=(0,T)$ yield, in terms of the bounds in  \eqref{def M0} and \eqref{apbd M12}, 
\EQ{
 \pt\|u\|_{L^\I_tH^{1/2}(0,T)}-\|u(0)\|_{H^{1/2}}
 \prq\lec \io_1^{3/2}M_2M_1 + \io_1^{3/2}M_2^2M_0 + \io_1^{1/2}M_0^2M_1 + \io_1^{-1/2}M_2M_0
 \prq= (\be^{-1/4}+\be^{-1/2}\io_1^{1/2}+\be^{-3/4}+\be^{-1/4})M_1 \lec \be^{-1/4}M_1,}
and similarly, \eqref{uu est4} and \eqref{Duh est N} yield
\EQ{
 \|N\|_{L^\I_tL^2(0,T)}-\|N(0)\|_2
 \pt\lec \io_2^{3/2}M_1^2 + \io_2^{1/2}M_2 M_0^2 + \io_2^{-1/2}M_0^2
 \pr= (\be^{-1/2}+\be^{-3/4}\io_1^{1/2}+\be^{-1/2})M_2 \lec \be^{-1/2}M_2.}
If $\be>1$ is large enough, then the above estimates imply
\EQ{ \label{ap Hbd}
 \|u\|_{L^\I_tH^{1/2}(0,T)}<2M_1, \pq \|N\|_{L^\I_tL^2(0,T)}<2M_2.}
Hence the maximality of $T$ for \eqref{apbd M12}, together with the continuity, implies $T=T^*$. 

Now that $(u,N)$ is bounded in $H^{1/2}\times L^2$, we may use Lemma \ref{lem:normal inv} with $s=1/2$ and $R=4(M_1+M_2)$, so that $(u,N)$ can be recovered from $(\ti u,\ti N)=\Psi_{\io_1,\io_2}(u,N)$ at each $t\in[0,T^*)$ by the inverse map $\Psi_{\io_1,\io_2}^{-1}$. 
The Duhamel formula, together with the free Strichartz estimate, implies that $(\ti u,\ti N)$ is continuous for $t\le T^*$, namely $(\ti u,\ti N)\in C([0,T^*];H^{1/2}\times L^2)$. 
Then by the continuity of $\Psi_{\io_1,\io_2}^{-1}$, $(u,N)$ belongs to the same function space. 
Then the local wellposedness in $H^{1/2}\times L^2$ of \cite{Zak4D1} extends the local solution $(u,N)$ beyond $T^*$, contradicting its maximality. 
\end{proof}
It is worth formulating the a priori bound on $H^{1/2}\times L^2$ obtained in \eqref{ap Hbd}: 
\begin{lem} \label{lem:Hbd}
There exists a constant $\be_0>1$ such that for any interval $I$ and any solution $(u,N)$ of \eqref{Zak} on $I$ in $H^{1/2}\times L^2$, we have 
\EQ{
 \pt \sup_{t\in I} \|u(t)\|_{H^{1/2}} \le 2\inf_{t\in I} \|u(t)\|_{H^{1/2}} + \be_0^3(1+\|u\|_{L^2_tB^{1/2}_{4,2}(I)}^9),
 \pr \sup_{t\in I} \|N(t)\|_2 \le 2\inf_{t\in I}\|N(t)\|_2 + \be_0^2(1+\|u\|_{L^2_tB^{1/2}_{4,2}(I)}^6).}
\end{lem}

We are now ready to prove the scattering criterion for $s<1$. The case of the energy space $s=1$ will be proven after Proposition \ref{prop:wo}. 
\begin{proof}[Proof of Proposition \ref{prop:sc} for $s<1$]
First assume \eqref{it:u-St}. Then by Lemma \ref{lem:Hbd}, $(u,N)$ is bounded in $H^{1/2}\times L^2$ for $t>0$. 
Moreover, the same estimates as in \eqref{Duh est u}--\eqref{Duh est N} imply that $\Uf(-t)\D_\io^0(u,N)$ converges in $H^{1/2}\times L^2$. In other words, $(\ti u,\ti N):=\Psi_\io(u,N)$ scatters in $H^{1/2}\times L^2$, so does $(u,N)$ for small $\io>0$, by Lemma \ref{lem:normal inv}. 
Then Lemma \ref{lem:scdec}, together with $u\in L^2_tB^{1/2}_{4,2}(0,\I)$, implies that $\|(u,N)\|_{Z^{1/2}_1(T,\I)}\to 0$ as $T\to\I$. 
Hence if we choose $\io>0$ small enough, and $T>1$ large enough depending on $\io$, then 
\EQ{
 \io^{1-s}\LR{\|(u,N)\|_{\cZ^{1/2}(T,\I)}} + \io^{-s}\|(u,N)\|_{Z^{1/2}_1(T,\I)} \ll 1.}
Plugging this into \eqref{Duh est sum} on $I=(T,T')$ yields 
\EQ{
 \|(u,N)\|_{\cZ^s(T,T')} \pt\lec \|\Uf(t-T)(\ti u,\ti N)(T)\|_{\Z^s(T,T')} \sim \|u(T)\|_{H^s}+\|N(T)\|_2,}
uniformly for $T'>T$, so that the limit $T'\to\I$ implies $(u,N)\in\cZ^s(T,\I)$. Then \eqref{it:sc}, namely the scattering in $H^s\times L^2$, follows as above in the case of $s=1/2$. 

Next assume \eqref{it:sc}. 
Then $(u,N)$ is bounded in $H^s\times L^2$, and Lemma \ref{lem:scdec} implies \eqref{it:Y1}. 

Next assume \eqref{it:Y1}. 
Then we can take small $\io>0$, ensuring as above the equivalence between $u$ and $\ti u$ in \eqref{normal equi}-\eqref{normal equi2}. Let $\|u\|_{L^\I_tH^s(0,\I)}+\|N\|_{L^\I_tL^2(0,\I)}\le M$. 
Choosing $\io>0$ small enough, and then $T>1$ large enough ensure, 
\EQ{
 \de:=\io^{1/2}M^2 +  \io^{-1/2}\|N\|_{Y_1(T,\I)} \ll 1.}
If $\|u\|_{L^2_tB^{1/2}_{4,2}(T,T')}\le \be M$ for some $\be>1$ and $T'>T$, then \eqref{Duh est u}, together with the norm equivalence of $u$ and $\ti u$, yields  
\EQ{
 \|u\|_{L^2_tB^{1/2}_{4,2}(T,T')} \lec \|u(T)\|_{H^{1/2}} + \de \be^2 M \le (1+\de \be^2)M.}
Hence choosing $\be>1$ large enough and then $\de>0$ small enough, we obtain 
\EQ{
 \|u\|_{L^2_t B^{1/2}_{4,2}(T,T')} \lec \|u(T)\|_{H^{1/2}} \le M \ll \be M.}
Sending $T'$ from $T$ to $\I$, the continuity of the norm in $T'$ implies $\|u\|_{L^2_t B^{1/2}_{4,2}(T,\I)} \lec M$, in particular \eqref{it:u-St}. 
\end{proof}

Next we construct the wave operator for $s\in[1/2,1)$. 
The case of $s=1$ will be treated separately by the weak limit argument. 
\begin{proof}[Proof of Proposition \ref{prop:wo} for $s<1$]
First we prove the unique existence part. 
Since we have proven Proposition \ref{prop:sc} for $s<1$ where the criteria \eqref{it:u-St} and \eqref{it:Y1} are  independent of $s$, it suffices to consider the case $s=1/2$. 
Let $(u_+,N_+)\in\B$, $\io\in(0,1)$, $T>0$ and define a mapping $\Phi_\io$ for $(u,N)\in\cZ^{1/2}(T,\I)$ by
\EQ{ \label{def Phi}
 \Phi_\io(u,N) \pt:= \Uf(t)(u_+,N_+) + \vec\Om_\io(u,N) + \D_\io^\I(u,N).}
The scattering of a solution $(u,N)$ with the profile $(u_+,N_+)$ is equivalent, by Lemma \ref{lem:normal inv}, 
 to that of $\Psi_\io(u,N)$ for small $\io>0$, and so to $(u,N)=\Phi_\io(u,N)$. 

\eqref{Duh est sum} implies that $\Phi_\io$ maps $\cZ^{1/2}(T,\I)$ into itself, and \eqref{Duh est sum dif}  implies that it is a contraction mapping on a closed subset with the constraints
\EQ{ \label{contraction set}
 \pt \|(u,N)\|_{Z^{1/2}_0(T,\I)}\le M,
 \pq \|(u,N)\|_{Z^{1/2}_1(T,\I)}\le \e}
for $0<\e<1<M<\I$ and $T\in(0,\I)$ satisfying   
\EQ{ \label{small for ScatStz}
 \pt \io^{1/2}M^2 + \io^{-1/2}\e \ll 1, \pq \|u_+\|_{H^{1/2}}+\|N_+\|_2 \le M/2,
 \pr \|\Uf(t)(u_+,N_+)\|_{Z^{1/2}_1(T,\I)} \le \e/2.}

Let $M=2(\|u_+\|_{H^{1/2}}+\|N_+\|_2)$, then there are $\io,\e\in(0,1)$ satisfying the first line of \eqref{small for ScatStz}, while the second line holds for large $T>0$ by the free Strichartz estimate and Lemma \ref{lem:scdec}. 
Then there exists a unique fixed point $(u,N)=\Phi_\io(u,N)$ in the above set \eqref{contraction set}. 
Proposition \ref{prop:sc} and Lemma \ref{lem:scdec} imply that if $(u,N)$ scatters then it should belong to the set \eqref{contraction set} for large $T$.  
Thus we have obtained the unique existence part. 

For the map continuity in $H^s\times L^2$, let $(u^k_+,N^k_+)\to(u_+,N_+)$ in $H^s\times L^2$ and let $(u_k,N_k)$ be the corresponding scattering solutions obtained above. Then for large $k$ and small $\io>0$, we have 
\EQ{
 (u_k,N_k)=\Uf(t)(u^k_+-u_+,N^k_+-N_+)+\Phi_\io(u_k,N_k).}
\eqref{Duh est sum dif} yields 
\EQ{ \label{contract at inf}
 \pt\|\Phi_\io(u_k,N_k)-\Phi_\io(u,N)\|_{\cZ^s(T,\I)}
 \prq\lec (\io^{1-s}M^2+\io^{-s}\e)\|(u_k,N_k)-(u,N)\|_{\cZ^s(T,\I)},}
provided that $0<\e<1<M<\I$ and both $(u,N)$ and $(u_k,N_k)$ satisfy 
\EQ{
 \pt \|(u,N)\|_{Z^s_0(T,\I)} \le M, \pq \|(u,N)\|_{Z^s_1(T,\I)} \le \e.}
Since Proposition \ref{prop:sc} implies that the $Z^s_1(T,\I)$ norms decay as $T\to\I$ and  
\EQ{
 \limsup_{T\to\I} \|(u_k,N_k)-(u,N)\|_{\cZ^s(T,\I)} \le \|u^k_+-u_+\|_{H^s}+ \|N^k_+-N_+\|_2,}
we deduce that for large $k$, we can choose large $M,T>0$ and small $\e,\io>0$, all independent of $k$, such that the above estimates hold with $\io^{1-s}M^2+\io^{-s}\e\ll 1$. Then \eqref{contract at inf} implies 
\EQ{
 \|(u_k,N_k)-(u,N)\|_{\cZ^s(T,\I)}\lec \|(u^k_+,N^k_+)-(u_+,N_+)\|_{H^s\times L^2} \to 0}
as $k\to\I$, hence the map continuity in $H^s\times L^2$. 
\end{proof}

To prove the case of the energy space, the last missing piece is the existence, which can be proved by solving from a finite time tending to infinity. 
\begin{lem} \label{lem:scexist}
Let $s\in[1/2,1]$. 
For any sequence $(\fy_k,\psi_k)\to(\fy,\psi)$ in $H^s\times L^2$ and any sequence $T_k\to\I$, let $(u_k,N_k)$ be the sequence of solutions to \eqref{Zak} in $H^s\times L^2$ with the initial data 
$(u_k(T_k),N_k(T_k))=\Uf(T_k)(\fy_k,\psi_k)$. 
Then there is $T\in(0,\I)$ such that $(u_k,N_k)$ converges in $L^\I((T,\I);H^s\times L^2)$ to the solution $(u,N)$ of \eqref{Zak} scattering in $H^s\times L^2$ with the profile $(\fy,\psi)$.
\end{lem}
\begin{proof}
Let $(u_k^0,N_k^0):=\Uf(t)(\fy_k,\psi_k)$ and $(u^0,N^0):=\Uf(t)(\fy,\psi)$. 
The Duhamel formula in the normal form reads 
\EQ{
 (u_k,N_k) \pt= (u_k^1,N_k^1)+ \vec\Om_\io(u_k,N_k)+\D_\io^{T_k}(u_k,N_k),}
with 
$(u_k^1,N_k^1):=(u_k^0,N_k^0)-\Uf(t-T_k)\vec\Om_\io(u_k^0,N_k^0)(T_k)$. 

Suppose $s<1$. Then \eqref{Duh est sum} yields, for any interval $I$ containing $T_k$ and contained in the maximal existence time of $(u_k,N_k)$, 
\EQ{ \label{diff T scat}
 \pt\|(u_k,N_k)-(u_k^1,N_k^1)\|_{\cZ^s(I)} 
 \pr\lec \BR{\io^{1-s}\LR{\|(u_k,N_k)\|_{\cZ^s(I)}}^2
  +\io^{-s}\|u_k\|_{X^s_1(I)}}\|(u_k,N_k)\|_{Z^s_1(I)}.}
Let $M:=1+\|(\fy,\psi)\|_{H^s\times L^2}$. Fix $\io>0$ small enough so that $\io^{1-s}M^2\ll 1$, then let $\e>0$ small enough so that $\io^{-s}\e\ll 1$. 
Since $(u_k^0,N_k^0)-(u^0,N^0)\to 0$ in $\cZ^s(\R)$ and $u^0(T_k)\to 0$ in $L^{2(-1/2)}$, Lemma \ref{lem:normal inv} implies that $(u_k^1,N_k^1)-(u^0,N^0)\to 0$ in $\cZ^s(\R)$. Hence by Lemma \ref{lem:scdec}, there exists $T_0>0$ such that for large $k$ we have $T_0<T_k$, 
$\|(u_k^1,N_k^1)\|_{Z^s_1(T,\I)}\le\e$ and $\|(u_k^0,N_k^0)\|_{Z^s_0(T_0,\I)}\le M$.  
If the interval $I$ satisfies
\EQ{ \label{T-sc bd}
 \|(u_k,N_k)\|_{Z^s_1(I)}\le 2\e, \pq \|(u_k,N_k)\|_{Z^s_0(I)}\le 2M,} 
we obtain from \eqref{diff T scat}
\EQ{ \label{T-sc cv}
 \pt \|(u_k,N_k)-(u_k^1,N_k^1)\|_{\cZ^s(I)} \lec (\io^{1-s}M^2+\io^{-s}\e)\e \ll \e <M.}
Hence by continuity of the norms on $I$, we deduce that \eqref{T-sc bd} and \eqref{T-sc cv} hold with $I=(T_0,\I)$ and the solution $(u_k,N_k)$ is defined on $(T_0,\I)$, for large $k$. In particular, by Proposition \ref{prop:sc}, $(u_k,N_k)$ is scattering in $H^s\times L^2$ with some profile $(\fy_k',\psi_k')$. Then \eqref{T-sc cv}, together with the decay of normal form, implies 
\EQ{
 \|(\fy_k',\psi_k')-(\fy_k,\psi_k)\|_{H^s\times L^2} \lec \e+\|\vec\Om_\io(u^0_k,N^0_k)(T_k)\|_{H^s\times L^2} \lec \e}
for large $k$. Taking $\e\to 0$ implies that $(\fy_k',\psi_k')\to(\fy,\psi)$ as $k\to\I$. 
Hence the continuity of the wave operator in Proposition \ref{prop:wo} implies that $(u_k,N_k)\to(u,N)$ in $\cZ^s([T,\I))$ for some $T<\I$. 

Next consider the case of $s=1$. Then Proposition \ref{prop:persist} implies $(u_k,N_k)$ is a solution in $H^1\times L^2$ on $[T_0,\I)$ for large $k$. 
The above argument, together with the wellposedness in $H^s\times L^2$, implies that $(u_k,N_k)\to(u,N)$ in $\cZ^s([T_0,\I))$ for all $s\in[1/2,1)$. 
Let $s\in(1/2,1)$. For any $\e>0$, there exists $\ti\psi\in H^{2-2s}$ such that $\|\psi-\ti\psi\|_2<\e$. 
Since $(u,N)$ is scattering in $H^s\times L^2$ with the profile $(\fy,\psi)$, and $(u_k,N_k)\to(u,N)$ in $\cZ^s([T_0,\I))$, there exists $R\in(T_0,\I)$ such that for all $k\ge R$ and $t\ge R$ we have 
\EQ{
 \|(u_k(t),N_k(t))-\Uf(t)(\fy,\psi)\|_{H^s\times L^2}+\|\ti\psi\|_{H^{2-2s}}\|u_k(t)\|_{2(-s)}^2<\e,}
where the second term is dominated by the first and the decay of $e^{-it\De}\fy$ in $L^{2(-s)}$. 
Then by H\"older and Sobolev,  we obtain 
\EQ{
 |\LR{N_k(t)||u_k(t)|^2}| \pt\le \|N_k(t)-e^{itD}\ti\psi\|_2\|u_k(t)\|_4^2+\|e^{itD}\ti\psi\|_{2(2-2s)}\|u_k(t)\|_{2(-s)}^2
 \pr\lec \e\|\na u_k(t)\|_2^2 + \e,}
which implies that the energy norm is bounded by the conserved energy, if $\e>0$ is small enough: 
\EQ{
 (1-C\e)\|\na u_k(t)\|_2^2 + \|N_k(t)\|_2^2/2 - C\e \pt\le 2E_Z(u_k,N_k),}
and the right side is convergent as $k\to\I$:  
\EQ{
 E_Z(u_k,N_k)= E_Z(\Uf(T_k)(\fy_k,\psi_k)) \to \|\na\fy\|_2^2/2+\|\psi\|_2^2/4.}
Thus $(u_k,N_k)$ is bounded in $H^1\times L^2$ uniformly in $t\ge 0$ and $k\in\N$. 
Hence using the equation, we see that $(u_k,N_k)\to(u,N)$ holds in $C([0,\I);\weak{H^1\times L^2})$. In particular, $(u,N)$ is a solution in $H^1\times L^2$ with conserved energy. The lower semi-continuity of the energy implies  
\EQ{
 E_Z(u(t),N(t)) \le \liminf_{k\to\I}E_Z(u_k,N_k) = \|\na\fy\|_2^2/2+\|\psi\|_2^2/4.}
Since $(u,N)$ is scattering in $H^s\times L^2$ and bounded in $H^1\times L^2$, a similar argument yields 
\EQ{
 \liminf_{t\to\I} E_Z(u(t),N(t)) \ge \|\na\fy\|_2^2/4+\|\psi\|_2^2/4.}
Hence these must be equality, which implies that the convergence is strong in $H^1$. 
\end{proof}

Now we are ready to prove in the case of energy space. 
\begin{proof}[Proof of Propositions \ref{prop:sc} and \ref{prop:wo} for $s=1$]
For Proposition \ref{prop:sc}, it remains only to prove \eqref{it:u-St}$\implies$\eqref{it:sc} for $s=1$. 
Let $(u,N)$ be a solution in $H^1\times L^2$ on $[0,\I)$ satisfying \eqref{it:u-St}. 
Since we have proven the case $s<1$, we know that $(u,N)$ is scattering with some profile $(\fy,\psi)$ in $H^s\times L^2$ for all $s<1$. 
Then the same argument as in the above lemma implies that $(u,N)$ is bounded in $H^1\times L^2$ for $t\ge 0$, and so $(\fy,\psi)\in H^1\times L^2$. Lemma \ref{lem:scexist} implies existence of a scattering solution in $H^1\times L^2$ with the profile $(\fy,\psi)$, which must be $(u,N)$ by the uniqueness in Proposition \ref{prop:wo} for $s=1/2$. 

For Proposition \ref{prop:wo}, the uniqueness is already shown in the best case $s=1/2$, while the existence and continuity follows from Lemma \ref{lem:scexist} for all $s\in[1/2,1]$. 
More precisely, the uniform existence time over $B$ follows from the lemma with $s=1/2$, since Corollary \ref{cor:perwo} ensures that the existence time is independent of $s$. 
The time interval for the continuity in $H^s\times L^2$ can be initially shorter in the result by the lemma, but it is extended to the uniform existence interval by the local wellposedness. 
\end{proof}

\section{Global wellposedness in the energy space below the ground state} \label{sect:GWP}
In the radial symmetry, our main Strichartz estimate (Theorem \ref{thm:Stz}) implies a stronger version of wellposedness, which is the same as in the subcritical problems, {\it as long as $\|N(t)\|_2$ stays below the threshold}. Recall that we are considering solutions of \eqref{Zak} in the sense of Definition \ref{def:sol}. 
\begin{thm} \label{thm:SLWP}
Let $s\in(1/2,1]$, $B\in(0,\|W^2\|_2)$ and $M\in(0,\I)$. 
Then there exist $T=T(s,B,M)>0$ and $C=C(s,B)>0$ such that for each $(\fy,\psi)\in H^s\rad(\R^4) \times L^2\rad(\R^4)$ with $\|\fy\|_{H^s}\le M$ and $\|\psi\|_{L^2}\le B$, the Zakharov system \eqref{Zak} has a unique local solution $(u,N)\in \cZ^1([0,T])$, which satisfies $\|u\|_{L^2_tB^s_{4,2}}\le CM$. 
\end{thm}
As an immediate consequence, we obtain the following blow-up criterion. 
\begin{cor} \label{cor:bc}
Let $s\in(1/2,1]$. 
For any initial data $(u(0),N(0))\in H^s\rad(\R^4)\times L^2\rad(\R^4)$, let $T^*\in(0,\I]$ be the maximal existence time of the solution $(u,N)$ of \eqref{Zak} in $H^s\times L^2$. If $T^*$ is finite, then we have one of the following:
\begin{enumerate}
\item $\Lim_{t\to T^*-0}\|u(t)\|_{\dot H^s}=\I$. 
\item $\Liminf_{t\to T^*-0}\|N(t)\|_2\ge\|W^2\|_2$.
\end{enumerate}
\end{cor}

Since the above results are stronger for smaller $s$ (thanks to the local wellposedness with regularity persistence in \cite{Zak4D1}), we need not consider the energy space. 
The global existence in Theorem \ref{main} follows from the above criterion with the variational estimates in the previous section.

Proposition \ref{prop:bc0} reduces the proof of Theorem \ref{thm:SLWP} to a priori bound on $\|u\|_{L^2_tB^{1/2}_{4,2}}$ on some short interval determined by $s,B,M$. 
\begin{proof}[Proof of Theorem \ref{thm:SLWP}]
Let $s\in(1/2,1)$ and let $(u,N)$ be a local solution from $t=0$ satisfying the initial size restriction of the theorem and with maximal existence time $T^*>0$. 
The Duhamel formula in the normal form yields a decomposition of $N$
\EQ{ \label{decop N}
 \pt N=N\sF+N\sN+N\sD,
 \pr N\sN:=D\ti\Om_\io(u,\ba u), \pq N\sF:=e^{itD}(N(0)-N\sN(0)),
 \pq N\sD:=\cN_\io^0(u,N).}
Since $\|N(0)\|_2\le B<\|W^2\|_2$ and $\|N\sN(0)\|_2\lec\io^{3/2}\|u(0)\|_{H^{1/2}}^2\le \io^{3/2}M^2$ by \eqref{uu est1} with $s=1/2$, choosing $\io>0$ small enough, depending on $B,M$, ensures that 
\EQ{
 \|N\sF\|_{L^\I_tL^2}=\|N\sF(0)\|_2 \le \ti B:=(B+\|W^2\|_2)/2<\|W^2\|_2.}
Suppose that, abbreviating $C\fn:=C\fn(s,0,\ti B,1)$, 
\EQ{ \label{asm bd uN}
 \|\LR{D}^su\|_{\ti \X^0(0,T)} < (C\fn+1)M, \pq \|N\|_{L^\I_tL^2(0,T)} < \|W^2\|_2+1,}
for some $T\in(0,T^*)$. 
Since $(u,N)$ is continuous in $t$, the above is satisfied at least for small $T>0$, thanks to the initial constraint. 
Then \eqref{uu est1} with the Sobolev embedding yields 
\EQ{
 \pt \|N\sN(t)\|_{L^2} \lec \|N\sN(t)\|_{B^1_{4/3,1}} \lec \|N\sN(t)\|_{B^{2s}_{4/3,2}} \lec \io^{1-s}\|u(t)\|_{\dot H^{1/2}}^2,}
while \eqref{uu est2} and \eqref{uu est3} respectively yield for $G_\io(u,N)$ 
\EQ{
 \pt \|D\ti\Om_\io(nu,\bar u)\|_{L^1_t L^2(0,T)} + \|D\ti\Om_\io(u, n\bar u)\|_{L^1_t L^2(0,T)} \lec \io^{1/2}\|N\|_{L^\I_tL^2}\|u\|_{L^2_tB^{1/2}_{4,2}}^2, 
 \pr\|D(u,\bar u)_{\HH_\io}\|_{L^1_tL^2(0,T)} \lec T^{s-1/2}\io^{s-1}\|u\|_{L^2_t B^s_{4,2}(0,T)}\|u\|_{L^{1(-s)}_t\dot B^{1-s}_{4,\I}(0,T)}
 \prQ\lec T^{s-1/2}\io^{s-1}\|u\|_{L^2_s B^s_{4,2}(0,T)}^{3-2s}\|u\|_{L^\I_t H^s(0,T)}^{2s-1},}
where the last inequality used the complex interpolation: 
\EQ{ 
  L^{1(-s)}_tB^{1-s}_{4,\I} \supset L^{1(-s)}_t B^s_{4(2s-1),2} = [L^2_t B^s_{4,2},L^\I_t H^s]_{2s-1}.} 
Hence, choosing $\io>0$ and then $T>0$ small enough, depending on $s,B,M,\io$, ensures
\EQ{
 \pt \|N\sN\|_{L^\I_t L^2(0,T)} \lec \|N\sN\|_{L^\I_t B^1_{4/3,1}(0,T)} \ll \e\fn:=\e\fn(s,0,\ti B,1)<1,
  \pr \|N\sD\|_{L^\I_t L^2(0,T)} \le \|G_\io(u,N)\|_{L^1_tL^2(0,T)} \ll \e\fn,}
where the smallness conditions on $\io,T$ can be written in the form 
\EQ{
 0<\io \le \io_0(s,B,M), \pq 0<T\le T_0(s,B,M,\io),}
with some functions $\io_0,T_0$. 
Fix $\io=\io_0(s,B,M)$. Then Theorem \ref{thm:Stz} yields 
\EQ{ \label{a priori Hs bd}
 \|\LR{D}^su\|_{\ti \X^0(0,T)} \le C\fn\|u(0)\|_2 \le C\fn M}
provided that $T\in(0,T^*)$ satisfies \eqref{asm bd uN} and 
\EQ{
 0<T\le T_1(s,B,M):=T_0(s,B,M,\io_0(s,B,M)).}
Hence by continuity of those norms on $T$, the condition \eqref{asm bd uN} is preserved, as well as the a priori bound \eqref{a priori Hs bd}, as long as $T<\min(T^*,T_1(s,B,M))$. 

Then the blow-up criterion of Proposition \ref{prop:bc0} implies $T^*\ge T_1(s,B,M)$,  
together with \eqref{a priori Hs bd} for $T=T_1(s,B,M)$. 
\end{proof}

\section{Scattering below the ground state} \label{sect:sc}
In this section, we consider the scattering problem for the global solutions obtained in the previous section, namely the scattering in the radial energy space below the ground state.  
It seems difficult to develop the profile decomposition with our Strichartz estimate (Theorem \ref{thm:Stz}), where perturbation of the potential can not be handled as a source term, because of the derivative loss, while the normal form argument is not quite compatible with the potential. 
So we will follow the idea of Dodson and Murphy \cite{DM}, which is closer to the more classical argument in the defocusing case, without any global perturbation argument. 
The idea is that the virial-Morawetz type estimate by Ogawa-Tsutsumi is directly applicable to arbitrary global solutions below the ground state energy, thanks to its variational character. 

\subsection{Morawetz-Strauss criterion}
In order to use a priori weak decay such as the Morawetz estimate, we need a scattering criterion that requires only some weak and temporary decay on (sufficiently) long time intervals. 
\begin{prop} \label{prop:MS}
Let $s\in(7/8,1)$, $p\in[1,\I]$, $M\in(1,\I)$, and $(u,N)$ be a radial global solution of \eqref{Zak} satisfying 
\EQ{ \label{unif ap bd}
 \sup_{T\ge 0} \|u(T)\|_{H^s}+\|N(T)\|_2+\|u\|_{L^2_tB^s_{4,2}(T,T+1)} \le M.}
Then there exist $T\in(0,\I)$, $L_*(s,p,M)\in(1,\I)$ and $\e_*(s,p,M)\in(0,1)$ such that if 
$\|u\|_{L^\I_t L^p(I)}\le\e_*(s,p,M)$ on some interval $I\subset(T,\I)$ satisfying $|I|\ge L_*(s,p,M)$, 
then the solution $(u,N)$ scatters in $H^s\times L^2$. 
\end{prop}
\begin{rem}
(1) The uniform local Strichartz norm $\|u\|_{L^2_tB^s_{4,2}(T,T+1)}$ in \eqref{unif ap bd} reflects the criticality of $N\in L^2$. We do not know if it is bounded in general for all solutions globally bounded in $H^s\times L^2$,  but Theorem \ref{thm:SLWP} gives a sufficient condition. \\
(2) The range $s>7/8$ is not optimal. In fact, a slight modification of the proof below would extend it to $s>3/4$. For our use in the energy space, however, it suffices to have one $s\in(1/2,1)$. 
\end{rem}

The proof of Proposition \ref{prop:MS} consists of two steps: First we prove certain decay of a free solution starting after the temporal decay of the nonlinear solution. 
The decay is slightly weaker than the full Strichartz estimate, but for all time. 
Secondly, the decay of the free solution is transferred to the nonlinear solution, which also implies the scattering. 
Specifically using 
\EQ{ \label{def Z2}
 \pt X_2^{s,\de}:=L^\I_tL^{2(-s)}\cap L^2_tB^{s-\de}_{4(-\de),2}, 
 \pq Y_2^\de:=L^\I_t\dot B^{-\de}_{2(-\de),2}\cap L^4_t\dot B^{-1/4+\de}_{8/3(\de),2},
 \pr Z_2^{s,\de}:=X_2^{s,\de}\times Y_2^\de,}
the above two steps are respectively formulated as follows. 
\begin{lem}\label{lem:delay free dec}
Let $s,p,M,(u,N)$ be as in Proposition \ref{prop:MS}. For any $s'\in(0,s)$, $\de\in(0,1/6)$ and $\e\in(0,1)$, there exist $T\in(0,\I)$, $L_3=L_3(s,p,M,s',\de,\e)\in(1,\I)$ and $\e_3=\e_3(s,p,M,s',\de,\e)\in(0,1)$ such that if $\|u\|_{L^\I_tL^p(T_1,T_2)}\le\e_3$ for some $T_1,T_2\in(0,\I)$ satisfying $T_1\ge T$ and $T_2-T_1\ge L_0$, then 
$\|\Uf(t-T_2)(u,N)(T_2)\|_{Z_2^{s',\de}}\le\e$. 
\end{lem}
\begin{lem}\label{lem:Z2 scat}
Let $s,M,(u,N)$ be as in Proposition \ref{prop:MS}. For any $\de\in(0,\min(1/6,s,1-s))$, there exists $\e_4=\e_4(s,M,\de)\in(0,1)$ such that if $\|\Uf(t-T)(u,N)(T)\|_{Z_2^{s,\de}}=:\e\le\e_4$ for some $T\in(0,\I)$, then $\|(u,N)\|_{Z_2^{s,\de}(T,\I)}\le 2\e$ and $(u,N)$ scatters in $H^s\times L^2$. 
\end{lem}
Proposition \ref{prop:MS} follows immediately from the above lemmas: Fix $s'\in(1/2,s)$ and $\de\in(0,\min(1/6,s',1-s))$ to decide $\e:=\e_4(s',M,\de)$ by Lemma \ref{lem:Z2 scat}, and then $L_*:=L_3(s,p,M,s',\de,\e)$ and $\e_*:=\e_3(s,p,M,s',\de,\e)$ by Lemma \ref{lem:delay free dec}. Then the above two lemmas with $I=(T_1,T_2)$ imply the scattering of $(u,N)$ in $H^{s'}\times L^2$, which is upgraded to the scattering in $H^s\times L^2$ by Corollary \ref{cor:perwo}. It remains to prove the two lemmas.

\begin{proof}[Proof of Lemma \ref{lem:delay free dec}]
Let $0=:T_0<T_1<T_2$ and $L:=T_2-T_1$. 
Fix $\io\in(0,1)$, say $\io=1/2$. Let $(\ti u,\ti N):=\Psi_\io(u,N)$ and $(u_f,N_f):=\Uf(t-T_2)(\ti u,\ti N)(T_2)$. Then we can decompose the Duhamel formula as follows. 
\EQ{
 u_f \pt= e^{-i(t-T_2)\De}u(T_2) + u^3, \pQ u^3:=-e^{-i(t-T_2)\De}\Om_\io(u,N)(T_2),
 \pr=\sum_{j=0}^3 u^j, \pq u^0:=e^{-it\De}u(0),
 \pr u^j:=[e^{-i(t-\ta)\De}u(\ta)]_{\ta=T_{j-1}}^{\ta=T_j} = -i\int_{T_{j-1}}^{T_j}e^{-i(t-\ta)\De}(nu)(\ta)d\ta,\ (j=1,2).}
Using \eqref{scat dec}, we have $\|u^0\|_{X_1^s(T_1,\I)}\to 0$ as $T_1\to\I$, while \eqref{uN est1} with the free Strichartz yields 
\EQ{
 \|u^3\|_{X_1^{s'}(T_2,\I)} \pt\lec \|u^3(T_2)\|_{H^{s'}} 
 \pr\lec \|N(T_2)\|_2\|u(T_2)\|_{2(-s')}
 \le M\sum_{j=0}^2\|u^j(T_2)\|_{2(-s')}.}
Hence, using $X^s_1\subset X^{s,\de}_2$, and $s'<s$, the desired decay for $u_f$ follows from that for $u^1,u^2$, for which the free Strichartz estimate yields $\|u^j\|_{X_1^s}\lec\|u\|_{L^\I_tH^s}\le M$. 
The dispersive decay for $e^{-it\De}$ yields, for $t>T_1$, 
\EQ{
 \|u^1(t)\|_\I \pt\lec \int_0^{T_1}|t-\ta|^{-2}\|(nu)(\ta)\|_1d\ta
 \pr\lec |t-T_1|^{-1}\|N\|_{L^\I_tL^2}\|u\|_{L^\I_tL^2} \le |t-T_1|^{-1}M^2.}
Then using H\"odler, we obtain
\EQ{
 \|u^1\|_{L^\I_tL^{2(-s)}(T_2,\I)} \le \sup_{t\ge T_2}\|u^1(t)\|_2^{1-s/2}\|u^1(t)\|_\I^{s/2} \lec L^{-s/2}M^{1+s/2},}
and, using $[B^s_{4,2},L^\I]_\de \subset B^{s-s\de}_{4(-\de),\I}\subset B^{s-\de}_{4(-\de),2}$, 
\EQ{
 \|u^1\|_{L^2_tB^{s-\de}_{4(\de),2}(T_2,\I)}
 \lec \|u^1\|_{L^2_tB^s_{4,2}}^{1-\de}\|u^1\|_{L^2_tL^\I(T_2,\I)}^\de 
 \lec L^{-\de/2}M^{1+\de/2}.}

Next for $u^2$, since $s>s'>0>s'-1$, there exist $\te_1,\te_2\in(0,1)$, depending only on $s,s',p$, such that by interpolation 
\EQ{ \label{interp u^2}
 \pt \|u(t)\|_{2(-s')} \lec \|u(t)\|_{H^s}^{1-\te_1}\|u(t)\|_p^{\te_1} \lec M^{1-\te_1}\|u(t)\|_p^{\te_1},}
and
\EQ{
 \|u^2\|_{X^{s'}_1} \lec \|u^2\|_{X^s_1}^{1-\te_2}\|u^2\|_{X^{s'-1}_1}^{\te_2} \lec M^{1-\te_2}\|u^2\|_{X^{s'-1}_1}.}
The last norm is estimated using the free Strichartz estimate and $B^{s'-1}_{4/3,2}\supset L^{2(2-s')}$
\EQ{
 \|u^2\|_{X^{s'-1}_1} \pt\lec \|nu\|_{L^2_tL^{2(2-s')}(T_1,T_2)}
 \lec L^{1/2}\|N\|_{L^\I_tL^2(T_1,T_2)}\|u\|_{L^\I_tL^{2(-s')}(T_1,T_2)}
 \pr\lec L^{1/2}M\|u\|_{L^\I_tH^s(T_1,T_2)}^{1-\te_2}\|u\|_{L^\I_tL^p(T_1,T_2)}^{\te_1} \lec L^{1/2}M^{2-\te_1}\e_3^{\te_2},}
where in the second last step \eqref{interp u^2} was used. 
Hence taking $L_3$ large enough and then $\e_3$ small enough, depending on $s,M,\de,s',p$, we can make both $u^2$ and $u^3$ as small as we wish, and so $u_f$. 

For $N_f$, we can not rely so much on interpolation to recover regularity as for $u_f$, so the Duhamel formula should be decomposed in the normal form: 
\EQ{
 N_f \pt=\sum_{j=0}^2\ti N^j, \pq \ti N^0:=e^{itD}\ti N(0),
 \prq N^j:=[e^{i(t-\ta)D}\ti N(\ta)]_{\ta=T_{j-1}}^{\ta=T_j}=-i\int_{T_{j-1}}^{T_j}e^{i(t-\ta)D}G_\io(u,N)(\ta)d\ta.}
In a way similar to $u_f$, the radially improved Strichartz estimate 
\EQ{ \label{rad Str wave}
 2/p+6/q<3, \pq 1/p+4/q-\si=2 \implies \|e^{itD}\fy\|_{L^p_t\dot B^\si_{q,2}} \lec \|\fy\|_2}
implies 
$\|\ti N^0\|_{Y_2^\de(T_1,\I)}\to 0$ as $T_1\to\I$, and, using \eqref{uu est4.1} as well, 
\EQ{
 \|\ti N^1\|_{L^\I_tL^2\cap L^4_t\dot B^{-1/4+\be}_{8/3(\be),2}}\lec\|\ti N\|_{L^\I_tL^2} \lec \|N\|_{L^\I_tL^2}+\|u\|_{L^\I_tL^{2(-1/2)}}^2
 \lec M^2,}
for any $\be\in(\de,1/6)$. On the other hand, the dispersive decay estimate 
\EQ{
 \|e^{itD}\fy\|_{\dot B^{-5/2}_{\I,2}} \lec |t|^{-3/2}\|\fy\|_{\dot B^0_{1,2}} }
yields, for $t>T_1$ and $\si\in(1/2,1)$, 
\EQ{
 \|\ti N^1(t)\|_{\dot B^{2\si-7/2}_{\I,2}}
 \pt\lec \int_0^{T_1}|t-\ta|^{-3/2}\|G_\io(u,N)(\ta)\|_{\dot B^{2\si-1}_{1,2}}d\ta
 \pr\lec |t-T_1|^{-1/2}\|G_\io(u,N)\|_{L^\I_t\dot B^{2\si-1}_{1,2}(0,T_1)}.}
Note that this decay order $t^{-1/2}$ prevents us from using $L^2_t$ norm for $N_f$. 
For the above norm on $G_\io$, we use product estimates  
\EQ{
 \|D(u,\ba u)_{\HH_\io}\|_{\dot B^{2\si-1}_{1,2}} \lec \|u\|_{H^\si}^2,}
which follows from the same argument as for \eqref{uN est5}, and 
\EQ{
 \|D\ti\Om_\io(nu,\ba u)+D\ti\Om_\io(u,n\ba u)\|_{\dot B^{2\si-1}_{1,2}} \lec \|nu\|_{1(-\si)}\|u\|_{\dot B^0_{2(-\si),2}}
 \lec \|N\|_2\|u\|_{\dot B^0_{2(-\si),2}}^2,}
which follows from \eqref{Om est0} with $a=2$ and $\{b,c\}=\{2-\si,-\si\}$. 
Integrating in time, we obtain for $2<q\le\I$ and $\si\in(1/2,s]$
\EQ{
 \|\ti N^1\|_{L^q_t\dot B^{2\si-7/2}_{\I,2}(T_2,\I)}
 \lec L^{-1/2+1/q}\LR{\|N\|_{L^\I_tL^2}}\|u\|_{L^\I_tH^\si}^2 \lec L^{-1/2+1/q}M^3.}
Then choosing $(\si,p)=(3/4,\I),(7/8,4)$, we obtain decay of $\ti N^1$ as $L\to\I$ in $(L^\I_t\dot B^{-2}_{\I,2}\cap L^4_t\dot B^{-7/4}_{\I,2})(T_2,\I)$, which is transferred to $Y_2^\de$ by interpolation: 
\EQ{
 [L^2,\dot B^{-2}_{\I,2}]_{\de/2}=\dot B^{-\de}_{2(-\de),2},
 \pq [\dot B^{-1/4+\be}_{8/3(\be),2},\dot B^{-7/4}_{\I,2}]_\te=\dot B^{-1/4+\de}_{8/3(\de),2},}
where $\te\in(0,1)$ is chosen such that $\de=(1-\te)\be-3\te/2$. 

The radially improved Strichartz \eqref{rad Str wave} and the Duhamel estimate in \eqref{Duh est N} yield 
\EQ{
 \|\ti N^2\|_{Y_2^\de} \lec \|G_\io(u,N)\|_{L^1_tL^2(T_1,T_2)}
 \lec M\|u\|_{L^2_tB^{1/2}_{4,2}(T_1,T_2)},}
where the last norm is bounded by H\"older in time, $B^{1/2}_{4,2}=(B^{s'}_{4,2},B^{s'-1}_{4,\I})_{\te,2}$ with $\te_3:=s'-1/2\in(0,1/2)$ and $L^{2(-s')}\subset B^{s'-1}_{4,\I}$, 
\EQ{
 \|u\|_{L^2_tB^{1/2}_{4,2}(T_1,T_2)}
 \pt\lec L^{\te_3/2}\|u\|_{L^{2/(1-\te_3)}_tB^{1/2}_{4,2}(T_1,T_2)}
 \pr\lec L^{\te_3/2}\|u\|_{L^2_tB^{s'}_{4,2}(T_1,T_2)}^{1-\te_3}\|u\|_{L^\I_tL^{2(-s')}(T_1,T_2)}^{\te_3}
 \lec LM^{1-\te_1}\e_3^{\te\te_1},}
where in the last step \eqref{interp u^2} was used. 
Thus, taking $L_3$ large enough and then $\e_3$ small enough, we can make both $\ti N^2$ and $\ti N^3$ small, and so $N_f$. 
\end{proof}

\begin{proof}[Proof of Lemma \ref{lem:Z2 scat}]
This is by standard perturbation argument, using estimates similar to the previous ones. 
Fix $\io\in(0,1)$, say $\io=1/2$. 

For the normal for of $u$, we have, on any time interval,  
\EQ{
 \|\Om_\io(u,N)\|_{X_2^{s,\de}} \lec \|N\|_{L^\I_t\dot B^{-\de}_{2(-\de),\I}}\|u\|_{X_2^{s,\de}},}
for $0<\de<s<1$, by \eqref{proest4} with 
\EQ{
 (s,\te_1,\te_2,\de_1,\de_2)\to (s-\de,1-\de,-\de,-\de,0),(0,2-s-\de,-\de,-\de,\de),} 
together with $B^0_{4(-\de),2}\subset\dot B^{\de-1}_{\I,\I}$, $B^{s-\te+1}_{2(-\de),2}\subset B^{s-\te}_{4(-\de),2}$ and $L^{2(-s)}\subset\dot B^{s+\de-2}_{\I(\de),\I}$. 
For the Duhamel form of $u$, we have 
\EQ{
 \|\LR{D}^s(n,u)_{\LH_\io}\|_{L^2_t\dot B^{-2\de}_{4/3(-2\de),2}+L^{4/3}_tB^0_{8/5,2}}
 \lec \|u\|_{X_2^{s,\de}}\|N\|_{Y_2^\de},}
which is proven by decomposing $u=u_{\le 1}+u_{>1}$. For the low frequency part, we have 
\EQ{
 \|(n,u_{\le 1})_{\LH_\io}\|_{8/5} \lec \|n_{\lec 1/\io}\|_{8/3(\de)}\|u\|_{4(-\de)}
 \lec \|N\|_{\dot B^{-1/4+\de}_{8/3(\de),2}} \|u\|_{B^{s-\de}_{4(-\de),2}},}
and for the high frequency part, \eqref{proest5} with $w(r)=\LR{r}$ and $\de\to2\de$ yields 
\EQ{
 \|\LR{D}^s(n,u_{>1})_{\LH_\io}\|_{\dot B^{-2\de}_{4/3(-2\de),2}}
 \lec \|N\|_{\dot B^{-\de}_{2(-\de),\I}}\|u_{>1}\|_{\dot B^{s-\de}_{4(-\de),2}}. }
By \eqref{uN est3}, we have 
\EQ{
 \|\Om_\io(N,nu)\|_{L^2_tB^s_{4/3,2}} \pt\lec \|N\|_{L^\I_t\dot B^{-\de}_{2(-\de),\I}}\|N\|_{L^\I_tL^2}\|u\|_{L^2_tL^{4(-s)}}.}
and by \eqref{uN est4} with $a=1$, 
\EQ{
 \|\Om_\io(D|u|^2,u)\|_{L^2_tB^s_{4/3,2}}
 \lec \|u\|_{L^\I_tL^{2(-1/2)}}^2\|u\|_{L^2_tL^{4(-s)}}.}

For the normal form of $N$, choosing $(a,b,c)=(-1+2\de,-1-\de,-s)$ in \eqref{Om est0} yields
\EQ{
 \|D\ti\Om_\io(u,\ba u)\|_{\dot B^{s-1+3\de}_{4(2\de),2}}
 \lec \|u\|_{4(-\de)}\|u\|_{2(-s)}.}
Since $s\ge 1/2$ and $\ti\Om_\io$ is restricted to the frequency $\gec 1/\io$, the Besov norm on the left may be replaced with $\dot B^{-1/2+2\de}_{4(2\de),2}$. 
Combining it with \eqref{uu est4.1} and $[L^\I_tL^2,L^2_t\dot B^{-1/2+2\de}_{4(2\de),2}]_{1/2}=L^4_t\dot B^{-1/4+\de}_{8/3(\de),2}$, we obtain 
\EQ{
 \|D\ti\Om_\io(u,\ba u)\|_{Y_2^\de} \lec M^{1/2}\|u\|_{X^{s,\de}_2}^{3/2}.}
For the Duhamel form of $N$, we have 
\EQ{
 \|D(u,\ba u)_{\HH_\io}\|_{\dot B^{1/2-2\de}_{4/3(-2\de),2}}
 \lec \|u\|_{B^{s-\de}_{4(-\de),2}}\|u\|_{B^{s-\de}_{2(-\de),\I}},}
which is proven in the same way as \eqref{uN est5}, using $2s-1\ge 1/2$ for high frequency output. 
Using interpolation $[H^s,L^{2(-s)}]_{\de/s}=H^{s-\de}_{2(-\de)}$ for the last norm, we obtain  
\EQ{
 \|D(u,\ba u)_{\HH_\io}\|_{L^2_t\dot B^{1/2-2\de}_{4/3(-2\de),2}} 
 \lec M^{1-\de}\|u\|_{L^2_t B^{s-\de}_{4(-\de),2}}\|u\|_{L^\I_t L^{2(-s)}}^\de.} 
By \eqref{uu est2} and $B^{s-\de}_{4(-\de),2}\subset L^{4(-s)}\cap L^{4(s-1)}$ (using $\de\le 1-s$), we have
\EQ{
 \|D\ti\Om_\io(nu,\ba u)+D\ti\Om_\io(u,n\ba u)\|_{L^1_tL^2} \lec \|N\|_{L^\I_tL^2}\|u\|_{L^2_tB^{s-\de}_{4(-\de),2}}^2.}

Plugging the above estimates into the radial Strichartz estimates yields (using $2\de<1/3<\de_\star$ in $d=4$)
\EQ{
 \pt \|u-u_f\|_{X_2^{s,\de}(T,T')}
 \lec M\|u\|_{X_2^{s,\de}(T,T')}\|N\|_{Y_2^\de(T,T')},
 \pr \|N-N_f\|_{Y_2^\de(T,T')} \lec M^{2-\de}\|u\|_{X_2^{s,\de}(T,T')}^{1+\de},}
which are uniform for $T'\in(T,\I)$, where $(u_f,N_f):=\Uf(t-T)(u,N)(T)$. 
In fact, choosing $\e_4$ small ensures that $M^{2-\de}\e^{1+\de}\ll\e$, so that we can deduce from $\|(u,N)\|_{Z_2^{s,\de}(T,T')}\le 2\e$ that $\|(u,N)\|_{Z_2^{s,\de}(T,T')}\le 1.5\e$. 
Then by continuity of the norm in $T'$, we deduce that $\|(u,N)\|_{Z_2^{s,\de}(T,\I)}\le 2\e$, and using the radial Strichartz estimate again, the scattering of $(u,N)$ in $H^s\times L^2$. 
\end{proof}

By the radial Sobolev inequality and the mass propagation estimate, the temporal decay condition in Proposition \ref{prop:MS} is reduced to a local $L^2$ decay in limit inf. 
\begin{lem} \label{MS cri}
Let $s\in(7/8,1)$ and $M\in(1,\I)$, there exist $\e\in(0,1)$ and $R\in(1,\I)$ such that 
every radial global solution $(u,N)$ satisfying 
\EQ{
 \pt \sup_{T>0}\|(u(T),N(T))\|_{H^1\times L^2}+\|u\|_{L^2_tB^s_{4,2}(T,T+1)} \le M,  
 \pr \liminf_{t\to\I}\|u(t)\|_{L^2(|x|<R)} \le \e}
scatters in $H^1\times L^2$.  
\end{lem}
\begin{proof}
Choose $\chi\in C_0^\I(\R^4)$ such that $\chi(x)=1$ for $|x|\le 1$, $\chi(x)=0$ for $|x|\ge 2$, and $|\na\chi|\le 2$, and let $\chi_R(x):=\chi(x/R)$ for $R>0$. Then 
\EQ{
 \p_t\LR{\chi_R u|u} \pt=2\LR{i\chi_R u|i\dot u}=2\LR{iu\chi_R|\De u}=2\LR{u\na\chi_R|i\na u}
  \le 4M^2/R.}
Hence for any $t_1,t_0\in(0,\I)$, we have
\EQ{ \label{propa est}
 \|u(t_1)\|_{L^2(|x|<R)}^2 \le \|u(t_0)\|_{L^2(|x|<2R)}^2 + 4M^2|t_1-t_2|/R.}
On the other hand, Proposition \ref{prop:MS} yields $\e_0=\e_0(M,s)\in(0,1)$, $L=L(M,s)\in(0,1)$ and $T\in(0,\I)$ such that if for some $S\in(T,\I)$ we have 
\EQ{ \label{scc L3}
 \|u\|_{L^\I_t L^3(S,S+L)} \le \e_0,}
then $(u,N)$ scatters in $H^s\times L^2$. The radial Sobolev inequality 
\EQ{
 \fy\in H^1\rad(\R^4) \pt\implies \sup_{x\in\R^4}|x|^{3/2}|\fy(x)| \lec \|\fy\|_{H^1(\R^4)}
 \pr\implies \|\fy\|_{L^3(|x|>R)} \lec R^{-1/2}\|\fy\|_{H^1(\R^4)} }
implies 
\EQ{ \label{ext L3}
 \|u\|_{L^\I_t L^3(S<t<S+L,\ |x|>R)} \lec R^{-1/2}M,}
while H\"older and Sobolev yield
\EQ{ \label{int L3}
 \|u\|_{L^\I_t L^3(S<t<S+L,\ |x|<R)} \lec \|u\|_{L^\I_t L^2(S<t<S+L,\ |x|<R)}^{1/3}M^{2/3}.}
Choose $0<\e=\e(M,s)<1$ small enough such that 
\EQ{
  \e^{1/3}M^{2/3} \ll \e_0, }
and then choose $1<R=R(M,s)<\I$ large enough such that 
\EQ{
 M^2L/R \ll \e^2, \pq R^{-1/2}M \ll \e_0.}
If there exists $t_0\in(T,\I)$ such that $\|u(t_0)\|_{L^2(|x|<2R)}<\e$, then by \eqref{propa est} we deduce that $\sup_{t_0<t<t_0+L}\|u(t)\|_{L^2(|x|<R)}<2\e$. Then by \eqref{ext L3} and \eqref{int L3}, we obtain \eqref{scc L3}, so $(u,N)$ scatters in $H^s\times L^2$, and also in $H^1\times L^2$, by Corollary \ref{cor:perwo}.
\end{proof}

\subsection{Virial-Morawetz estimate}
Here we derive a virial-Morawetz type estimate for the Zakharov system similar to that for NLS by Ogawa and Tsutsumi \cite{OT}, for smooth solutions $(u,N)$ on $\R^d$. 
Specifically, the estimate is obtained from \cite[Section 3.1]{GNW}, choosing $\psi=\LR{r}^{-1}$ instead of the cut-off function. Let 
\EQ{ \label{def psiR As}
 \psi_R:=\psi(x/R)=\LR{r/R}^{-1}, \pq A_s:=x\cdot\na+(d+s)/2}
for $R>0$ and $s\in\R$. 
Then a localized virial quantity with a scaling parameter $R>0$ is defined as in \cite{GNW} by 
\EQ{ \label{def VR}
 \V_R:=\LR{u|i(A_0\psi_R + \psi_R A_0)u}+\frac12\LR{D^{-1}N|i(A_1\psi_R+\psi_R A_1)N},}
which satisfies the following identity, cf.~\cite[Section 3.1]{GNW}. 
It holds for general radial $\psi$ and general (non-radial) solutions. 
\EQ{ \label{def NS}
 \pt \frac{d}{dt} \V_R = NS + QN + CC,
 \prq NS:=2\LR{E_S'(u)|(A_0\psi_R + \psi_R A_0)u}
 \prQ =4\LR{\na u|\na u\psi_R}+4\LR{u_r|u_r r\p_r\psi_R}
  -\LR{|u|^4|A_d\psi_R}-\LR{u|u A_{d+4}\De\psi_R},
 \prq QN:=\frac12\LR{\nu|\psi_R\nu}+\frac12\LR{\na\y|\psi_R\na\y}+\LR{\y_r|\y_r r\p_r\psi_R}-\frac14\LR{\y|\y A_{d+2}\De\psi_R},
 \prq CC:=-\LR{\nu u |uA_{d-2}\psi_R} + CC_3',
 \prQ CC_3':=\LR{\nu|\{\psi_R\}A_1+\{r\p_r\psi_R\}/2)|u|^2},}
where 
\EQ{ \label{def nu}
 \nu:=N-|u|^2,\pq \y:=D^{-1}\nu, \pq \{f\}:=DfD^{-1}-f.}
Without the weight $\psi_R$, we have the virial identity \cite[Lemma 2.1]{GNW}:
\EQ{
 \pt \V_\I :=2\LR{u|iA_0 u}+\LR{D^{-1}N|i A_1 N}
 \pr \implies \frac{d}{dt} \V_\I =4K(u)+\|\nu\|_2^2-(d-1)\LR{\nu u|u} =:\dot \V_\I(u,\nu).}
In order to rewrite the right side of \eqref{def NS}, the following weight functions $f_{j,R}:=f_j(x/R)$ with $j=0\etc 5$ are introduced for general $\psi$ and $R$: 
\EQ{ \label{def fj}
 \pt f_0:=\sqrt{(1+r\p_r)\psi}, \pq f_1:=-r\p_r\psi, 
 \pr f_2:=-A_{d+4}\De\psi+4f_0\De f_0, 
 \pq f_3:=A_d\psi-df_0^4,
 \pr f_4:=-\frac14 A_{d+2}\De\psi,
 \pq f_5:=A_{d-2}\psi-(d-1)f_0^3.}

\subsubsection{NLS part}
The NLS part is rewritten as follows, for general radial $\psi$, 
\EQ{
 NS=4K(f_{0,R}u)+\int_{\R^d}4|u_\te|^2f_{1,R}+|u/R|^2f_{2,R}-|u|^4f_{3,R}dx,}
using that $\|f\na u\|_2^2=\|\na fu\|_2^2+\LR{|u|^2|f\De f}$. 
For the specific $\psi=\LR{r}^{-1}$, the functions $f_j$ can be easily computed, using the identites $r\p_r\psi=-r^2\psi^3=\psi^3-\psi$ and $\De=r^{-2}(r\p_r+d-2)r\p_r$: 
\EQ{ 
 \pt f_0=\psi^{3/2}, \pq f_1=r^2\psi^3, 
 \pq -\De\psi=(d-3)\psi^3+3\psi^5,
 \pr f_2=(d-3)(d-1)\psi^3+3\psi^5-6\psi^7,
 \pr f_3=(d-1)\psi+\psi^3-d\psi^6.}
They are all positive for $d\ge 4$, and also for $d=3$ except $f_2$. 

The term $|u|^4f_{3,R}$ can be controlled on $\R^4$ in the radial case by the following $L^4$ propagation estimate. 
Let 
\EQ{ \label{def La}
 \La:=r^2/(1+r)^4, \pq \La_R:=\La(x/R).} 
Then, using the equation and $|\La_R|+|r\p_r\La_R|\lec(r/R)^2$, we have 
\EQ{
 \p_t\LR{|u|^4|\La_R}/4\pt=\LR{|u|^2u|\La_R\dot u}=\LR{i|u|^2u|\La_R\De u}
 \pr=\LR{|u|^2u\na\La_R|i\na u}+\LR{\La_R u^2\na \ba u|i\na u}
 \pr\lec R^{-2}(\|u/r\|_2+\|\na u\|_2)\|ru\|_\I^2\|\na u\|_2.}
Hence using the Hardy and the radial Sobolev inequalities, we obtain
\EQ{
 \sup_{0\le t\le T} \int_{\R^4}|u|^4\La_R dx \le \int_{\R^4}|u(0)|^4\La_R + C\frac{T}{R^2}\|\na u\|_{L^\I_tL^2(0,T)}^4.}
Interpolation with the Sobolev inequality yields
\EQ{ \label{est L4tail}
 \sup_{T_0\le t\le T_1}\LR{|u|^4|f_{3,R}} \lec \LR{|u(0)|^4|\La_R}^{1/2}\|\na u\|_{L^\I_tL^2}^2+R^{-1}T^{1/2}\|\na u\|_{L^\I_t L^2}^4.}

\subsubsection{Wave part}
Using a bilinear commutator:
\EQ{ \label{def be}
 \be_R(f,g)\pt:=\LR{h_RDf|Dg}-\LR{h_R\na f|\na g}, \pq h_R:=A_{d-1}\psi_R,}
the wave part can be rewritten as 
\EQ{
 QN=\|f_{0,R}\nu\|_2^2-\be_R(\y,\y)+\int_{\R^d}|\y_\te|^2f_{1,R}+|\y/R|^2f_{4,R}dx.}
For the specific $\psi=\LR{r}^{-1}$, we have 
\EQ{
 4f_4=(d-3)(d-2)\psi^3+3(2d-7)\psi^5+15\psi^7,}
which is positive for $d\ge 4$. Since
\EQ{
 \be_R(f,g)\sim\int_{\R^d\times\R^d}(|\x_1||\x_2|-\x_1\cdot\x_2)\F h_R(\x_1-\x_2)\hat f(\x_1)\ba{\hat g(\x_2)}d\x_1d\x_2}
and $||\x_1||\x_2|-\x_1\cdot\x_2| \lec |\x_1-\x_2|\min(|\x_1|,|\x_2|)$, we have 
\EQ{
 |\be_R(f,g)| \lec \|\F\na h_R\|_1\min(\|\na f\|_2\|g\|_2,\|f\|_2\|\na g\|_2).}
The norm for $h_R:=A_{d-1}\psi_R$ is bounded by $R^{-1}$, since 
\EQ{ \label{FL1 bd}
 \|\F\na h_R\|_1 = R^{-1} \|\F\na h_1\|_1 \lec R^{-1} \sum_{|\al|=d,d+2}\|\LR{x}^{1+d}\p^\al h_1\|_\I,}
and $|\p^\al\LR{r}^{-1}|\lec\LR{r}^{-1-|\al|}$. 
In order to use this decay, we decompose $\y$ in the frequency with a small parameter $\de\in(0,1)$ which depends on $R$: 
\EQ{ \label{decop eta-1}
 \y=\y_L+\y_H, \pq \y_L:=\y_{<\de}.}
Then, using $\|h_R\|_\I=\|h_1\|_\I\lec 1$ and $\|\x\F h_R\|_1\lec R^{-1}$, we have 
\EQ{
 |\be_R(\y,\y)| \pt\le |\be_R(\y_L,\y_L)|+|\be_R(\y_H,\y+\y_L)|
 \pr\lec \|\y_L\|_{\dot H^1}^2 + \|\y_H\|_2\|\na\y\|_2
 \pr\lec \|\nu_{<\de}\|_2^2 + (\de R)^{-1}\|\nu\|_2^2.}
To bound the low frequency $\y_L$ for long time, we use the Duhamel formula as in \cite{GNW}: 
\EQ{  \label{decop eta-2}
 \pt \y=\y^0+\y^2+\y^3+\y^4, \pq \y^j_L:=\y^j_{<\de},
 \pr \y^0:=e^{iDt}\y(0),
 \pq \y^2:=D^{-1}[e^{iDt}|u(0)|^2-|u(t)|^2], 
 \pr \y^3:=-i\int_0^{(t-1)_+}e^{iD(t-\ta)}|u(\ta)|^2d\ta, 
 \pq \y^4:=-i\int_{(t-1)_+}^t e^{iD(t-\ta)}|u(\ta)|^2d\ta,}
where $(t-1)_+:=\max(0,t-1)$. 
$\y^2_L$ and $\y^4_L$ are bounded in $L^2$ because 
\EQ{ \label{y24-L2bd}
 \||u|^2_{<1}\|_2 \lec \||u|^2\|_1 \lec \|u\|_2^2=\|u(0)\|_2^2.} 
Hence for $j=2,4$, 
\EQ{
 \|\be_h(\y^j_L,\y_L)| \lec R^{-1}\|\y^j_L\|_2\|\na\y_L\|_2 \lec R^{-1}\|u(0)\|_2^2\|\nu\|_2.}
By the dispersive decay estimate for $e^{iDt}$ we have, for $d\ge 4$, 
\EQ{
 \|\na \y^3_L\|_\I \pt\lec \int_0^{(t-1)_+}|t-\ta|^{-(d-1)/2}\de^{(d+3)/2}\||u(\ta)|^2\|_1 d\ta
 \pn\lec \de^{(d+3)/2}\|u(0)\|_2^2,}
while \eqref{y24-L2bd} implies  
\EQ{
 \|\na\y^3_L\|_2 \le \|\na\y\|_2+\sum_{j=0,2,4}\|\na\y^j_L\|_2 \lec \|\nu\|_2+\|\nu(0)\|_2+\|u(0)\|_2^2.}
Hence the interpolation inequality for $(L^2,L^\I)_{2/d,2}=L^{2^*,2}$ yields
\EQ{
 \|\na\y^3_L\|_{L^{2^*,2}} \lec \de^{1+3/d}(\|\nu\|_{L^\I_tL^2}+\|u(0)\|_2^2),}
and so, using the generalized H\"older and $|h_1|\lec 1/r$, 
\EQ{
 |\be_h(\y^3_L,\y_L)| \lec \|h_R\|_{L^{d,\I}}\|\na\y^3_L\|_{L^{2^*,2}}\|\na\y_L\|_2 \lec R\de^{1+3/d}(\|\nu\|_{L^\I_tL^2}+\|u(0)\|_2^2).}
Gathering the above estimates with $\de:=R^{-2d/(3+2d)}$, we obtain 
\EQ{ \label{est beR}
 |\be_R(\y,\y)| \lec [\|\nu_{<\de}(0)\|_2+R^{-3/(3+2d)}(\|\nu\|_{L^\I_tL^2}+\|u(0)\|_2^2)]\|\nu\|_2.}

\subsubsection{Cross terms}
Extracting a leading term, $CC$ is rewritten as 
\EQ{
 CC=(1-d)\LR{\nu u|u f_{0,R}^3} - \LR{\nu u|u f_{5,R}}+CC_3'.}
For the specific $\psi=\LR{r}^{-1}$, we have 
\EQ{
 f_5=(d-1)(\psi-\psi^{9/2})-r^2\psi^3,}
which satisfies $0<f_5\lec r^2/\LR{r}^3$ for all $d\ge 2$. Hence
\EQ{ \label{nuu tail}
 |\LR{\nu u|u f_{5,R}}| \lec \|\nu\|_2\||u|^2f_{5,R}\|_2 \lec R^{-1}\|\nu\|_2\|x|u|^2\|_2.}
The other error $CC_3'$ is bounded in the same way as in \cite[(3.27)-(3.29)]{GNW}
\EQ{
 |CC_3'| \pt\lec \|\nu\|_2 (\|\F\na\psi_R\|_1+\|\F\na\z_R\|_1)(\|x|u|^2\|_2+\|D^{-1}|u|^2\|_2)
 \pr\lec R^{-1}\|\nu\|_2\|x|u|^2\|_2,}
using \eqref{FL1 bd} and the (dual) Hardy inequality for $d\ge 3$. 
The last norm is bounded by using the radial Sobolev inequality for $3\le d\le 6$: 
\EQ{ \label{xuu bd}
 \|x|u|^2\|_2 \lec \|u\|_2^{(6-d)/2} \|\na u\|_2^{(d-2)/2}.}

\subsection{Virial bound below the ground state}
Adding the above estimates, we conclude the following estimate for any radial smooth solution $(u,N)$ in $\R^4$ on any interval $(0,T)$ with $\|(u,N)\|_{L^\I_t(H^1\times L^2)}\le M\ge 1$: 
\EQ{
 \frac{d}{dt}\V_R(u,\nu) \pt= \dot \V_\I(u f_{0,R}, \nu f_{0,R}) - err
 \prq+ \int_{\R^4}(4|u_\te|^2+|\y_\te|^2)|f_{1,R}|+|u/R|^2f_{2,R}+|\y/R|^2f_{4,R} dx,}
where, with $\de:=R^{-8/11}$,  
\EQ{ \label{V err bd}
 \pt err:= \LR{|u|^4|f_{3,R}}+\be_R(\y,\y) + \LR{\nu u|uf_{5,R}} - CC_3',
 \pr |err| \lec \LR{|u(0)|^4|\La_R}^{1/4}M^2+R^{-1}T^{1/2}M^4
 \prQ +(\|\nu_{<\de}(0)\|_2+R^{-3/11})M^3+R^{-1}M^3.}
The following upper bound is easily obtained: 
\EQ{ \label{V upper bd}
 |\V_R(u,\nu)| \lec R(\|u\|_2\|\na u\|_2 + \|N\|_2^2) \lec RM^4.}

In order to bound the leading term $\dot \V_\I$ from below, let us assume now that $(u,N)$ is a radial global solution satisfying 
\EQ{
 E_Z(u,N) \le E_S(W)-\e, \pq K(u)\ge 0}
for some $\e>0$. Then for any measurable function $a:\R^4\to[0,1]$, the Sobolev inequality $\|\fy\|_4\|W\|_4 \le \|\na\fy\|_2$ and Lemma \ref{lem:estK} yield 
\EQ{
 K(\chi u) \ge (\|W\|_4^2-\|\chi u\|_4^2)\|\chi u\|_4^2 \ge (\|W\|_4^2-\|u\|_4^2)\|\chi u\|_4^2
 \ge \|\nu\|_2\|\chi u\|_4^2,}
while 
\EQ{
 \|W\|_4^2-\|u\|_4^2 = \frac{4E_S(W)-4E_S(u)+2K(u)}{\|W\|_4^2+\|u\|_4^2} \ge 2\e \|W\|_4^{-2} = 2C_S^2\e
}
implies 
\EQ{
 K(\chi u) \ge  2C_S^2\e \|\chi u\|_4^2.}
Hence, using $0\le f_0\le 1$ as well, we obtain 
\EQ{ \label{V main bd}
 \pt\dot \V_\I(u f_{0,R},\nu f_{0,R}) = 4K(uf_{0,R})+\|\nu f_{0,R}\|_2^2-3\LR{\nu u|u f_{0,R}^3}
 \pr\ge K(uf_{0,R}) + \|\nu f_{0,R}\|_2^2 + 3\|\nu\|_2\|u f_{0,R}\|_4^2 - 3\|\nu f_{0,R}\|_2\|u f_{0,R}\|_4^2
 \pr\ge 2C_S^2\e\|u f_{0,R}\|_4^2 + \|\nu f_{0,R}\|_2^2.}

\subsection{Scattering below the ground state}
Let $(u,N)$ be a global solution in $H^1\rad(\R^4)\times L^2\rad(\R^4)$ satisfying 
\EQ{
 E_Z(u,N) \le E_S(W)-\de, \pq K(u)\ge 0,}
and let $\nu:=N-|u|^2$. Then $u$ is uniformly bounded in $H^1(\R^4)$, and $N,\nu$ are uniformly bounded in $L^2(\R^4)$. Hence for any $R\ge 1$ and any $T_0<T_1$, we have 
\EQ{
 \pt \frac{R}{T_1-T_0} \gec \frac{[\V_R(u,\nu)]_{T_0}^{T_1}}{T_1-T_0} \ge \frac{1}{T_1-T_0}\int_{T_0}^{T_1}[\dot \V_\I(u f_{0,R},\nu f_{0,R})-err]dt
 \pr\ge \inf_{t\in(T_0,T_1)}2C_S^2\e\|\psi_R u(t)\|_4^2-C\LR{|u(T_0)|^4|\La_R}^{1/2}-CR^{-1}|T_1-T_0|^{1/2}+o(1),
}
where $o(1)\to 0$ as $R\to\I$ uniformly for all $t\ge 0$. Choosing $T_1=T_0+R^{4/3}$, and using the dominated convergence theorem, we obtain  
\EQ{
 \inf_{t>T_0} \e\|u(t)f_{0,R}\|_4^2 \lec \LR{|u(T_0)|^4|\La_R}^{1/2}+R^{-1/3}+o(1) \to 0 \pq(R\to\I).}
Since $f_{0,R}$ is increasing in $R$, we deduce that 
\EQ{
 0<\forall R,\forall T<\I, \pq \inf_{t>T}\|u(t)f_{0,R}\|_4=0.}
Since $\|u(t)\|_{L^2(|x|<R)} \lec R\|u(t)f_{0,R}\|_4$, Lemma \ref{MS cri} implies the scattering of $(u,N)$ in $H^1\times L^2$. Thus we have proven Theorem \ref{main} in the scattering case (1).

\section{Blow-up below the ground state} \label{sect:bup}
By a similar argument, we prove the blow-up part (2) of Theorem \ref{main}. 
\begin{proof}[Proof of Theorem \ref{main} (2)]
Let $\e:=E_S(W)-E_Z(u(0),N(0))>0$ and let $(u,N)$ be the unique solution on the maximal interval of existence $I$. Let $\nu:=N-|u|^2$. 
Lemma \ref{K cond Zak} implies that $K(u)<0$ and $\|N\|_2>\|W\|_4^2$, and Lemma \ref{lem:estK} with $a^2=\|\nu\|_2^2+4\e$ implies, uniformly on $I$,
\EQ{
 \dot\V_\I(u,\nu) \le 4K(u)+\|\nu\|_2^2-3\|\nu\|_2\|u\|_4^2 \le -4\e,}

Now for contradiction suppose that $I\supset[0,\I)$ and $\sup_{t\ge 0}\|(u,N)\|_{H^1\times L^2}\le M$. 
By the identities in the previous section, using $\psi\le 1$, $\psi_r\le 0$, we have  
\EQ{
 \dot\V_R(u,\nu) \pt\le \dot\V_\I(u,\nu) + \LR{|u|^4|A_d(1-\psi_R)} - \LR{|u|^2|A_{d+4}\De\psi_R} 
 \prq-\ti\be_R(\y,\y)
 -\frac14\LR{|\y|^2|A_{d+2}\De\psi_R}+\LR{\nu |u|^2|A_{d-2}(1-\psi_R)}+CC_3',}
where $\ti\be_R(\y,\y):=\LR{\psi_RD\y|D\y}/2-\LR{\psi_R\na\y|\na\y}/2$ satisfies the same estimate as $\be_R(\y,\y)$ as in \eqref{est beR}. 
Since $|1-\psi|+|r\p_r(1-\psi_R)|\lec r^2/\LR{r}^3$, the $|u|^4$ and $\nu|u|^2$ error terms also satisfy the same estimate as before, as in \eqref{est L4tail} and \eqref{nuu tail}. Since $|A_{d+4}\De\psi|\lec\psi^3$, we have 
\EQ{
 |\LR{|u|^2|A_{d+4}\De\psi_R}| \lec \|x|u|^2\|_2\|r^{-1}R^{-2}\psi_R^3\|_2 \lec R^{-1}\|x|u|^2\|_2,}
where the last norm is bounded by \eqref{xuu bd}. Finally, 
\EQ{ \label{eta2 tail}
 |\LR{|\y|^2|A_{d+2}\De\psi_R}| \lec \LR{|\y|^2|R^{-2}\psi_R^3}}
is estimated by using the same decomposition of $\y$ as before, namely \eqref{decop eta-1}--\eqref{decop eta-2}. Thus we obtain 
\EQ{
 \eqref{eta2 tail} \pt\lec (\|\y_H\|_2+\|\y_L^2+\y_L^4\|_2)\|\y\|_4\|R^{-2}\psi_R^3\|_4 
 \prq+ (\|\y_L^0\|_4+\|\y_L^3\|_4)\|\y\|_4\|R^{-2}\psi_R^3\|_2
 \pr\lec  (\de^{-1}M+M^2)MR^{-1}+(\|\nu_{<\de}(0)\|_2+\|\y_L^3\|_4)M,}
where the last norm is bounded by $L^4\supset\dot B^0_{4,2}\supset[\dot H^1,\dot B^{-1}_{\I,2}]_{1/2}$ and 
\EQ{
 \|\y_L^3\|_{\dot B^{-1}_{\I,2}} \lec \de^{3/2}\|\y_L^3\|_{\dot B^{-5/2}_{1,\I}}
 \pt\lec \de^{3/2}\int_0^{(t-1)_+}|t-\ta|^{-3/2}\||u(\ta)|^2\|_{\dot B^0_{1,\I}}d\ta
 \pr\lec \de^{3/2}M^2.} 
Hence, as $R\to\I$ with $\de=R^{-8/7}$ as before, we have $\eqref{eta2 tail}\to 0$ uniformly for $t\ge 0$. 
Therefore if $R>1$ is large enough, then uniformly for all $0<t<T<\I$,
\EQ{
 \dot\V_R(u,v) \le -2\e + CR^{-1}T^{1/2}M^4.}
Integration over $0<t<T$ yields 
\EQ{
 -RM^4 \lec [\V_R(u,v)]_0^T \le -2\e T + CM^4 R^{-1}T^{3/2}.}
Let $T=R^{4/3}$. Then 
\EQ{
 -M^4R \le -2\e R^{4/3} + CM^4 R,}
which becomes a contradiction as $R\to\I$. 
\end{proof}

\appendix 

\section{Failure of uniform Strichartz estimates beyond the potential mass threshold} \label{ss:fail}
The ground state implies the following negative results for uniform Strichartz estimate with potential. 
For static potentials, the uniformness breaks down almost completely. 
\begin{prop}
Let $p\in[1,\I]$, $q\in(4,\I]$ and $T\in(0,\I]$. Then there is no $C\in(0,\I)$ such that 
\EQ{ \label{Stz abvW}
 \|u\|_{L^p_tL^q(0,T)} \le C\|u(0)\|_{H^1}} 
holds for all $u\in C([0,T];H^1\rad(\R^4))$ satisfying $i\dot u-\De u =V u$ on $0<t<T$ 
for some $V(x):\R^4\to\R$ with $\|V\|_2\le\|W^2\|_2$. 
\end{prop}
Note that the endpoint Strichartz for $L^2$ free solutions is $L^2_tL^4_x$, while every sharp Strichartz for $H^1$ free solutions except the energy norm dominates $L^p_tL^q_x$ for some $q>4$. 
Thus we can not hope for local Strichartz (any better than $L^4_x$, which is provided by Sobolev for $H^1$ solutions) that is uniform with respect to the $L^2$ norm of potential, once it reaches the threshold $\|W^2\|_2$. 
\begin{proof}
Let $\la>0$, $W_\la:=\la W(\la x)$, $V=W_\la^2$ and $u=\chi W_\la+\ga$ with a radial cut-off function $\chi\in C_0^\I(\R^4)$ with the initial data $\ga(0)=0$. 
Then 
\EQ{
 i\dot \ga - \De \ga - V \ga = V\chi W_\la + \De \chi W_\la = 2\na\chi\cdot\na W_\la + W_\la \De\chi =:F}
Using the explicit decay of $W$, we obtain 
\EQ{
 \|F\|_{H^1} \lec \|\la\na^2 W\|_{L^2(|x|\sim \la)} + \|\la^{-1}W\|_{L^2(|x|\sim \la)}
 \lec \la^{-1}. }
Hence if \eqref{Stz abvW} holds for some finite $C$ then by Duhamel 
\EQ{
 \|\ga\|_{L^p_tL^q(0,T)} \le C\|F\|_{L^1_t H^1(0,T)} \lec C\la^{-1}.}
On the other hand,
\EQ{
 \|\chi W_\la\|_q \ge \la^{1-4/q}\|W\|_{L^q(|x|\lec\la)} \gec \la^{1-4/q}}
for $\la>1$ and $q>4$. Hence 
\EQ{
 \|u\|_{L^p_tL^q(0,T)} \gec T^{1/p}\la^{1-4/q} \to \I}
as $\la\to\I$, contradicting \eqref{Stz abvW}. 
\end{proof}

For wave potentials, the uniform estimate fails in the scaling invariant setting, namely for the admissible exponents globally in time. 
\begin{prop}
Let $p\in[2,\I]$, $q\in(4,\I]$ satisfy $1/p+2/q=1$. Then for any $\al\in\R$, there is no $C\in(0,\I)$ such that 
\EQ{ \label{Stz abvW2} 
 \|u\|_{L^p_tL^q(0,\I)} \le C\|u(0)\|_2}
holds for all $u\in C([0,\I);L^2\rad(R^4))$ and $V(t,|x|):[0,\I)\times\R^4\to\C$ satisfying $i\dot u-\De u = (\re V)u$, $\|V(0)\|_2\le\|W^2\|_2$ and $i\dot V=\al DV$. 
\end{prop}
\begin{proof}
If $\al=0$ then this is weaker than the above, as the condition on $C$ is stronger. 
The case of $\al\not=0$ is reduced to the case of $\al=0$ by scaling as follows. 
Suppose that \eqref{Stz abvW2} holds for some $\al>0$ and $C<\I$. 
For any $u,V$ satisfying the conditions, define $u_\la,V_\la$ for $\la>0$ by the rescaling in \eqref{parab scal}. Then we have
\EQ{
 \pt i\dot u_\la - \De u_\la = (\re V_\la)u_\la, \pq i\dot V_\la=\la\al DV_\la,
 \pr \|u_\la\|_{L^p_tL^q(0,\I)}=\|u\|_{L^p_tL^q(0,\I)}, \pq \|u_\la(0)\|_2=\|u(0)\|_2, \pq \|V_\la(0)\|_2=\|V(0)\|_2.}
Hence \eqref{Stz abvW2} holds uniformly for all $\al>0$. 
Now for any radial $\fy,v\in L^2$ with $\|v\|_2\le\|W^2\|_2$, and $\al>0$, let $u_\al,V_\al$ be the solutions of $i\dot u_\al-\De u_\al=(\re V_\al)u_\al$, $i\dot V_\al=\al DV_\al$, $u_\al(0)=\fy$, $V_\al(0)=v$. 
Then we have $\|u_\al\|_{L^p_tL^q(0,\I)}\le C\|\fy\|_2$. Let $\al\to+0$. Then $u_\al\to u_0$ and $V_\al\to v$ in $C(\R;L^2)$, where $u_0=e^{it(-\De+v)}\fy$. Hence by the lower semi-continuity of norm, we obtain $\|u_0\|_{L^p_tL^q(0,\I)}\le C\|\fy\|_2$, namely the uniform estimate in the case of $\al=0$, which was already precluded. 
\end{proof}

\section{Table of Notation} 
{\small \begin{longtable}{l|l|l}
 \hline 
 symbols & description & defined in \\
 \hline
 $u,N$ & Unknown variables of the equations & \eqref{Zak0}, \eqref{Zak}\\ 
 ``solution of \eqref{Zak}" & Solution with a Strichartz norm & Def.~\ref{def:sol}\\
 $D,\LR{D}$ & Fourier multipliers & above \eqref{Zak}, below \eqref{RFD}\\
 $M(u),E_Z(u,N)$ & Energy-type functionals for \eqref{Zak} & \eqref{def M EZ}\\
 $E_S(u),K(u)$ & Energy-type functionals for NLS & \eqref{def ES}, \eqref{def K}\\
 $\|\cdot\|_p$ & $L^p(\R^d)$ norm & \eqref{def norm-p}\\
 $X\rad$ & Radial subspace & \\
 $C_S$ &The Sobolev best constant on $\R^4$ & \eqref{Sob}, \eqref{max Sob}\\
 $2^*,2_*,p(s)$ & Sobolev exponent and its dual & below \eqref{end Str}, \eqref{def p(s)}\\
 $W$ & Ground state of NLS & \eqref{def W}\\
 $H^s,\dot H^s,\dot B^s_{p,q}$ & Sobolev and Besov spaces & above \eqref{def p(s)}\\
 $\X^\de,\ck\X^\de,\ti\X^\de$ & Endpoint Strichartz spaces & \eqref{def X}\\
 $\X^\de_*,\ck\X^\de_*,\ti\X^\de_*$ & Dual spaces of the above & \eqref{def X}\\
 $\cL^p_t$ & Time-frequency mixed norms & \eqref{def cL}\\
 $\fy_j,\fy_Q,\fy_{\cP}$ & Littlewood-Paley decomposition & \eqref{def LP}, \eqref{def LPproj}\\
 $\F\fy=\hat\fy$ & Fourier transform on $\R^d$ & below \eqref{def LP}\\
 $X(I)$ & Restriction to the interval & \eqref{def X(I)}\\
 $C\rd,C\fn,$ & Constants in Strichartz estimates & Lem.~\ref{lem:freeStz}, Thm.~\ref{thm:Stz},\\
 \pq $C^\de_0,C\wt,C\ds$ & & \eqref{def C0de}, Lems.~\ref{lem:St-wt}, \ref{lem:Bsmall}\\
 $\e\fn,\e\wt,\e\pert,\e\ds$ & Smallness for Strichartz estimates & Thm. \ref{thm:Stz}, Lems.~\ref{lem:St-wt}--\ref{lem:Bsmall}\\
 $V\sF,V\sN,V\sD$ &Decomposition of wave potential & Thm.~\ref{thm:Stz}, \eqref{decop Vn}\\
 $\de_\star$ & Regularity gain in radial Strichartz & Lem.~\ref{lem:freeStz}\\
 $\squm$ & Square sum & \eqref{def squm}\\
 $\HL_\io,\LH_\io,\HH_\io$ & Bilinear frequency restrictions & \eqref{def HL}--\eqref{def HL-prod}, \eqref{def HH-prod}\\
 $\Om_\io^\pm,\Om_\io,\ti\Om_\io,\vec\Om_{\io_1,\io_2},$ & Normal forms & \eqref{def Ompm}, \eqref{def Om}, \eqref{def tiOm}, \eqref{def normal}\\
 \pq $\Psi_{\io_1,\io_2},\vec\Om_\io,\Psi_\io$ & & \\
 $F_\io(u,N),G_\io(u,N)$ &Nonlinear terms in the equations & \eqref{def FG}\\
 $\Uf(t)$ &The free propagatpsr & \eqref{def Uf}\\
 $\D_\io^T,\cU_\io^T,\cN_\io^T$ &The Duhamel integrals &\eqref{def Duh}\\
 $S(\si),A_s$ & Dilation and generator & below \eqref{initial PRD}, \eqref{def psiR As}\\
 $\psi^j,t^j_n,\si^j_n,\si^A_n$ & Profiles, their times and frequencies & \eqref{initial PRD}--\eqref{scal lim0}, above \eqref{scal sep}\\
 $\Ga^{0,J}_n,\Ga^{1,J}_n,\Ga^J_n$ & Remainders in the decomposition & \eqref{initial PRD}, \eqref{second PRD}, \eqref{def GaJn}\\
 $\ti V^j_n,\ta^j_n$ & Rescaled profiles and times & \eqref{def tiVjn}\\
 $J_\sim$ & Quotient index set by scaling & \eqref{def Jsim}\\
 $V^A_n,\ti V^A_n$ & Wave trains & \eqref{def VnA}\\
 $\ti S$ & Frequency separation in the set $S$ & \eqref{sep S}\\
 $w$ & Frequency weight with flatness $\be$ & \eqref{def w1}, \eqref{def w2}\\
 $\U{w}^s$ & Fourier multiplier to the power $s$ & \eqref{def ws}\\
 $v_*^*, \ga_*^*$ & Real part of wave potential & \eqref{def vga}\\
 $J^0,C_1,\e_0,c_1$ & Constants in proof of Theorem \ref{thm:Stz} & \eqref{choice J0}, \eqref{def C1}, \eqref{def e0}, \eqref{def c1}\\
 $[A],]A[,\vec A,w_A$ & Frequencies and the weight for $\si^A_n$ & \eqref{freq proj},\eqref{def wA}\\
 $S_+,S_-,\tle S,\tge S$ & Frequencies before or after $S$ & \eqref{freq befaf}\\
 $f^A,u^{\LR{A}},f^{\LR{A}},f^{\vec A}$ & Groups of nonlinear terms & \eqref{def fA}, \eqref{def ufLR}, \eqref{eq vecA}\\
 $Z^s_0,Z^s_1,\cZ^s,Z_2^{s,\de}$ & Function spaces for $(u,N)$ & \eqref{def Zs}, \eqref{def cZs}, \eqref{def Z2}\\
 $X^s_1,X_2^{s,\de},Y_1,Y_2^\de$ & Function spaces for $u$ and for $N$ & \eqref{def Zs}, \eqref{def Z2}\\
 $\psi_R,f_{j,R},\La_R,h_R$ & Weight functions in $x$ with scale $R$ & \eqref{def psiR As}, \eqref{def fj}, \eqref{def La}, \eqref{def be}\\
 $\V_R$ & Localized virial& \eqref{def VR}\\
 $NS,QN,CC,CC_3'$ & Derivative terms of the virial & \eqref{def NS}\\
 $\nu,\y,\y_L,\y_H,\y^j,$ & Variables associated with $N$ & \eqref{def nu}, \eqref{decop eta-1}, \eqref{decop eta-2}\\
 $\be_R$ & A bilinear commutator & \eqref{def be}\\
 
 \hline
\end{longtable}}

\end{document}